\documentclass[11pt,twoside]{article}
\usepackage{amsmath}
\usepackage{amssymb}
\usepackage{amsthm}
\usepackage{mathrsfs}
\usepackage{simplewick}
\usepackage{times}
\usepackage{color}
\usepackage{hep}
\usepackage{leftidx}
\usepackage{graphicx}
\usepackage{float}
\usepackage{enumerate}
\usepackage{titletoc}
\usepackage{epstopdf}
\usepackage{ulem}
\usepackage{appendix}
\usepackage{enumitem}
\usepackage{enumerate} \usepackage{hyperref}
\allowdisplaybreaks
\def\al{\alpha}
\def\be{\beta}
\def\de{\delta}
\def\dl{\delta}

\def\na{\nabla}

\def\Si{\Sigma}
\def\th{\theta}

\def\ve{\varepsilon}
\def\ve0{\varepsilon_{0}}
\def\vp{\varphi}

\def \no{\nonumber}

\def\d{\operatorname{d}}

\def\ds{\slashed{\operatorname{d}}}

\def\D{\mathscr{D}}

\def\f{\frac}

\def\les{\lesssim}

\def\Lc{\check{L}}

\def\Lies{\slashed{\mathcal{L}}}

\def\Lr{\mathring{L}}

\def\p{\partial}
\def\pis{\slashed{\pi}}

\def\tr{\operatorname{tr}}

\def\Tc{\check{T}}

\def\dps{\displaystyle}

\newcommand{\checkX}{\check{\mathscr{X}}} 

\begin{document}
\footskip=0pt
\footnotesep=2pt
\let\oldsection\section
\renewcommand\section{\setcounter{equation}{0}\oldsection}
\renewcommand\thesection{\arabic{section}}
\renewcommand\theequation{\thesection.\arabic{equation}}
\newtheorem{claim}{\noindent Claim}[section]
\newtheorem{theorem}{\noindent Theorem}[section]
\newtheorem{lemma}{\noindent Lemma}[section]
\newtheorem{proposition}{\noindent Proposition}[section]
\newtheorem{definition}{\noindent Definition}[section]
\newtheorem{remark}{\noindent Remark}[section]
\newtheorem{corollary}{\noindent Corollary}[section]
\newtheorem{example}{\noindent Example}[section]
\newcommand{\address}[1]{%
  \begingroup
  \centering
  \small
  #1\\
  \endgroup
  \vspace{0.2cm}
}
\title{Shock formation for 3D steady supersonic flows with general short pulse data}

\author{
Ding Bingbing$^{1}$ \quad Xin Zhouping$^{2}$ \quad Yin Huicheng$^{1,}$
\footnotemark[1]
\\[0.5cm]
\small
$^{1}$ School of Mathematical Sciences and IMS, Nanjing Normal University, Nanjing 210023, China\\
\small
$^{2}$ Institute of Mathematical Sciences, The Chinese University of Hong Kong, Shatin, Hong Kong, China
}

\footnotetext[1]{Ding Bingbing (13851929236@163.com, bbding@njnu.edu.cn) and Yin Huicheng (huicheng@nju.edu.cn, 05407@njnu.edu.cn) are
supported by the NSFC (No.12571237, No.12331007). Xin Zhouping (zpxin@ims.cuhk.edu.hk)
is partially supported by the Zheng Ge Ru Foundation, Hong Kong RGC Earmarked Research Grants
CUHK-14301421, CUHK-14300819, CUHK-14302819, CUHK-14301023, Basic and Applied Basic Research Foundations of Guangdong Province 2020131515310002, and the Key Projects of National Natural Science Foundation of China (No.12131010, No.11931013).\\
}

\date{}
\maketitle

\begin{abstract}
This paper concerns the shock formation problem for
the 3D steady supersonic potential equation of polytropic gases.
The potential equation is described by a second order quasilinear wave equation
$\dps\sum_{i=1}^{3}[\left(\partial_i\Phi\right)^2-c^2(\rho)]\partial_i^2\Phi
+2\dps\sum_{1\leq i<j\leq 3}\partial_i\Phi\partial_j\Phi\partial_{ij}^2\Phi=0$,
where $x=(x^1,x^2,x^3)$, $(\p_1,\p_2,\p_3)=(\p_{x^1},\p_{x^2},\p_{x^3})$,
$c(\rho)=\sqrt{p'(\rho)}$ is the sonic speed with
$p(\rho)=A\rho^{\gamma}$ ($\gamma>1$), and $\p_3\Phi>c(\rho)$.
For the short pulse boundary data $\Phi|_{x^3=0}
=\delta^\nu\Phi_0\left( \frac{r-1}{\delta}, \omega \right)$ and
$\partial_3\Phi|_{x^3=0}=q_0+\delta^{\nu-1}\Phi_1\left( \frac{r-1}{\delta}, \omega \right)$
with $r=\sqrt{(x^1)^2 + (x^2)^2}$, $\omega=(\f{x^1}{r}, \f{x^2}{r})\in\Bbb S$, $1<\nu<2$
and small $\delta>0$, it is shown that a shock will be formed in a finite $x^3$-distance as long as
the boundary data are supersonic and satisfy $(\Phi_0, \Phi_1)\not\equiv 0$. This coincides with physical phenomenon that strong compression of supersonic polytropic gases yields shocks. One of our main ingredients is to find a good unknown so that the previously imposed
compatibility conditions on the short pulse initial data are removed as well as the required weighted energy estimates
in the existing literatures are derived. It is expected that the method here will be applied to study the
shock formation problem with general short pulse initial data for the 3D steady supersonic Euler equations of polytropic gases.
\end{abstract}

\medskip
{\bf Keywords:} Supersonic flow, potential equation, short pulse data, shock formation,

\qquad \qquad \quad null frame, Lorentz geometry

\medskip
\textbf{Mathematical Subject Classification}: 35L05, 35L72

\medskip
\tableofcontents

\section{Introduction}\label{Section 1}
In this paper, we are concerned with the 3D isentropic, irrotational and steady compressible Euler equations
of polytropic gases
\begin{equation}\label{EulerSystem}
\left\{
\begin{aligned}
&\displaystyle\sum_{j = 1}^{3} \partial_{j} \left( \rho v_{j} \right) = 0, \\
&\displaystyle\sum_{j = 1}^{3} \partial_{j} \left( \rho v_{i} v_{j} \right) + \partial_{i} p = 0 ,\quad i = 1, 2, 3,
\end{aligned}
\right.
\end{equation}
where $x = (x^1, x^2, x^3)$, $\partial_i= \partial_{x^i}$ for $1\le i\le 3$, $\rho>0$ is the density, and $p=A\rho^{\gamma}$ is the pressure
with constants $A>0$ and $\gamma>1$. The sonic speed is given by $c(\rho)=\sqrt{p'(\rho)}$, and $v= \left( v_1, v_2, v_3 \right)$ denotes the velocity field of gases. It is assumed that the flow is irrotational (i.e.,
$\operatorname{rot}v
=(\p_2v_3-\p_3v_2,\p_3v_1-\p_1v_3,
\p_1v_2-\p_2v_1)\equiv0$) and $v_3>c(\rho)$ holds. At this time,  $x^3$ is the time-like direction for \eqref{EulerSystem}.

Under the irrotational assumption for \eqref{EulerSystem}, there exists a potential function $\Phi(x)$
such that $v_i = \partial_i \Phi$ ($1\le i\le 3$). This, together with the momentum equations in \eqref{EulerSystem},
yields the following Bernoulli's law
\begin{equation}\label{YHCC-1}
\frac{1}{2} |\p\Phi|^2 + h(\rho) \equiv C_0,
\end{equation}
where $\p=(\p_1,\p_2,\p_3)$, $|\p\Phi|^2=(\p_1\Phi)^2+(\p_2\Phi)^2+(\p_3\Phi)^2$, $C_0=\f12q_0^2$ is the
Bernoulli's constant of the uniform supersonic incoming flow
with constant density $\rho_0>0$ and constant velocity $(0,0,q_0)$ for $q_0>c(\rho_0)$,
and $h(\rho)$ represents the enthalpy satisfying $h'(\rho) = \frac{c^2(\rho)}{\rho}>0$ for $\rho>0$ and $h(\rho_0)=0$.

Substituting \eqref{YHCC-1} into the first equation in \eqref{EulerSystem} and performing direct computations  yield
\begin{equation}\label{PotentialEquation}
\sum_{i = 1}^{3}[\left(\partial_i\Phi\right)^2-c^2(\rho)]\partial_i^2\Phi+2\sum_{1\leq i<j\leq 3}\partial_i\Phi\partial_j\Phi\partial_{ij}^2\Phi = 0,
\end{equation}
where $\rho=h^{-1}(C_0-\tfrac 12|\p\Phi|^2)$ and $h^{-1}$ is the inverse function of $h(\rho)$.
It is easy to verify that \eqref{PotentialEquation} is strictly hyperbolic with respect to the $x^3$-direction due to $\p_3\Phi>c(\rho)$.

Consider the following  short pulse boundary data for \eqref{PotentialEquation}
\begin{equation}\label{InitialData}
\left\{
\begin{aligned}
&\Phi|_{x^3 = 0} = \delta^\nu \Phi_0\big(\frac{r-1}{\delta}, \omega\big), \\
&\partial_3\Phi|_{x^3 =0} = q_0 + \delta^{\nu - 1} \Phi_1\big(\frac{r-1}{\delta}, \omega\big),\\
\end{aligned}
\right.
\end{equation}
where $r=\sqrt{(x^1)^2 + (x^2)^2}$, $\omega=(\omega_1,\omega_2)=(\f{x^1}{r}, \f{x^2}{r})\in\Bbb S$, $1<\nu<2$,
$\delta>0$ is small,  \( \Phi_i(s, \omega) \in C_0^\infty\big((-1, 0) \times \mathbb{S}\big) \)
for \( i = 0, 1 \), and $(\Phi_0, \Phi_1)\not\equiv0$. Note that such a class of short pulse data
in \eqref{InitialData} are first introduced
to study the formation of trapped surfaces for the Einstein vacuum spacetime in the monumental work \cite{Ch2},
and subsequently generalized in \cite{K-R} by enlarging the admissible set of initial data.
Meanwhile, crucial extra smallness restrictions are imposed along the incoming directional  derivative $\p_t-\p_r$,
but the largeness is still kept for the outgoing directional derivative
$\p_t+\p_r$ in \cite{Ch2} and \cite{K-R}. Analogous smallness conditions are also required
in order to obtain the shock formation for the 3D quasilinear wave equations $-(1+(\p_t\phi)^p)\p_t^2\phi+\Delta\phi=0$
($p\ge 2, p\in\Bbb N$) (see \cite{M-Y} and \cite{L-Y}), that is, the following  compatibility conditions are imposed
\begin{equation}\label{YHCCC-1}
\begin{split}
&(\p_t+\p_r)^k\phi|_{t=0}=O(\delta^\nu)\ \text{(or}\ (\p_t-\p_r)^k\phi|_{t=0}=O(\delta^\nu)) \ \text{for}\  k=1,2,\\
&(\p_t-\p_r)^l\phi|_{t=0}=O(\delta^{\nu-l})\ \text{(or }(\p_t+\p_r)^l\phi|_{t=0}
=O(\delta^{\nu-l})\text{)}\   \text{for any}\ l\in\Bbb N.
\end{split}	
\end{equation}
It is pointed out that the short pulse solutions of differential
equations admit many important physical senses, which means that the related  quantities vary slowly
but their derivatives may change very rapidly. One can see more explanations in \cite{Alt1, Alt2, Hunter, Majda}.

It follows from \eqref{InitialData}, \eqref{YHCC-1} and $(\Phi_0, \Phi_1)\not\equiv0$ that
\begin{align}\label{YHCCC-2}
\begin{split}
&v_i|_{x^3=0}=O(\delta^{\nu-1})\quad\text{for $i=1,2$},\quad
v_3|_{x^3=0}=q_0+O(\delta^{\nu-1}),\\
&\rho|_{x^3=0}=\rho_0+O(\delta^{\nu-1}),\quad
\|\p v|_{x^3=0}\|_{L^\infty}\ge \f{C}{\delta^{2-\nu}}.
\end{split}
\end{align}
This implies that for small $\delta$ and $q_0>c(\rho_0)$, the flow on the boundary $x^3=0$ is supersonic
and strongly squeezed at some local positions due to $\Phi_i(s, \omega) \in C_0^\infty\big((-1, 0) \times \mathbb{S}\big)$
($i= 0, 1$).

On the other hand, for $(\Phi_0, \Phi_1)\not\equiv0$, it is easy to know that
\begin{align}\label{condition1}
&\textit{there exists a point $(s_0, \omega_0) \in (-1, 0) \times \mathbb{S}$\ such that}\ \partial_s\Phi_1(s_0,\omega_0)-\mathcal A\partial_s^2\Phi_0(s_0,\omega_0)>0
\end{align}
or
\begin{align}\label{condition2}
&\textit{there exists a point $(\tilde s_0, \tilde\omega_0) \in (-1, 0) \times \mathbb{S}$\  such that}\ -\partial_s\Phi_1(\tilde s_0,\tilde \omega_0)-\mathcal A\partial_s^2\Phi_0(\tilde s_0,\tilde \omega_0)>0,
\end{align}
where and hereafter $\mathcal A=\dps\frac{1}{\sqrt{M_0^2-1}}$ with $M_0=\dps\f{q_0}{c(\rho_0)}$ being the Mach number of the
uniform supersonic incoming flow, $\omega_0=(\omega_{0,1},\omega_{0,2})$ and $\tilde\omega_0=(\tilde\omega_{0,1}, \tilde\omega_{0,2})$.
Otherwise, for any $(s, \omega) \in (-1, 0) \times \mathbb{S}$, it holds that
$\partial_s\Phi_1(s,\omega)-\mathcal A\partial_s^2\Phi_0(s,\omega)\le 0$
and $-\partial_s\Phi_1(s,\omega)-\mathcal A\partial_s^2\Phi_0(s,\omega)\le 0$.
This yields $\partial_s^2\Phi_0(s,\omega)\ge 0$ and hence $\Phi_0(s,\omega)\equiv0$ due to $\Phi_0(s, \omega) \in C_0^\infty\big((-1, 0) \times \mathbb{S}\big)$. Based on this, one has $\partial_s\Phi_1(s,\omega)\le 0$
and $-\partial_s\Phi_1(s,\omega)\le 0$, which implies $\partial_s\Phi_1(s,\omega)\equiv0$
and hence $\Phi_1(s,\omega)\equiv0$ by $\Phi_1(s, \omega) \in C_0^\infty\big((-1, 0) \times \mathbb{S}\big)$.
This obviously contradicts $(\Phi_0, \Phi_1)\not\equiv0$.

In the present paper, we will remove the crucial compatibility condition \eqref{YHCCC-1}  for \eqref{InitialData} imposed
in \cite{Ch2,K-R,M-Y,L-Y}, and show the shock formation for \eqref{PotentialEquation}
with \eqref{InitialData} as long as $(\Phi_0, \Phi_1)\not\equiv0$. Our main result is stated as

\begin{theorem}\label{maintheorem}
Assume $(\Phi_0, \Phi_1)\not\equiv 0$ and $1<\nu<2$. For small $\delta>0$,

(A) when \eqref{condition1} holds, then the solution of \eqref{PotentialEquation}--\eqref{InitialData}
will develop a shock before the $x^3$-distance
\begin{equation}\label{distance1}
x^3_* = \delta^{2-\nu}\frac{4c(\rho_0)}{(\gamma+1)\big(\partial_s\Phi_1(s_0,\omega_0)-\mathcal A\partial_s^2\Phi_0(s_0,\omega_0)\big)}\big(1-\frac{1}{M_0^2}\big)^{3/2};
\end{equation}

(B) when \eqref{condition2} holds,
then the solution of  \eqref{PotentialEquation}--\eqref{InitialData} will develop a shock wave before the $x^3$-distance
\begin{equation}\label{distance2}
\tilde x^3_* = \delta^{2-\nu}\frac{4c(\rho_0)}{(\gamma+1)\big(-\partial_s\Phi_1(\tilde s_0,\tilde \omega_0)-\mathcal A\partial_s^2\Phi_0(\tilde s_0,\tilde \omega_0)\big)}\big(1-\frac{1}{M_0^2}\big)^{3/2}.
\end{equation}
\end{theorem}

\begin{remark}\label{aboutcondition}
It follows from the proof procedure of Theorem \ref{maintheorem} that
the shock will be formed near the outgoing or incoming light conic surface
under the condition \eqref{condition1} or \eqref{condition2}, respectively.
\end{remark}

\begin{remark}\label{Yin.0}
By \eqref{YHCCC-2} and $(\Phi_0, \Phi_1)\not\equiv 0$,  we know that the initial supersonic flow is strongly squeezed
at some local positions. Indeed, under the condition \eqref{condition1}, one has either $\p_s\Phi_1(s_0,\omega_0)>0$
or $\p_s\Phi_1(s_0,\omega_0)<0$ or $\p_s\Phi_1(s_0,\omega_0)=0$ but $\p_s^2\Phi_0(s_0,\omega_0)<0$,
which, in the neighborhood of $(x^1_0,x^2_0)=((1+s_0\delta)\omega_{0,1}, (1+s_0\delta)\omega_{0,2})$, corresponds to the outward squeezing of $v_3|_{x^3=0}$ or the inward squeezing of $v_3|_{x^3=0}$
or the inward squeezing of the radial velocity $v_r|_{x^3=0}$ with $v_r=v_1\f{x^1}{r}+v_2\f{x^2}{r}$. An analogous statement
holds for the condition \eqref{condition2}. By Theorem \ref{maintheorem}, the 3D supersonic shock is formed in a finite $x^3$-distance.
This is consistent with the corresponding physical phenomena that the rapid compression of polytropic gases can yield shocks 
(see the books \cite{CF, Bers}).
\end{remark}

\begin{remark}\label{Yin.78}
Under the condition \eqref{condition1}, let $X_*$ denote the precise blowup distance ($0<X_*\le x^3_*$)
for equation \eqref{PotentialEquation} with \eqref{InitialData}.
By the proof procedure of Theorem \ref{maintheorem} (see the bootstrap assumptions \eqref{bootstrap-assumptions}
and \eqref{shock-0}-\eqref{shock} below), we have $\|\p\Phi-(0,0,q_0)\|_{L^{\infty}}\le M\delta^{\nu-1}$ 
and $\|\p^2\Phi\|_{L^{\infty}}=+\infty$ in some domain $D_*$ which is a neighborhood of
$\{((\mathcal A-(1+s_0\delta))\omega_{0,1}, (\mathcal A-(1+s_0\delta))\omega_{0,2}, x^3): 0\le x^3\le X_*\}$. 
Therefore, for small $\delta>0$ and $x\in D_*$,
one has $v_3>c(\rho)$, $\rho>0$,
and $|\p v|+|\p\rho|\to\infty$ as $x^3\to X_*-$.
This shows that in $D_*$, the flow remains supersonic and non-vacuum, while a shock wave is formed. An
analogous conclusion also holds under the condition \eqref{condition2}.
\end{remark}

\begin{remark}\label{Yin.1}
It should be emphasized that in Theorem \ref{maintheorem}, we do not impose such crucial smallness constraints on the initial short pulse data
along the incoming  or outgoing directional derivatives:
\begin{equation}\label{Yin.1-1}
(\partial_3-\partial_r)^{\leq N_0}\Phi|_{x^3=0}=O(\delta^{\nu}) \quad \text{or} 
\quad (\partial_3+\partial_r)^{\leq N_0} \Phi|_{x^3=0}=O(\delta^{\nu}),
\end{equation}
where $N_0\ge 2$ and $N_0\in\Bbb N$. Note that \eqref{Yin.1-1} generally leads to the over-determination of
the functions $(\Phi_0, \Phi_1)$ in \eqref{InitialData} due to $(\partial_3\pm\partial_r)^{\leq N_0}\Phi|_{x^3=0}=O(\delta^{\nu-N_0})$
rather than $O(\delta^{\nu})$. Consequently, $(\Phi_0, \Phi_1)$ should be chosen delicately by looking for the
approximate solutions of \eqref{PotentialEquation}, moreover, it is difficult to find suitable initial functions $(\Phi_0, \Phi_1)$
to fulfill \eqref{Yin.1-1} for $N_0\ge 2$ (see Section 2 of \cite{Ding3}).
Meanwhile, analogous assumptions to \eqref{Yin.1-1} play key roles in proving  the related blowup results in
\cite{Ch2, K-R, M-Y, L-Y}. In fact, by the methods in the paper, all the smallness
restriction conditions of the form \eqref{Yin.1-1}  in \cite{M-Y, L-Y} can be removed. Hence, the scope of the short pulse initial data may
be chosen more broadly.
\end{remark}

\begin{remark}\label{Yin.4}
We now explain the role of condition \eqref{condition1} or  \eqref{condition2} in proving Theorem \ref{maintheorem}.
Consider the general 3D quasilinear wave equation
\begin{equation}\label{YHCC14}
\begin{split}
\dps\sum_{\al,\beta=0}^3g^{\al\beta}(\p\phi)\p_{\al\beta}^2\phi=0,
\end{split}
\end{equation}
where $(x^0, x)=(t, x^1, x^2, x^3)\in [0,\infty)\times\Bbb R^3$,
$\p=(\p_0,\p_1,\p_2,\p_3)=(\p_{x^0}, \p_{x^1}, \p_{x^2}, \p_{x^3})$,
$g^{\al\beta}(\p\phi) =g^{\beta\al}(\p\phi)$ are smooth
functions of their arguments, and $g^{\al\beta}(\p\phi)=m^{\al\beta}+\dps\sum_{\kappa=0}^3 g^{\al\beta,
\kappa}\p_\kappa\phi+O(|\p\phi|^2)$
with $m^{00}=-1$, $m^{ii}=1$ for $1\le i\le 3$, $m^{\al\beta}=0$ for $\al\neq\beta$, and
$g^{\al\beta,\kappa}$ being constants. In \cite{Ch1}, an inverse foliation density $\mu=-\dps\f{1}{\sum_{\al=0}^3g^{0\al}(\p\phi)\p_{\al}u}$ is introduced
for \eqref{YHCC14},
and the optical function $u$
is determined by the outgoing characteristic conic surface equation $\dps\sum_{\al,\beta=0}^3g^{\al\be}(\p\phi)\p_{\al}u\p_{\be}u=0$ with the initial data $u(0,x)=-r=-\sqrt{(x^1)^2+(x^2)^2+(x^3)^2}$ and the condition $\p_{t}u>0$.
It is known that for the blowup time $T^*$,
\begin{equation*}
\lim_{t\to T^*-0}\|\p^2\phi(t,\cdot)\|_{L^{\infty}(\Bbb R^3)}=\infty
\end{equation*}
holds when $\mu\to 0+$ for $t\to T^*-0$ (also see \cite{Ch3,Luk,Sp,M-Y}).
For \eqref{PotentialEquation} and \eqref{condition1}, we can introduce {a} similar inverse foliation density $\mu$
and show that (see \eqref{YHCC-2} below)
\begin{equation}\label{YHCC3-1}
\mu=1+\frac{\gamma+1}{4c(\rho_0)}\Big(1-\frac{1}{M_0^2}\Big)^{-3/2}\delta^{\nu-2}
\big(\mathcal A\partial_s^2\Phi_0(s_0,\omega_0)-\partial_s\Phi_1(s_0, \omega_0)\big)
\big(x^3+ O((x^3)^2)\big)+O(\delta^{\nu-1}).
\end{equation}
Note that $\p_{x^3}\mu<0$ and $\mu|_{x^3=0}=1+O(\delta^{\nu-1})>0$ for small $\delta>0$.
Thus, by \eqref{condition1} and \eqref{YHCC3-1},
$\mu\to 0+$ holds before
\[
x^3_*=\frac{4c(\rho_0)}{(\gamma+1)\big(\p_s\Phi_1(s_0,\omega_0)-\mathcal A\p_s^2\Phi_0(s_0,\omega_0)\big)}
\big(1-\frac{1}{M_0^2}\big)^{3/2}\delta^{2-\nu}
\]
and leads to the shock formation as asserted in (A) of Theorem \ref{maintheorem} (see the detail \eqref{shock} below).
Analogously, under the condition \eqref{condition2}, by adjusting the corresponding quantities and estimates
in light of the incoming characteristic surfaces, (B) of Theorem \ref{maintheorem} can be derived.
\end{remark}

\begin{remark}\label{Yin.5}
Consider the small data solution problem for the 2D quasilinear wave equation
\begin{equation}\label{quasi-0}
\sum_{\al,\beta=0}^2g^{\al\beta}(\p\phi)\p_{\al\beta}^2\phi=0,
\end{equation}
where $(x^0, x)=(t, x^1, x^2)\in [0,\infty)\times\Bbb R^2$,
$\p=(\p_0,\p_1,\p_2)=(\p_{x^0}, \p_{x^1}, \p_{x^2})$, and
$g^{\al\beta}(\p\phi) =g^{\beta\al}(\p\phi)$ are smooth
functions of their arguments.
Without loss of generality, it is assumed that for small $\p\phi$,
\begin{equation}\label{H-01}
\dps g^{\al\beta}(\p\phi)=m^{\al\beta}+\sum_{\kappa=0}^{2} g^{\al\beta,\kappa}\p_\kappa\phi+\sum_{\kappa,\varsigma=0}^2h^{\al\beta,\kappa\varsigma}\p_\kappa\phi\p_\varsigma\phi+O(|\p\phi|^3),
\end{equation}
where $m^{00}=-1$, $m^{ii}=1$ for $1\le i\le 2$, $m^{\al\beta}=0$ for $\al\neq\beta$, $g^{\al\beta,\kappa}$
and $h^{\al\beta,\kappa\varsigma}$ are constants.
If \eqref{quasi-0} is equipped with the small initial data
\begin{equation}\label{Y1-0}
\phi(0,x)=\dl \phi_0(x),\quad \p_t\phi(0,x)=\dl\phi_1(x),
\end{equation}
where $\dl>0$ is small, $(\phi_0, \phi_1)(x)\in C_0^{\infty}(\mathbb R^2)$,
then it is shown in \cite{A} that the problem \eqref{quasi-0}-\eqref{Y1-0} admits a global smooth solution $\phi$ satisfying $|\partial\phi|\le C\delta(1+t)^{-1/2}$ if both the first and the second null conditions hold (i.e.,
$\dps\sum_{\al,\beta,\kappa=0}^2g^{\al\beta,\kappa}\xi_\al\xi_\beta\xi_\kappa\equiv0$ and $\dps\sum_{\al,\beta,\kappa,\varsigma=0}^2 h^{\al\beta,\kappa\varsigma}\xi_\al\xi_\beta\xi_\kappa\xi_\varsigma\equiv0$
for $\xi_0=-1$ and $(\xi_1,\xi_2)\in\mathbb S$, respectively).
Otherwise, if either the first or the second null condition fails,
under the generic non-degenerate assumption on the unique minimum point of  $(\phi_0, \phi_1)$, then the small smooth solution $\phi$ of \eqref{quasi-0}-\eqref{Y1-0} blows up in finite time. More precisely,
at the blowup time, $|\phi|$ and $|\partial\phi|$ remain small but $|\partial^{2}\phi|$ becomes infinite at the unique blowup point,
while it is unknown whether the characteristic surfaces can squeeze and further the shocks are formed (see \cite{Ali1,Ali2,H}).
On the other hand, it is easy to verify that the equation \eqref{PotentialEquation} does not satisfy the first null condition.
Then, according to \cite{Ch3,Sp}, under the {\bf suitable restrictions} on $(\Phi_0(x),\Phi_1(x))$,
the small perturbed smooth solution $\Phi$ forms a shock within a finite $x^{3}$ distance.
However, from Remark \ref{Yin.0}, we have established an interesting  conclusion that the shocks
are definitely formed for \eqref{PotentialEquation}-\eqref{InitialData}
as long as $(\Phi_0(x),\Phi_1(x))\not\equiv0$.
\end{remark}

\begin{remark}\label{Yin.6}
If the inverse foliation density $\mu\ge C>0$ holds, then the equation \eqref{quasi-0}
admits a global smooth solution $\phi$ (see \cite{MPY, Ding2, Ding3, Ding4, Wang}).
In addition, the authors in \cite{An-1, An-2} also apply the Lorentzian geometric tools from \cite{Ch1}
to study the local well-posedness or ill-posedness of low regularity solutions to the elastic wave system.
\end{remark}

\begin{remark}\label{Yin.7}
Consider the following 3D steady supersonic Euler equations of polytropic gases with the
short pulse boundary data
\begin{equation}\label{EulerSystem-1}
\left\{
\begin{aligned}
&\displaystyle\sum_{j = 1}^{3} \partial_{j}(\rho v^{j})=0, \\
&\displaystyle\sum_{j = 1}^{3} \partial_{j}(\rho v^{i} v^{j})+\partial_{i}p=0 ,\quad i = 1, 2, 3,\\
&\rho|_{x^3=0}=\bar\rho+\dl^{\nu}\rho_0\big(\frac{r-1}{\delta}, \omega\big),\quad v|_{x^3=0}=(0,0,\bar q)+\dl^{\nu} v_0\big(\frac{r-1}{\delta}, \omega\big),
\end{aligned}
\right.
\end{equation}
where $x=(x^1, x^2, x^3)$, $v=(v^1,v^2,v^3)$, $p=A\rho^{\gamma}$ with constants $A>0$ and $\gamma>1$,
$v_0=(v^1_0,v^2_0,v^3_0)$, $(\rho_0,v_0)(s,\omega)\in C_0^\infty\big((-1, 0) \times \mathbb{S}\big)$ and
$(\rho_0,v_0)\not\equiv 0$. In addition,  $\nu, \bar\rho, \bar q$ are some positive constants
with $0<\nu<1$ and $\bar q>c(\bar\rho)=\sqrt{p'(\bar\rho)}$,
and the constant $\delta>0$ is sufficiently small. Motivated by \cite{Luk} and \cite{LS2}, we can
introduce the following inverse Lorentzian metric and Lorentzian metric, respectively,
\begin{equation}\label{YHCC-4}
g^{-1} = -\frac{1}{c^2(\rho)}B\otimes B + \sum_{i=1}^3 \partial_i\otimes\partial_i,
\end{equation}
\begin{equation}\label{YHCC-5}
g = -\frac{1}{|v|^2-c^2(\rho)}\sum_{i,j=1}^3v_id x^i\otimes v_jd x^j+\sum_{i=1}^3 dx^i\otimes\ d x^i,
\end{equation}
where $|v|^2=(v^1)^2+(v^2)^2+(v^3)^2$, $B=\sum_{i=1}^3v^i \partial_i$ stands for the material derivative. It follows from
\eqref{EulerSystem-1}-\eqref{YHCC-5}
and direct computations that $\varrho=ln (\frac{\rho}{\bar\rho})$ and $v$ satisfy the quasilinear wave systems as follows
\begin{equation}\label{Mainequations2}
\Box_{g}\varrho=\frac{c'(\rho)(2c^2(\rho)-3|v|^2)-c^3(\rho)}{c(\rho)(|v|^2-c^2(\rho))}\sum_{i,j=1}^3g^{ij}\partial_i\varrho \p_j \varrho+c^{-2}(\rho)[\sum_{i=1}^3\p_i v^i\sum_{j=1}^3\p_j v^j-\sum_{i,j=1}^3\p_i v^j\p_j v^i]
\end{equation}
and
\begin{equation}\label{Mainequations1}
	\Box_{g}{v^{k}}=-\bigg(1+\frac{c'(\rho)|v|^2+c^3(\rho)}{c(\rho)(|v|^2-c^2(\rho))}\bigg)
\sum_{i,j=1}^3g^{ij}\p_i\varrho \p_j v^k,\quad k=1,2,3,
\end{equation}
where $\Box_{g}=\sum_{1\le i,j\le 3}g^{ij}\partial_{ij}^2$.
In our forthcoming paper, based on the introduction of a good unknown and related estimates in this paper,
we will study the shock formation problem of \eqref{Mainequations2} and \eqref{Mainequations1} with the general short pulse initial data
in \eqref{EulerSystem-1} and without the compatibility condition as in \eqref{YHCCC-1}.
\end{remark}

\begin{remark}\label{Yin.8}
Consider the boundary value problem of 2D supersonic steady full compressible Euler equations of polytropic gases
with short pulse data
\begin{equation}\label{FFF-1}
\begin{cases}
&\displaystyle\sum_{j=1}^{2}\partial_{j}(\rho v^{j})=0, \\
&\displaystyle\sum_{j=1}^{2}\partial_{j}(\rho v^{i} v^{j})+\partial_{i} p = 0 ,\quad i = 1, 2,\\
&\displaystyle\sum_{j=1}^{2}\partial_j((\rho e+\f12 \rho |v|^2+p)v^j))=0,\\
&\rho|_{x^2=0}=\bar\rho+\delta^{\nu}\rho_0(\f{x_1-1}{\delta}), v|_{x^2=0}=(0,\bar q)+\delta^{\nu} v_0(\f{x_1-1}{\delta}),
S|_{x^2=0}=\bar S+\delta^{\nu} S_0(\f{x_1-1}{\delta}),
\end{cases}
\end{equation}
where $x=(x^1, x^2)\in\mathbb{R}^2$, $v=(v^1,v^2)$, $p=p(\rho,S)=A\rho^{\gamma}e^{\f{S}{c_v}}$ with constants $A>0$,  $\gamma>1$
and $c_v>0$, $e=e(\rho, S)$ is smooth on its arguments, $\p_Se(\rho, S)>0$, $|v|^2=(v^1)^2+(v^2)^2$, $v_0=(v^1_0,v^2_0)$, $(\rho_0,v_0,S_0)(s)\in C_0^\infty(-1, 0)$, $(\rho_0,v_0,S_0)\not\equiv 0$, $0<\nu<1$,
$\bar\rho>0$, $\bar S>0$, $\bar q>c(\bar\rho,\bar S)$ with $c(\rho,S)=\sqrt{\p_{\rho}p(\rho, S)}$,
and $\delta>0$ is sufficiently small.
Under the transformation $(X^1,X^2)=(\f{x^1-1}{\delta}, \f{x^2}{\delta})$, set $(\tilde \rho, \tilde p, \tilde v, \tilde e, \tilde S)(X)=
(\rho, p, v, e, S)(1+\delta X^1, \delta X^2)$. We still denote $(\p_1,\p_2)=(\p_{X^1}, \p_{X^2})$, then it follows
from \eqref{FFF-1} that
\begin{equation}\label{FFF-01}
\begin{cases}
&\displaystyle\sum_{j=1}^{2}\partial_{j}(\tilde\rho \tilde v^{j})=0, \\
&\displaystyle\sum_{j=1}^{2}\partial_{j}(\tilde \rho \tilde v^{i} \tilde v^{j})+\partial_{i} \tilde p = 0 ,\quad i = 1, 2,\\
&\displaystyle\sum_{j=1}^{2}\partial_j((\tilde \rho\tilde  e+\f12 \tilde\rho |\tilde v|^2+\tilde p)\tilde v^j))=0,\\
&\tilde\rho|_{X^2=0}=\bar\rho+\delta^{\nu}\rho_0(X^1), \tilde v|_{X^2=0}=(0, \bar q)+\delta^{\nu} v_0(X^1),
\tilde S|_{X^2=0}=\bar S+\delta^{\nu} S_0(X^1).
\end{cases}
\end{equation}
By Chapter 4 of \cite{H} and \cite{DingM-Yin}, as long as $(\rho_0,v_0)\not\equiv 0$,
the smooth solution $(\tilde \rho, \tilde v, \tilde S)$ blows up before $X_2^*\le\f{C^*}{\delta^{\nu}}$
and the shock is actually produced from the blowup point. Motivated by this result for 2D case and Theorem \ref{maintheorem} in the paper,
one can guess that as long as $(\rho_0,v_0)\not\equiv 0$ for the 3D problem \eqref{EulerSystem-1}, the shock will be formed
in finite $x^3-$distance.
\end{remark}

We now rewrite the equation \eqref{PotentialEquation}.
Set $\phi=\Phi-q_0x^3$. Then one has from \eqref{PotentialEquation} and \eqref{InitialData}  that
\begin{equation}\label{equationofphi}
\begin{cases}
&-\partial_3^2\phi
-\dps\f{1}{(\partial_3\phi+q_0)^2-c^2(\rho)}\bigg\{\dps\sum_{i=1}^2\left[(\partial_i \phi)^2-c^2(\rho)\right]\partial_i^2\phi+ 2\partial_1\phi\partial_2\phi\partial_{12}^2\phi\\
&\qquad \qquad \qquad \qquad \qquad \qquad \qquad \qquad +2\dps\sum_{1\leq i\leq 2}\partial_i\phi(\partial_3\phi
+q_0)\partial_{i3}^2\phi\bigg\} = 0, \\
&\phi|_{x^3=0}=\delta^{\nu}\Phi_0(\frac{r-1}{\delta},\omega), \quad
\partial_3\phi|_{x^3=0}=\delta^{\nu-1}\Phi_1(\frac{r-1}{\delta},\omega),
\end{cases}
\end{equation}
where $c^2(\rho)=c^2(h^{-1}(-\f12|\p\phi|^2-q_0\p_3\phi))$.

Let $y^0=t=x^3$ and $y^i=\mathcal A^{-1}x^i$
for $i=1,2$. Meanwhile, $\phi=\phi(t,y^1,y^2)$ and $\p=(\p_{y^0}, \p_{y^1},\p_{y^2})=(\p_0,\p_1,\p_2)$ are written
from now on.
Then the coefficients of
the equation in \eqref{equationofphi} can be expressed as
\begin{equation}\label{g}
\begin{split}
g^{00}(\p\phi)&=-1, \qquad \qquad\qquad\qquad
g^{ii}(\p\phi)=\frac{\mathcal A^{-2} [c^2(\rho)-\mathcal A^{-2}\left(\partial_{i}\phi\right)^2]}{\left(\partial_{0}\phi+q_0\right)^2 - c^2(\rho)}\quad\text{for $i=1,2$},\\
g^{12}(\p\phi)&=\frac{-\mathcal A^{-4}\partial_{1}\phi\partial_{2}\phi}{\left(\partial_{0}\phi+q_0\right)^2-c^2(\rho)},\quad
g^{0i}(\p\phi)=\frac{-\mathcal A^{-2}\partial_{i}\phi\left(\partial_{0}\phi+q_0\right)}{\left(\partial_{0}\phi
+q_0\right)^2-c^2(\rho)}\quad \text{for $i=1,2$}.
\end{split}
\end{equation}
For convenience, $g^{\al\beta}(\p\phi)$ is denoted by  $g^{\al\beta}$ for $\alpha, \beta=0,1,2$.
A direct computation shows that the matrix $\big(g^{\alpha\beta}\big)_{0 \leq \alpha,\beta \leq 2}$
has an inverse $\big(g_{\alpha\beta}\big)_{0 \leq \alpha,\beta \leq 2}$ given by
\begin{equation}\label{metric}
\begin{split}
&g_{00}=\mathcal G\mathcal A^{-2}\left(c^2(\rho)-\mathcal A^{-2}\varphi_1^2-\mathcal A^{-2}\varphi_2^2\right), \quad g_{0i}=g_{i0}=\mathcal G\mathcal A^{-2}\varphi_i(\varphi_0+q_0)\ \text{for}\ i=1,2,\\
&g_{11}=\mathcal G\left(c^2(\rho)-\mathcal A^{-2}\varphi_2^2-(\varphi_0+q_0)^2\right),\quad g_{22}=\mathcal G\left(c^2(\rho)-\mathcal A^{-2}\varphi_1^2-(\varphi_0+q_0)^2\right),\\
&g_{12}=g_{21}=\mathcal G\mathcal A^{-2}\varphi_1\varphi_2,
\end{split}
\end{equation}
where $\vp_{j}=\p_{j}\phi$ for $j=0,1,2$, $\mathcal{G}=\dps\frac{(\vp_0+q_0)^2-c^2(\rho)}
{\mathcal{A}^{-2}c^2(\rho)\left(c^2(\rho)-\mathcal{A}^{-2}\vp_1^2-\mathcal{A}^{-2}\vp_2^2-(\vp_0+q_0)^2\right)}$
and $c^2(\rho)=c^2\big(h^{-1}(-\f12\vp_0^2-\f{1}{2\mathcal{A}^{2}}(\vp_1^2+\vp_2^2)-q_0\vp_0)\big)$.

In this case, \eqref{equationofphi} becomes
\begin{equation}\label{mainequation}
\begin{cases}
\displaystyle\sum_{\alpha,\beta=0}^2 g^{\alpha\beta} \partial_{\alpha\beta}^2 \phi = 0,\\[6pt]
\phi|_{t=0}=\delta^{\nu}\Phi_0(\frac{\mathcal{A}r - 1}{\delta}, \omega), \quad
\partial_{0}\phi|_{t=0}=\delta^{\nu - 1}\Phi_1(\frac{\mathcal{A}r - 1}{\delta}, \omega),
\end{cases}
\end{equation}
where $r=\sqrt{(y^1)^2 + (y^2)^2}$, $(y^1,y^2)=(r\omega_1, r\omega_2)$ with $\omega=(\omega_1, \omega_2)=\dps(\f{y^1}{r}, \f{y^2}{r})$.
These notations are still used without confusion of the original ones $r=\sqrt{(x^1)^2+(x^2)^2}$ and $\omega=(\omega_1, \omega_2)=\dps(\f{x^1}{r}, \f{x^2}{r})$. Thus, it suffices to study the equivalent problem \eqref{mainequation}
instead of \eqref{PotentialEquation} with \eqref{InitialData}.

We next only give the comments on the main ideas in the proof of Theorem \ref{maintheorem} (A) since the
proof on  Theorem \ref{maintheorem} (B) is completely analogous.
As indicated in Remark \ref{Yin.1}, no compatibility condition  analogous to \eqref{Yin.1-1} for
$(\Phi_0,\Phi_1)$ is imposed in the present work.
The key observations that enable us to remove the compatibility condition are inspired by the  delicate analysis on the global
existence of short pulse solutions to general quasilinear wave equations with suitable null forms in \cite{DLY,Ding2}
and a new key good unknown is found.
In addition, as in \cite{M-Y, L-Y},
strongly motivated by the geometric methods of D. Christodoulou \cite{Ch1},
one can introduce the inverse foliation density $\mu=-\dps\f{1}{\sum_{\al=0}^2g^{0\al}\p_{\al}u}$ from the equation
in \eqref{mainequation} to measure the squeezing of the characteristic surfaces, where the  optical function $u$ solves
for $t\ge t_0=\frac{2\delta}{\mathcal{A}}$,
\begin{equation}\label{YHCC-60}
\displaystyle\sum_{\alpha,\beta=0}^2 g^{\al\be}\p_{\al}u\p_{\be}u=0
\end{equation}
with $u(t_0,y)=t_0-r$ and $\p_0u>0$. If $\mu\to 0+$ can be shown for $t\le x^3_*$,
then a shock will be formed for problem \eqref{mainequation}.
We now give more concrete explanations.

At first, such a refined  property of the local smooth solution $\phi$ to \eqref{mainequation}
on \(t=t_0\) is derived (see details in Theorem \ref{Local} below): on the domain
$\{(t_0, y^1,y^2): \frac{1}{\mathcal{A}}-t_0\leq r\leq\frac{1}{\mathcal{A}}+t_0\}$,
it holds that for any $m\in\Bbb N$,
\begin{equation}\label{localsmallness}
(\partial_{0}+\partial_r)^m \phi(t_0,y)
=O(\delta^{\nu-(2-\nu)m}).
\end{equation}
This means that $(\partial_{0}+\partial_r)^m\phi(0, y)$ has a better smallness order $O(\delta^{\nu-(2-\nu)m})$ than the
usual order $O(\delta^{\nu-m})$. Meanwhile, due to the condition \eqref{condition1}, the crucial uniform negativity of  $\delta^{2-\nu}(\partial_{0}-\partial_r)\partial_{0}\phi(\frac{2\delta}{\mathcal{A}},$ $(t_0+\frac{1+s_0\delta}{\mathcal A})\omega_0)$
is established for small $\delta>0$.
Together with \eqref{YHCC3-1} and Theorem \ref{maintheorem} (A), this implies
that the outgoing characteristic conic surfaces for the equation of \eqref{mainequation} can intersect within a finite $t$-distance,
and then a shock is formed.

Secondly, compared with the smallness order $O(\delta^{\nu})$ restricted in \eqref{Yin.1-1} for the initial data,
the estimate \eqref{localsmallness} contains an additional large factor $\delta^{-(2-\nu)m}$. This will cause some
essential difficulties in obtaining the high order derivative estimates of $\phi$
by the energy method as in \cite{M-Y, L-Y}. Next we give some detailed explanations how to
overcome this difficulty.

\medskip
\textbf{$\bullet$ Introduction of a crucial good unknown $\check{\mathscr{X}}_{XX}$}
\medskip

To derive the energy estimates from the initial distance
$t=t_0$ for \eqref{mainequation}
as in \cite{Ch1, Sp}, the null frame $\{\mathring L, \mathring{\underline L}, X\}$ is introduced for the Lorentz
metric $(g_{\alpha\beta})$, where $\mathring L, \mathring{\underline L}$ and $X$ approximate $\p_0+\p_r$,  $\p_0-\p_r$ and $\p_\theta$
as $t\to t_0+0$, respectively. Under the coordinate transformation $(t, y^1, y^2)\longrightarrow (t, u, \theta)$ with $u=u(t,y)$ determined by \eqref{YHCC-60}, the corresponding energies in the domain $D^{t,u}=\{({t}', u',\theta): t_0\leq {t}'<{t}, -\frac{1}{\mathcal A}\leq u'\leq u, 0\le \theta\le 2\pi\}$
near the outgoing conic surface
$\{(t,y): r-t=\frac 1{\mathcal A}\}$ are defined as follows (see Section \ref{EE})
\begin{align}\label{YHCC-61}
\begin{split}
E_1[\phi](t, u) &:= \frac{1}{2} \int_{\Sigma_t^{u}} \mu \big\{ (\mathring{L}\p\phi)^2 + |\slashed{d}\p\phi|^2 \big\}, \\
E_2[\phi](t, u) &:= \frac{1}{2} \int_{\Sigma_t^{u}} \big\{ (\mathring{\underline{L}}\p\phi)^2 + \mu^2 |\slashed{d}\p\phi|^2\big\}, \\
F_1[\phi](t, u) &:= \int_{C_{u}^t} (\mathring{L}\p\phi)^2, \\
F_2[\phi](t, u) &:= \int_{C_{u}^t} \mu \, |\slashed{d}\p\phi|^2,\\
\end{split}
\end{align}
where $u\in [u_0, U_0]$ with $u_0=-\f 1{\mathcal A}$ and $U_0=-\f{1-4\delta}{\mathcal A}$,
$\Sigma_{ t}^{u}=\{({ t}',u',\theta): { t}'= t, u_0\leq u'\leq  u, 0\le \theta\le 2\pi\}$,
$C_{u}^{t}=\{({t}',u',\theta): t_0\leq {t}'\leq {t}, u'=u, 0\le \theta\le 2\pi\}$
and $\slashed{d}$ stands for the restriction of differential operator $d$ on the circle $\Bbb S$.

For $m\in\Bbb N$ and $\vp_{\kappa}=\p_{\kappa}\phi$ ($\kappa=0,1,2$), set $\Psi=Z^m\vp_{\kappa}$ with $Z'$s being some suitable vector fields
(the meanings of $Z\in\{\mathring{L}, T, R\}$ see \eqref{bootstrap-assumptions}).
It follows from the integration of $\mu(\Box_g\Psi) (V_i\Psi)$ with $V_1\Psi=\mathring{L}\Psi$ and $V_2\Psi=\mathring{\underline{L}}\Psi$ over
the domain $D^{t,u}$ that for $i=1,2$,
\begin{equation}\label{EI-Y}
E_i[\Psi](t, u)-E_i[\Psi](t_0, u)+F_i[\Psi](t, u)=-\int_{D^{t, u}}\mu\Box_g\Psi\cdot V_i\Psi-\int_{D^{t, u}}\f12\mu \dps\sum_{\alpha,\beta=0}^2Q_{\al\beta}[\Psi]\leftidx{^{(V_i)}}\pi^{\al\beta},
\end{equation}
where
$Q_{\al\beta}[\Psi]=(\p_\al\Psi)(\p_\beta\Psi)-\f12 g_{\al\beta}\dps\sum_{\nu,\lambda=0}^2g^{\nu\lambda}(\p_\nu\Psi)(\p_\lambda\Psi)$,
and $\leftidx{^{(V_i)}}\pi$ is the deformation tensor with respect to the vector field $V_i$ which is defined as
\begin{equation}\label{dt-Y}
{}^{(V_i)}\pi_{\alpha\beta}=g(\mathscr{D}_{\alpha}V_i, \partial_{\beta})+g(\mathscr{D}_{\beta}V_i, \partial_{\alpha}),
\end{equation}
where $\D$ is the Levi-Civita connection of the metric $g=(g_{\alpha\beta})$.

Obviously, under the frame $\{\mathring{L}, \mathring{\underline{L}}, X\}$, the expression of ${}^{(V_1)}\pi_{\alpha\beta}$
in \eqref{dt-Y} includes the second fundamental form $\chi_{XX} = g(\mathscr{D}_X \mathring{L}, X)$.
It is pointed out that due to the special structure of the 3D nonlinear equation $-(1+(\p_t\phi)^p)\p_t^2\phi+\Delta\phi=0$
($p\in\Bbb N$) in \cite{M-Y,L-Y}, $\chi_{XX}$ admits some relatively favorable $L^\infty$ or $L^2$ estimates by investigating the
first order partial differential equation of $\chi_{XX}$ along $\mathring{L}$.
However, owing to the more complicated structure of the quasilinear wave equation in \eqref{mainequation} and
the loss of smallness order in the estimate \eqref{localsmallness}, direct analysis on $\chi_{XX}$
along $\mathring{L}$ is no longer effective. Indeed, set
\begin{equation}\label{FG-Y}
G_{\alpha\beta}^{\kappa}= \partial_{\vp_{\kappa}} g_{\alpha\beta}\quad \text{for $\kappa=0,1,2$}
\end{equation}
and
\begin{equation}\label{FG-Y0}
G_{UV}^{\kappa} = \dps\sum_{\alpha,\beta=0}^2G_{\alpha\beta}^{\kappa} U^{\alpha} V^{\beta} \quad\text{for any vector fields
$U = \dps\sum_{\alpha=0}^2U^{\alpha} \partial_{\alpha}$ and $V = \dps\sum_{\alpha=0}^2V^{\alpha} \partial_{\alpha}$}.
\end{equation}
Let $\slashed g_{XX}=g(X, X)$ and $\varrho={t}-u$.
It is emphasized that the combination
\begin{equation}\label{FG-Y1}
\chi_{XX}-\frac{1}{2}\dps\sum_{\kappa=0}^2G_{XX}^\kappa\mathring{L}\vp_\kappa-\f{1}{\varrho}\slashed g_{XX}
\end{equation}
appears in several important relations (see \eqref{deL}, \eqref{Rpi}, \eqref{fequation}-\eqref{H} and so on), meanwhile
\begin{equation}\label{FG-Y001}
\chi_{XX}-\frac{1}{2}\dps\sum_{\kappa=0}^2G_{XX}^\kappa\mathring{L}\vp_\kappa-\f{1}{\varrho}\slashed g_{XX}=O(\delta^{\nu-1})
\end{equation}
is expected in almost all estimates in Sections \ref{Section 3}-\ref{ert}.
However, since  $|\mathring{L}\vp_\kappa| \lesssim \delta^{2\nu-3}$ and hence $|\dps\sum_{\kappa=0}^2G_{XX}^\kappa\mathring{L}\vp_\kappa|$
$\lesssim \delta^{2\nu-3}$ can be only obtained from Theorem \ref{Local} when $t\in [0, t_0]$, this implies that the term $\dps\sum_{\kappa=0}^2G_{XX}^\kappa\mathring{L}\vp_\kappa$
fails to achieve the required smallness order $O(\delta^{\nu-1})$ owing to the lack of the compatibility condition \eqref{Yin.1-1}.
Based on this key observation, we naturally choose a good unknown as follows
\begin{equation}\label{FG-Y2}
\check{\mathscr{X}}_{XX}=\chi_{XX}- \frac{1}{2}\dps\sum_{\kappa=0}^2G_{XX}^\kappa\mathring{L}\vp_\kappa-\f{1}{\varrho}\slashed g_{XX},
\end{equation}
which satisfies $\check{\mathscr{X}}_{XX}=O(\delta^{\nu-1})$  on the initial hypersurface $\Sigma_{t_0}^{U_0}$
(see \eqref{errorv-1} below).

By estimating $\mathring{L}\bigl(\varrho^2\operatorname{tr}\check{\mathscr X}\bigr)$ with
$\operatorname{tr}\check{\mathscr X}=\slashed g^{XX}\check{\mathscr{X}}_{XX}$ and integrating along the integral curves
of $\mathring{L}$, one can derive \eqref{FG-Y001} (see Lemma \ref{YHCC-555} and Proposition \ref{prop:lower-order-Linfty-estimates}
for details). Here we point out that the factor $\mathring{L}\vp_\kappa$ and its derivatives
in \eqref{FG-Y1} always admit favorable smallness orders and more precise long time decay rates
due to the null forms and the related compatibility condition in \cite{Ding2, Ding4}, and $\dps\sum_{\kappa=0}^2G_{XX}^\kappa\mathring{L}\vp_\kappa\equiv0$
holds for the special structures of nonlinear wave equations in \cite{M-Y, L-Y}.
Therefore, the term $\dps\sum_{\kappa=0}^2G_{XX}^\kappa\mathring{L}\vp_\kappa$
does not cause any serious difficulties in  \cite{Ding2, Ding4, M-Y, L-Y}. However, it is not the case
for our problem \eqref{mainequation} due to the different nonlinear structure and the lack of null condition.

\medskip
\textbf{$\bullet$ Choice of suitable weighted $\delta$-factor in energy estimate}
\medskip

Based on the required properties of $\check{\mathscr{X}}_{XX}$, through studying the delicate
structure of $\mu\Box_gZ^{m}\varphi_\kappa$ and deriving the energy estimates from \eqref{EI-Y},
we can obtain that in $D^{t,u}$ with $t_0\leq t \leq t_*$, it holds
\begin{equation}\label{bootstrap-assumptions-Y}
\delta^{l+k(2-\nu)} \| Z^{\alpha}\vp_{\kappa} \|_{L^{\infty}(\Sigma_t^u)}+\|\slashed\triangle\vp_{\kappa}\|_{L^{\infty}(\Sigma_t^u)}
\lesssim \delta^{\nu-1},
\quad \kappa= 0, 1, 2,
\end{equation}
where $|\alpha| \leq N$ with $N$ being a large positive integer, the vector field
$Z \in \{\mathring{L}, T, R\}$ with $T=\mu(-\dps\sum_{\alpha=0}^2g^{\al 0}\p_\al-\mathring L)$ and
$R=\slashed\Pi\Omega$ being the projection of $\Omega=y^1\p_2-y^2\p_1$ on $S_{t, u}$, and $l$ (resp. $k$) denotes the number of $T$ (resp. $\mathring{L}$)
contained in $Z^{\alpha}$. It is worth emphasizing that the choice of the important weighted $\delta$-factor
$\delta^{k(2-\nu)}$ in front of
$\|Z^{\alpha}\vp_{\kappa} \|_{L^{\infty}(\Sigma_t^u)}$ in \eqref{bootstrap-assumptions-Y} is strongly motivated by
\eqref{localsmallness}.

By \eqref{bootstrap-assumptions-Y}{,} together with delicate estimates on some related geometric quantities in Proposition \ref{prop:higher-order-L-infty-estimates}, one can achieve (see \eqref{YHCC-2} below)
\begin{equation*}
\begin{split}
\mu=1 + \frac{(\gamma+1) q_0^3\mathcal A}{4c^2(\rho_0)\bigl(q_0^2 - c^2(\rho_0)\bigr)} \delta^{\nu-2} \bigl(\mathcal A\p_s^2\Phi_0(s_0,\omega_0)-\p_s\Phi_1(s_0, \omega_0)\bigr)
\bigl(t + O(t^2)\bigr) + O(\delta^{\nu-1}).
\end{split}
\end{equation*}
From this, Theorem \ref{maintheorem} (A) will be shown.

Thirdly, in order to prove the crucial estimates in \eqref{bootstrap-assumptions-Y}, we first
give the following bootstrap assumptions for $\vp_{\kappa}$ ($\kappa=0,1,2$) in $D^{t,u}$ with $t_0\leq t \leq t_*$ and for
$|\alpha|\le N$ with a suitably large integer $N$:
\begin{equation}\label{bootstrap-assumptions-S}
\delta^{l+k(2-\nu)} \| Z^{\alpha}\vp_{\kappa} \|_{L^{\infty}(\Sigma_t^u)}+\|\slashed\triangle\vp_{\kappa}\|_{L^{\infty}(\Sigma_t^u)}
\le M\delta^{\nu-1},
\end{equation}
where $M>0$ is a suitably chosen constant, and the meanings of $l,k$ and $Z$ are as in \eqref{bootstrap-assumptions-Y}.
To close the bootstrap assumption \eqref{bootstrap-assumptions-S} up to the $N$-th order derivatives,
as in \cite{Ch2, K-R, M-Y, L-Y}, we need to obtain the
uniform positive constant $M$. For this purpose, our key ingredient is as follows: at first, it follows from the equation
in \eqref{mainequation} that for $\kappa=0,1,2,$
\begin{equation}\label{fequation-Y}
\mathring{L}\mathring{\underline{L}}\varphi_\kappa + \frac{1}{2\varrho}\mathring{\underline{L}}\varphi_\kappa
= \mu\slashed{\triangle}\varphi_\kappa + H_\kappa,
\end{equation}
where \begin{equation}\label{H-Y}
\begin{split}
H_\kappa=-\left(\operatorname{tr}\check{\mathscr X}\right) T\varphi_\kappa+\frac{1}{2\varrho}\mu\mathring{L}\varphi_\kappa+
\text{``good terms".}
\end{split}
\end{equation}
Together with those inequalities already obtained, one has (see \eqref{YHCCC-X1})
\begin{align*}
|\mathring{L}(\varrho^2\operatorname{tr}\check{\mathscr{X}})| \lesssim M\delta^{2\nu-3} + (\varrho^2\operatorname{tr}\check{\mathscr{X}})^2 + M\delta^{2\nu-3}(\varrho^2\operatorname{tr}\check{\mathscr{X}}),
\end{align*}
this yields the required precise estimates on $\operatorname{tr}\check{\mathscr{X}}$ and $\check{\mathscr{X}}$.
From this, one can get that in $D^{t,u}$ with $t_0\leq t \leq t_*$ and for
$|\alpha|\le N-2$,
\begin{equation}\label{YHCCC-X2}
\delta^{l+k(2-\nu)} \| Z^{\alpha}\vp_{\kappa} \|_{L^{\infty}(\Sigma_t^u)}+\|\slashed\triangle\vp_{\kappa}\|_{L^{\infty}(\Sigma_t^u)}
\le C\delta^{\nu-1},
\end{equation}
where the positive constant $C$ is independent of $M$.

Based on \eqref{YHCCC-X2} and the precise  smallness or largeness orders of
$\delta$ for all related quantities,
the bootstrap assumptions for the $(N-1)$-th order and $N$-th order derivatives of $\varphi_\kappa$ can be closed by the energy estimates
in \eqref{EI-Y} since the related constants only depend on $C$ rather than $M$ (see details in Sections \ref{EE}--\ref{ert}).
In this process, special attention will be paid to the large factor $\delta^{-(2-\nu)k}$  when $\mathring{L}$ acts on $\vp$ repeatedly.
On the other hand, we emphasize that although the proof procedure for \eqref{bootstrap-assumptions-S} is somewhat analogous to that
in \cite{L-Y, M-Y}, the details will be given in full, because the equation in \eqref{mainequation} is more complicated
than the concise 3D quasilinear wave equation $-(1+(\p_t\phi)^p)\p_t^2\phi+\Delta\phi=0$
($p\ge 2, p\in\Bbb N$) in \cite{L-Y, M-Y} as well as the crucial application and the treatment for the good unknown
$\check{\mathscr{X}}_{XX}$.

\medskip

The remainder of the paper is organized as follows.

In Section \ref{Section 2}, we first establish some basic properties of the local solution $\phi$ to \eqref{mainequation}.
In addition, the necessary geometric preliminaries including the introduction of optical function, inverse foliation density $\mu$
and null frame are presented. Meanwhile, a good unknown is introduced.
We also derive a series of important formulas for the covariant derivatives of the null frame and
for the deformation tensors as well as the evolution equation of $\mu$.

In Section \ref{Section 3}, the bootstrap assumptions are given and subsequently the uniform $L^\infty$ bounds for the
lower order derivatives of $\vp$ are established.
These crucial uniform bounds, which are independent of the constant in the bootstrap assumptions,
will be applied to the later analysis in the next sections.

In Section \ref{Sectionmu}, based on the estimates in Section \ref{Section 3},
as in  \cite{Ch1, Sp}, a detailed study on $\mu$ is conducted and then shock formation is shown under
the condition \eqref{condition1}.

In Sections \ref{EE}-\ref{ert}, through deriving the  energy estimates for the higher order derivatives
of $\vp$ and utilizing the Sobolev embedding theorem, the bootstrap assumptions in Section \ref{Section 3} are then closed.
This leads to the proof of Theorem \ref{maintheorem} (A).

Finally, in Section \ref{incoming}, we sketch the proof of Theorem \ref{maintheorem} (B) under condition \eqref{condition2}.
In this procedure, it is derived that
the inverse foliation density $\mu$ corresponding to the incoming characteristic conic surface
will tend to zero before the distance $\tilde {x}^3_*$.

\medskip

\textbf{Notations.}
Throughout the paper, Greek indices $\{\al, \beta, \cdots\}$ and Latin indices $\{i, j, k, \cdots\}$
denote the spacetime and spatial components, taking values in $\{0, 1, 2\}$ and $\{1, 2\}$, respectively.
In addition, the Einstein summation convention for repeated upper and lower indices is used.

For the nonnegative quantities $f$ and $g$,  $f\lesssim g$ means that there exists a universal constant $C>0$
such that $f\le Cg$, where $C$ is independent of the small parameter $\delta>0$.

For convenience, set $g^{\al\beta}=g^{\al\beta}(\p\phi)$ and let $(g_{\al\beta})$ be the inverse of the matrix $(g^{\al\beta})$.

The following notations will be used
\begin{align*}
\omega^i&=\omega_i, \quad i=1,2, \\
\omega&=(\omega_1,\omega_2), \quad \omega_\perp=(\omega_\perp^1,\omega_\perp^2)=(-\omega^2,\ \omega^1), \\
\partial&=(\partial_{y^0},\partial_{y^1},\partial_{y^2})=(\partial_t,\partial_{y^1},\partial_{y^2})=(\p_0,\p_1,\p_2), \\[2mm]
L&=\partial_t+\partial_r, \quad \underline{L} = \partial_t-\partial_r,\quad \Omega = y^1\p_2-y^2\p_1,\\[2mm]
S&=t\partial_t+r\partial_r=\frac{t- r}{2}\underline{L}+\frac{t+r}{2}L, \\[2mm]
H_i &=t\partial_{i}+y^i\partial_t=\omega^i\big(\frac{r-t}{2}\underline{L}+\frac{t+r}{2}L\big)+ \frac{t\omega^i_\perp}{r}\Omega,
\quad i=1,2, \\[2mm]
t_0 &= \frac{2\delta}{\mathcal A}, \quad t_* = x^3_*, \quad {\tilde t}_* = {\tilde  x}^3_*,\\[2mm]
\Sigma_t &= \{(t,y): y\in\mathbb R^2\}, \\[2mm]
\varphi_\kappa &= \partial_\kappa\phi, \quad \kappa=0,1,2.
\end{align*}

\section{Some preliminaries}\label{Section 2}
\subsection{Basic properties of the local solution}\label{Section 2-1}
In this subsection, we derive some crucial properties of the local smooth solution $\phi$ to problem \eqref{mainequation}.

\begin{theorem}\label{Local}
For small $\delta > 0$, \eqref{mainequation} admits a local solution $\phi \in C^{\infty}([0, t_0] \times \mathbb{R}^2)$.
Meanwhile, for any $m \in \mathbb{N}_0$, $q \in \mathbb{N}_0^3$ and $k \in \mathbb{N}_0$, the following estimates hold
\begin{align}
&|L^m \partial^q \Omega^k \phi(t_0, y)|
\leq \delta^{\nu-m(2-\nu)-|q|} \quad \text{for } \frac{1 - 2\delta}{\mathcal{A}} \leq r \leq \frac{1 + 2\delta}{\mathcal{A}},\label{initial1}\\
&|\underline L^m \partial^q \Omega^k \phi(t_0, y)|
\leq \delta^{\nu-m(2-\nu)-|q|} \quad \text{for } \frac{1 - 3\delta}{\mathcal{A}} \leq r \leq \frac{1 + \delta}{\mathcal{A}}.\label{initial2}
\end{align}
In addition, for any $r_0\in (\frac{1-3\delta}{\mathcal A}, \frac{1+2\delta}{\mathcal A})$,
\begin{equation}\label{y0}
\delta^{2-\nu}\underline L\partial_{0}\phi(t_0, r_0\omega)=\mathcal{A} \big(\mathcal A\partial_s^2\Phi_0\big(\tfrac{\mathcal A(r_0-t_0)-1}{\delta}, \omega\big)-\partial_s\Phi_1\big(\tfrac{\mathcal A(r_0-t_0)-1}{\delta}, \omega\big)\big)+ O(\delta^{\nu-1})
\end{equation}
and
\begin{equation}\label{y1}
\delta^{2-\nu}L\partial_{0}\phi(t_0, r_0\omega)=\mathcal{A} \big(\mathcal A\partial_s^2\Phi_0\big(\tfrac{\mathcal A(r_0+t_0)-1}{\delta}, \omega\big)+\partial_s\Phi_1\big(\tfrac{\mathcal A(r_0+t_0)-1}{\delta}, \omega\big)\big)+ O(\delta^{\nu-1}).
\end{equation}
\end{theorem}

\begin{proof}
At first, we give the following bootstrap assumption:
\begin{equation}\label{ba}
\begin{split}
&\text{\it when $0\leq t\leq t_0$ and integer $N_0\ge 6$,
$|\partial^\kappa\Omega^k\phi|\leq\delta^{\nu_0-|\kappa|}$ holds for $|\kappa|+k\leq N_0$},\\
&\text{\it where $\nu_0\in(1, \nu)$ is any fixed constant.}
\end{split}
\end{equation}
For \eqref{mainequation} and $n\in\mathbb N_0$, define the energy
$$
M_n(t)=\sum_{|\kappa|+k\leq n}\delta^{2|\kappa|}\|\partial\partial^\kappa\Omega^k\phi(t,\cdot)\|_{L^2(\mathbb R^2)}^2.
$$
Set $w=\delta^{|\kappa|}\partial^\kappa\Omega^k\phi$ with $|\kappa|+k\leq 2N_0-2$.
It follows from \eqref{mainequation} and integration by parts that
\begin{equation}\label{Y-4}
\begin{split}
&\int_0^t\int_{\Sigma_\tau}\partial_t w\,g^{\alpha\beta}\partial_{\alpha\beta}^2 w\,dyd\tau\\
=&\int_{\Sigma_{t}}\big(g^{0\alpha}\partial_\alpha w\,\partial_t w
-\frac12g^{\alpha\beta}\partial_\alpha w\,\partial_\beta w\big)dy
-\int_{\Sigma_{0}}\big(g^{0\alpha}\partial_\alpha w\,\partial_t w
-\frac12g^{\alpha\beta}\partial_\alpha w\,\partial_\beta w\big)dy\\
&+\int_0^t\int_{\Sigma_\tau}\big(-(\partial_\alpha g^{\alpha\beta})\partial_\beta w\,\partial_t w
+\frac12(\partial_t g^{\alpha\beta})\partial_\alpha w\,\partial_\beta w\big)dyd\tau,
\end{split}
\end{equation}
where
\begin{equation}\label{w}
\begin{split}
|g^{\alpha\beta}&\partial_{\alpha\beta}^2w|
\lesssim\delta^{|\kappa|}\sum_{\substack{
|\kappa_1|+|\kappa_2|\leq|\kappa|,\ k_1+k_2\leq k,\\
|\kappa_2|+k_2+2<|\kappa|+k}}
|\partial^{\kappa_1}\Omega^{k_1}g^{\alpha\beta}|
\cdot|\partial^{\kappa_2}\partial^2\Omega^{k_2}\phi|.
\end{split}
\end{equation}
Combining \eqref{ba}--\eqref{w} with Gronwall's inequality yields that for $0\leq t\leq t_0$,
$$
M_{2N_0-2}(t)\lesssim M_{2N_0-2}(0)\lesssim\delta^{2\nu-1}.
$$
Applying the Sobolev embedding on the circle $\mathbb S_r$ (centered at the origin with radius $r$), we arrive at
$$
|w(t,y)|\lesssim \frac{1}{\sqrt{r}}\|\Omega^{\le 1}w\|_{L^2(\mathbb S_r)}.
$$
Due to $r\sim 1$ for $t\in[0,t_0]$ and $(t,y)\in\mathrm{supp}\,w$, then one has that for $|\kappa|+k\leq N_0$,
\begin{equation}\label{local}
|\partial^\kappa\Omega^k\phi(t,y)|
\lesssim \|\Omega^{\le 1}\partial^\kappa\Omega^k\phi\|_{L^2(\mathbb S_r)}
\lesssim\delta^{1/2}\|\partial\Omega^{\le 1}\partial^\kappa\Omega^k\phi\|_{L^2(\Sigma_{t})}
\lesssim\delta^{\nu-|\kappa|}.
\end{equation}
Therefore, for small $\delta>0$ and $\nu_0<\nu$, the bootstrap assumption \eqref{ba} can be closed by the continuity argument.

Based on \eqref{local}, we now improve the $L^\infty$ estimate of $\phi$ on the domain $D_1=\{(t,y): 0\le t \le t_0,$
$\frac{1}{\mathcal A} - t\le r \le t+\frac{ 1}{\mathcal A}\}$.
The equation in \eqref{mainequation} can be rewritten as
\begin{equation}\label{me}
L\underline{L}\phi
= \frac{1}{2r}L\phi
- \frac{1}{2r}\underline{L}\phi
+ \frac{1}{r^2}\Omega^2\phi
+ 2g^{0i}\partial_{0i}^2\phi
+ (g^{ij}-m^{ij})\partial_{ij}^2\phi,
\end{equation}
where $m^{ii} = 1$ for $i=1, 2$ and $m^{ij} = 0$ for $i\neq j$.
It follows from \eqref{local} and \eqref{me} that
\begin{equation}\label{underlineLL}
|\underline{L}L\partial^q\Omega^k\phi(t,y)|
\le \delta^{2\nu-3-|q|},
\quad 0 \le t \le t_0.
\end{equation}
Note that $\phi$ vanishes on the surface $\{(t,y) : r- t = \frac{1}{\mathcal A}\}$, then integrating \eqref{underlineLL} along
the integral curves of $\underline{L}$ yields
\begin{equation}\label{YHCC4}
|L\partial^q\Omega^k\phi(t,y)|
\le \delta^{2\nu-2-|q|},
\quad (t,y) \in D_1.
\end{equation}
By applying $L$ repeatedly to both sides of \eqref{me} and using an induction argument,
we conclude that for $m\in\Bbb N$,
\[
|L^m\partial^q\Omega^k\phi(t,y)|
\le \delta^{\nu-m(2-\nu)-|q|},
\quad (t,y) \in D_1,
\]
which yields \eqref{initial1}.

Similarly, for any $(t,y) \in D_2=\{(t,y): 0\le t \le t_0,
\frac{1-\delta}{\mathcal A} - t \le r \le t+\frac{ 1-\delta}{\mathcal A}\}$,
\[
|\underline L^m\partial^q\Omega^k\phi(t,y)|
\le \delta^{\nu-m(2-\nu)-|q|},
\]
and hence \eqref{initial2} is proved.

For any fixed $\omega\in\Bbb S$, since both $(0,(r_0-t_0)\omega)$ and $(t_0, r_0\omega)$ lie on the line $\{(t,r\omega): r-t=r_0-t_0\}$,
then it follows from \eqref{me} that
\begin{equation}\label{YHCCC-3}
\begin{split}
&\underline{L}\partial_{0}\phi(t_0,r_0\omega)
=\underline{L}\partial_{0}\phi(0,(r_0-t_0)\omega)+\int_{(0,(r_0-t_0)\omega)}^{(t_0,r_0\omega)}L\underline{L}\partial_{0}\phi\\
=&\delta^{\nu-2}\mathcal{A}\big(\mathcal{A}\partial_s^2\Phi_0\big(\frac{\mathcal{A}(r_0-t_0)-1}{\delta}, \omega\big)-\partial_s\Phi_1\big(\frac{\mathcal{A}(r_0-t_0)-1}{\delta}, \omega\big)\big)+O(\delta^{2\nu-3}),
\end{split}
\end{equation}
which yields \eqref{y0}. Analogously, \eqref{y1} can be obtained
by integrating \eqref{me} along the integral curves of $\underline L$.
\end{proof}

\subsection{Lorentzian geometry, basic computations and introduction of a good unknown $\check{\mathscr{X}}$}\label{Section 2-2}
In this subsection, we present some preliminary concepts from Lorentzian geometry and related quantities near the outermost characteristic
conic surface introduced in \cite{Ch1},
one can also see \cite{M-Y, Sp, L-Y}. However, for the convenience of readers, we still provide a detailed exposition here.

As in \cite{Ch1}, the optical function is introduced as follows.
\begin{definition}\label{Definition optical function}
A $C^1$ function $u(t,y)$ is called the optical function of the equation in \eqref{mainequation}, if $u(t,y)$ satisfies the eikonal equation
\begin{equation}\label{eikonal equation}
g^{\al\be}\p_{\al}u\p_{\be}u=0
\end{equation}
with the initial data $u(t_0,y)=t_0-r$ and the condition $\p_{0}u>0$.
\end{definition}

According to \cite{Ch1, Sp}, the inverse foliation density $\mu$ is defined as
\begin{equation}\label{inverse foliation density}
\mu=-\f{1}{g^{0\al}\p_{\al}u}.
\end{equation}
Note that $\mu^{-1}$ measures the foliation density of the outgoing
characteristic surfaces{,} and the shock will be formed when $\mu\rightarrow 0+$.

Let us recall the null frame $\{\mathring L, \mathring{\underline L}, X\}$ with respect to the metric $(g_{\al\beta})$
(see \cite{Ding2, Ding4}):
\begin{equation}\label{Lr}
\begin{split}	
&\Lr=\mu\tilde L\quad\text{with $\tilde L=-\na u=-g^{\al\be}\p_{\al}u\p_{\be}$},\\
&\mathring{\underline L}=\mu\mathring L+2T
\quad \text{with $T=\mu\tilde T$ and $\tilde T=-g^{\al 0}\p_\al-\mathring L$},\\
&X=\f{\p}{\p\vartheta}\quad\text{with $\vartheta$ being determined by $\mathring L\vartheta=0$
and $\vartheta|_{t=t_0}=\th\in\mathbb{S}$}.\\	
\end{split}
\end{equation}
By direct computations, we have the following conclusion.
\begin{lemma}\label{nullframe}
It holds that
\begin{align*}
&g(\mathring L,\mathring L)=g(\mathring{\underline L}, \mathring{\underline L})=g(\mathring L, X)=g(\mathring{\underline L}, X)=0,
\quad g(\mathring L, \mathring{\underline L})=-2\mu,
\end{align*}
\begin{align}\label{LTTT}
g(\mathring L, T)=-\mu,\quad g(T,T)=\mu^2
\end{align}
and
\begin{align}\label{T}
&\mathring L t=1,\quad\mathring L u=0, \quad \mathring{\underline L}t=\mu,\quad \mathring{\underline L} u=2, \quad Tt=0,\quad T u=1.
\end{align}
\end{lemma}

As in \cite{Sp,M-Y}, one can perform the change of coordinates:
$(t, y^1, y^2)\longrightarrow (t, u, \vartheta)$ near $\{(t,y): r-t = \frac 1{\mathcal A}\}$ with
$({t}, u, \vartheta)=(t, u(t,y), \vartheta(t,y))$.
In the new coordinate $({t}, u, \vartheta)$, the following subsets are introduced.
\begin{definition} Set
\begin{align*}
&\Sigma_{ t}^{u}:=\{({t}',u',\vartheta): {t}'= t, u_0\leq u'\leq u\},\quad\text{where $u\in [u_0, U_0], u_0=-\frac1{\mathcal A}, U_0=-\frac{1-4\delta}{\mathcal A}$},\\
&C_{u}:=\{({t}',u',\vartheta): {t}'\geq t_0, u'=u\},\\
&C_{u}^{t}:=\{({t}',u',\vartheta): t_0\leq {t}'\leq {t}, u'=u\},\\
&S_{{t}, u}:=\Sigma_{t}\cap C_{u},\\
&D^{{t}, u}:=\{({t}', u',\vartheta): t_0\leq {t}'<{t}, 0\leq u'\leq u\}.
\end{align*}
\end{definition}
Next, the following geometric notations are used.

\begin{definition} For the metric $g$ on the spacetime,
\begin{itemize}
\item $\underline g=(g_{ij})$ is defined as the induced metric of $g$ on $\Sigma_{t}$, i.e.,
$\underline g(U,V)=g(U,V)$ for any tangent vectors $U$ and $V$ of $\Sigma_{t}$;
\item $\slashed \Pi_\al^\beta:=\delta_\al^\beta-\delta_\al^0\mathring L^\beta+\mathring L_\al\tilde{T}^\beta$ is the
projection tensor field of type $(1,1)$ on $S_{{t},u}$, where $\delta_\al^\beta$ is the Kronecker delta;
\item For any $(m,n)$-type
spacetime tensor field
$\eta$, $\slashed\eta=\slashed\Pi\eta$ is the tensor field on $S_{{t},u}$  with components $$\slashed\eta^{\al_1\cdots\al_m}_{\beta_1\cdots\beta_n}:=(\slashed\Pi\eta)^{\al_1\cdots\al_m}_{\beta_1\cdots\beta_n}
=\slashed\Pi_{\beta_1}^{\beta_1'}\cdots\slashed\Pi_{\beta_n}^{\beta_n'}
\slashed\Pi_{\al_1'}^{\al_1}\cdots\slashed\Pi_{\al_m'}^{\al_m}\eta^{\al_1'\cdots\al_m'}_{\beta_1'\cdots\beta_n'}.$$
In particular, $\slashed g=(\slashed g_{\al\beta})$ is the induced metric of $g$ on $S_{{t},u}$;
\item $\slashed g^{XX}$ is defined as the inverse of $\slashed g_{XX}$ with $\slashed g_{XX}=g(X, X)$;
\item $\mathscr D$ and $\slashed\nabla$ denote the Levi-Civita connections of $g$ and $\slashed g$, respectively;
\item $\Box_g:=g^{\al\beta}\mathscr{D}^2_{\al\beta}$, $\slashed\triangle:=\slashed g^{XX}\slashed\nabla^2_{X}$;
\item $\mathcal L_V\eta$ is the Lie derivative of $\eta$ with respect to $V${;} $\slashed{\mathcal L}_V\eta:=\slashed\Pi(\mathcal L_V\eta)$
for any tensor field $\eta$ and vector $V$;
\item For any $(m,n)$-type spacetime tensor field $\eta$,
$$
|\eta|^2:=g_{\al_1\al_1'}\cdots g_{\al_m\al_m'}g^{\beta_1\beta_1'}\cdots g^{\beta_n\beta_n'}\eta_{\beta_1\cdots\beta_n}^{\al_1\cdots\al_m}\eta_{\beta_1'\cdots\beta_n'}^{\al_1'\cdots\al_m'};
$$
\item $\text{div}U:=\mathscr D_\al U^\al$ for any vector field $U$; $\slashed{\text{div}}Y:=\slashed{\nabla}_X Y^X$
and $\slashed{\text{div}}\zeta:=\slashed\nabla^X\zeta_X$ denote the angular divergence for any vector field $Y$ and 1-form $\zeta$ on $S_{{t},u}$;
\item If $\eta$ is a $(0,2)$-type spacetime tensor, $\Lambda$ is a 1-form, $U$ and $V$ are vector fields, then the contraction
of $\eta$ with respect to $U$ and $V$ is defined as
$\eta_{UV}:=\eta_{\al\beta}U^{\al}V^{\beta},$
and the contraction of $\Lambda$ with respect to $U$ is
$\Lambda_U:=\Lambda_{\al}U^{\al}$;
\item If $\eta$ is a $(0,2)$-type tensor on $S_{t,u}$, the trace of $\eta$ is defined as
$\operatorname{tr}\eta:=\slashed g^{XX}\eta_{XX}$;
\item  $\ds$ stands for the restriction of differential operator $d$ on the circle $\Bbb S$, and $\ds f=Xf$ for any smooth function $f$.
\end{itemize}
\end{definition}

In the frame $\{\mathring L,T,X\}$, the following decompositions hold
\begin{align}
g^{\alpha\beta} &= -\mathring L^\alpha\mathring L^\beta -\widetilde T^\alpha\,\mathring L^\beta -\mathring L^\alpha\,\widetilde T^\beta +(\slashed d_X y^\alpha)(\slashed d^X y^\beta), \label{gab}\\ \partial_\alpha &= \delta_\alpha^0\,\mathring L -\mu^{-1}\mathring L_\alpha\,T + g_{\alpha i}(\slashed d^X y^i)X. \label{pal}
\end{align}

Define the second fundamental forms
\[
\chi_{XX} = g(\mathscr D_X\mathring L,X), \qquad \sigma_{XX} = g(\mathscr D_X\widetilde T,X)
\]
and the one-forms
\[
\zeta_X = g(\mathscr D_X\mathring L,\widetilde T), \qquad \xi_X = -g(\mathscr D_X T,\mathring L).
\]
Then $\mu \zeta_X = -X\mu + \xi_X$ holds.

Introduce the deformation tensor
\begin{equation}\label{dt}
 {}^{(V)}\pi_{\alpha\beta} = g(\mathscr{D}_{\alpha} V, \partial_{\beta}) + g(\mathscr{D}_{\beta} V, \partial_{\alpha}).
\end{equation}

Let
\[
G_{\alpha\beta}^{\kappa} := \partial_{\varphi_\kappa} g_{\alpha\beta}\quad\text{with $\varphi_\kappa=\partial_\kappa\phi$
for $\kappa=0,1,2$.}
\]
For any vector fields $U = U^{\alpha} \partial_{\alpha}$ and $V = V^{\alpha} \partial_{\alpha}$, set
$G_{UV}^{\kappa} = G_{\alpha\beta}^{\kappa} U^{\alpha} V^{\beta}$.

On the initial hypersurface $\Sigma_{t_0}^{U_0}$, it holds that $\tilde T^i=-\f{y^i}{r}+O(\delta^{\nu-1})$,
$\mathring L^0=1$, $\mathring L^i=\f{y^i}{r}+O(\delta^{\nu-1})$ and
\begin{equation}\label{YHCC5}
 \chi_{XX}-\f12G_{XX}^\kappa\mathring L\vp_\kappa=\f1r\slashed g_{XX}+O(\delta^{\nu-1}).
\end{equation}
Note that on $\Sigma_{t_0}$, $r$ coincides with $t_0-u$. For $t\geq t_0$, we define the following \textit{error vectors} with the components
\begin{equation}\label{errorv}
\begin{split}
&\check{L}^0:=0,\ \check{L}^i:=\mathring L^i-\f{y^i}{\varrho},\ \check{T}^i:=\tilde T^i+\f{y^i}{\varrho},\ \check\varrho:=\f{r}{\varrho}-1,\\
\end{split}
\end{equation}
and introduce the following new good unknown
\begin{equation}\label{errorv-1}
\begin{split}
&\check{\mathscr{X}}_{XX}:=\chi_{XX}- \frac{1}{2}G_{XX}^\kappa\mathring{L}\vp_\kappa-\f{1}{\varrho}\slashed g_{XX},
\end{split}
\end{equation}
where $\varrho:={t}-u$.

In the new coordinate system $(t, u, \vartheta)$, one has $\mathring{L} = \frac{\partial}{\partial t}$. Additionally, it
follows from \eqref{T} that
there exists a smooth function $\Theta^X$ such that $T=\frac{\partial}{\partial u} - \Theta^X X$.
In addition, an analysis analogous to that in \cite[Lemma 3.66]{Sp} yields the following result.

\begin{lemma}
In domain $D^{t,u}$, the Jacobian determinant of the map $(t, u, \vartheta) \rightarrow (t, y^1, y^2)$ is given by
\begin{equation}\label{Jacobian}
\det\frac{\partial(t, y^1, y^2)}{\partial(t, u, \vartheta)} = \mu \left(\det\underline{g}\right)^{-1/2} \sqrt{\slashed{g}_{XX}}.
\end{equation}
\end{lemma}

\begin{remark}
The formula~\eqref{Jacobian} implies that for regular  metrics $\underline g$ and $\slashed g$  (i.e., $\det\underline{g} > 0$ and
$\slashed{g}_{XX} > 0$), the transformation $(t,u,\vartheta)\mapsto (t,y)$
becomes singular as $\mu\to 0^+$, which corresponds to the  formation of shocks.
\end{remark}

By (3.39b) of \cite{Sp},
the projections of $\Omega=y^1\p_2-y^2\p_1$ and  the differential $d$ on $S_{t, u}$ are defined
as
\begin{equation*}
R:=\slashed\Pi\Omega,\quad \slashed d:=\slashed\Pi d,
\end{equation*}
respectively. Note that the explicit expression of $R$
can be written as
\begin{equation}\label{R}
R=(\delta_j^i+g_{\al j}\mathring L^\al\tilde T^i){\Omega}^j\p_i
=\Omega-g_{aj}\tilde{T}^a{\Omega}^j\tilde{T},
\end{equation}
where $\delta_j^i$ is Kronecker delta symbol. For brevity, set
\begin{equation}\label{omega}
\upsilon:=g_{ab}\tilde{T}^a{\Omega}^b=g_{aj}\check{T}^a\epsilon_i^jy^i-(g_{aj}-m_{aj})\f{\epsilon_i^j y^iy^a}\varrho,
\end{equation}
where $\epsilon_1^2=1$, $\epsilon_2^1=-1$ and $\epsilon_1^1=\epsilon_2^2=0$.
Then we have
\begin{equation}\label{YHCC16}
R=\Omega-\upsilon\tilde{T}.
\end{equation}

For the quantities $\mu, \check{\mathscr{X}}$, $\sigma$, $\zeta$ and $\xi$ defined above, we can derive the following
basic equalities in the frame $\{\mathring{L}, \mathring{\underline{L}}, X\}$ or $\{T, \mathring{\underline{L}}, X\}$.
Since the proof is completely analogous to that in \cite{Ding4}, here we omit the details.

\begin{lemma}\label{Lem3.3}
It holds that
\begin{itemize}[leftmargin=*, labelsep=6pt, itemsep=10pt]
\item $\mu$ satisfies
\begin{equation}\label{lmu}
\mathring{L}\mu = -\frac{1}{2}\mu G_{\mathring{L}\mathring{L}}^\gamma \mathring{L}\vp_\gamma
- \mu G_{\tilde{T}\mathring{L}}^\gamma \mathring{L}\vp_\gamma + \frac{1}{2}G_{\mathring{L}\mathring{L}}^\gamma T\vp_\gamma;
\end{equation}
\item $\zeta$ and $\sigma$ satisfy the following relations
\begin{align}
&\zeta_X = -\frac{1}{2}\big(G_{\tilde{T}\mathring{L}}^\gamma \slashed{d}_X\vp_\gamma + G_{\tilde{T}\tilde{T}}^\gamma \slashed{d}_X\vp_\gamma
- G_{X\tilde{T}}^\gamma \mathring{L}\vp_\gamma + G_{X\mathring{L}}^\gamma \tilde{T}\vp_\gamma \big),\label{zeta}\\
&\sigma_{XX} = -G_{X\mathring{L}}^\gamma \slashed{d}_X\vp_\gamma - G_{X\tilde{T}}^\gamma \slashed{d}_X\vp_\gamma
+ \frac{1}{2}G_{XX}^\gamma \tilde{T}\vp_\gamma - \check{\mathscr{X}}_{XX} - \frac{\slashed{g}_{XX}}{\varrho};\label{theta}
\end{align}
\item the components of the vector fields $\mathring{L}$ and $\check{L}$ fulfill
\begin{align}
&\mathring{L}\mathring{L}^i = \frac{1}{2}G_{\mathring{L}\mathring{L}}^\gamma \mathring{L}\vp_\gamma \tilde{T}^i
-(G_{X\mathring{L}}^\gamma \mathring{L}\vp_\gamma - \frac{1}{2}G_{\mathring{L}\mathring{L}}^\gamma \slashed{d}_X\vp_\gamma )\slashed{d}^X y^i,\label{LL}\\
&T\mathring{L}^i = (\slashed{d}_X\mu - \frac{1}{2}\mu G_{\tilde{T}\tilde{T}}^\gamma \slashed{d}_X\vp_\gamma
- G_{X\mathring{L}}^\gamma T\vp_\gamma)\slashed{d}^X y^i + \frac{1}{2}G_{\mathring{L}\mathring{L}}^\gamma T\vp_\gamma \tilde{T}^i \nonumber\\
&\qquad\quad +(\frac{1}{2}\mu G_{\mathring{L}\mathring{L}}^\gamma \mathring{L}\vp_\gamma + \mu G_{\tilde{T}\mathring{L}}^\gamma \mathring{L}\vp_\gamma
+ \frac{1}{2}\mu G_{\tilde{T}\tilde{T}}^\gamma \mathring{L}\vp_\gamma)\mathring{L}^i,\label{TL}\\
&\mathring{L}\left( \varrho\check{L}^i \right) = \varrho\mathring{L}\mathring{L}^i,\label{LeL}\\
&\slashed{d}_X\check{L}^i = \operatorname{tr}\check{\mathscr{X}} \slashed{d}_X y^i-(G_{\mathring{L}\tilde{T}}^\gamma \slashed{d}_X\vp_\gamma
+ \frac{1}{2}G_{\tilde{T}\tilde{T}}^\gamma \slashed{d}_X\vp_\gamma)\tilde{T}^i;\label{deL}
\end{align}
\item the covariant derivatives of the vector fields $\{T, \mathring{L}, X\}$ are given by
\begin{equation}\label{cdf}
\begin{split}
&\mathscr{D}_{\mathring{L}}\mathring{L} = (\mu^{-1}\mathring{L}\mu)\mathring{L}, \quad
\mathscr{D}_{T}\mathring{L} = -\mathring{L}\mu\mathring{L} + \xi^X X, \quad
\mathscr{D}_X\mathring{L} = -\zeta_X\mathring{L} + \operatorname{tr}\chi \, X,\\
&\mathscr{D}_{\mathring{L}}T = -\mathring{L}\mu\mathring{L} - \mu\zeta^X X, \quad
\mathscr{D}_T T = \mu\mathring{L}\mu\mathring{L} + (\mu^{-1}T\mu + \mathring{L}\mu)T - \mu(\slashed{d}^X\mu) X,\\
&\mathscr{D}_X T = \mu\zeta_X\mathring{L} + \mu^{-1}\xi_X T + \mu\operatorname{tr}\sigma \, X,\\
&\mathscr{D}_X X = \slashed{\nabla}_X X + (\sigma_{XX} + \chi_{XX})\mathring{L} + \mu^{-1}\chi_{XX} T;
\end{split}
\end{equation}
\item the covariant derivatives of the vector fields $\{\mathring{\underline{L}}, \mathring{L}, X\}$ are
\begin{equation}\label{LuL}
\begin{split}
&\mathscr{D}_{\mathring{\underline{L}}}\mathring{L} = -\mathring{L}\mu\mathring{L} + 2\xi^X X, \quad
\mathscr{D}_{\mathring{L}}\mathring{\underline{L}} = -2\mu\zeta^X X,\\
&\mathscr{D}_{\mathring{\underline{L}}}\mathring{\underline{L}} = (\mu^{-1}\mathring{\underline{L}}\mu + \mathring{L}\mu)\mathring{\underline{L}} - (2\mu \slashed{d}^X\mu) X.
\end{split}
\end{equation}
\end{itemize}
\end{lemma}
Let $\leftidx{^{(V)}}{\slashed{\pi}}_{UX} := \leftidx{^{(V)}}{\pi}_{UX}$ for $U \in \{\mathring{L}, \underline{\mathring{L}}, X\}$. As in \cite[Appendix A]{Ding4}, the deformation tensors under the frame $\{\mathring{L}, \underline{\mathring{L}}, X\}$ are listed in the following lemma.

\begin{lemma}The deformation tensors corresponding to different vector $V$ are given as follows:
\begin{itemize}\label{lem3.4}
\item For $V = T$,
\begin{equation}\label{Lpi}
\begin{split}
&\leftidx{^{(T)}}{\pi}_{\mathring{L}\mathring{L}} = 0,
\quad \leftidx{^{(T)}}{\pi}_{T\tilde{T}} = 2T\mu,
\quad \leftidx{^{(T)}}{\pi}_{\mathring{L}T} = -T\mu,
\quad \leftidx{^{(T)}}{\slashed{\pi}}_{TX} = 0, \\
&\leftidx{^{(T)}}{\slashed{\pi}}_{\mathring{L}X} = -2\mu\zeta_X - \slashed{d}_X\mu,
\quad \leftidx{^{(T)}}{\slashed{\pi}}_{XX} = 2\mu\sigma_{XX};
\end{split}
\end{equation}

\item For $V = \mathring{L}$,
\begin{equation}\label{uLpi}
\begin{split}
&\leftidx{^{(\mathring{L})}}{\pi}_{\mathring{L}\mathring{L}} = 0,
\quad \leftidx{^{(\mathring{L})}}{\pi}_{T\tilde{T}} = 2\mathring{L}\mu,
\quad \leftidx{^{(\mathring{L})}}{\pi}_{\mathring{L}T} = -\mathring{L}\mu,
\quad \leftidx{^{(\mathring{L})}}{\slashed{\pi}}_{\mathring{L}X} = 0, \\
&\leftidx{^{(\mathring L)}}{\slashed\pi}_{TX}=2\mu\zeta_X+\slashed d_X\mu,
\quad \leftidx{^{(\mathring{L})}}{\slashed{\pi}}_{XX} = 2\chi_{XX};
\end{split}
\end{equation}

\item For $V = R$,
\begin{equation}\label{Rpi}
\begin{split}
&\leftidx{^{(R)}}{\pi}_{\mathring{L}\mathring{L}} = 0,
\quad \leftidx{^{(R)}}{\pi}_{T\tilde{T}} = 2R\mu,
\quad \leftidx{^{(R)}}{\pi}_{\mathring{L}T} = -R\mu, \\
&\leftidx{^{(R)}}{\slashed{\pi}}_{\mathring{L}X} = -{R}^X\check{\mathscr{X}}_{XX}
-\upsilon(G_{\mathring{L}\tilde{T}}^\gamma\slashed{d}_X\vp_\gamma
+ \frac{1}{2}G_{\tilde{T}\tilde{T}}^\gamma\slashed{d}_X\vp_\gamma
- G_{X\tilde{T}}^\gamma{\mathring{L}}\vp_\gamma)
+ \epsilon_i^jg_{ja}\check{L}^i\slashed{d}_Xy^a, \\
&\leftidx{^{(R)}}{\slashed{\pi}}_{TX} = \mu{R}^X\check{\mathscr{X}}_{XX}
+ \upsilon\slashed{d}_X\mu
+ \mu G_{X\mathring{L}}^\gamma R\vp_{\gamma}
+ \upsilon(G_{X\tilde{T}}^\gamma T\vp_\gamma
- \frac{1}{2}\mu G_{\tilde{T}\tilde{T}}^\gamma\slashed{d}_X\vp_{\gamma})\\
&\qquad\quad + \mu G_{X\tilde{T}}^\gamma R\vp_{\gamma}
+ \mu\epsilon_i^j g_{ja}\check{T}^i\slashed{d}_Xy^a, \\
&\leftidx{^{(R)}}{\slashed{\pi}}_{XX} = 2\upsilon(\check{\mathscr{X}}_{XX}
+ \frac{\slashed{g}_{XX}}{\varrho})
+ \upsilon(2G_{X\mathring{L}}^\gamma\slashed{d}_X\vp_{\gamma}
+ 2G_{X\tilde{T}}^\gamma\slashed{d}_X\vp_{\gamma})
+ G_{XX}^\gamma R\vp_{\gamma} \\
&\qquad\quad + 2\epsilon_i^j\check{g}_{ja}(\slashed{d}_Xy^a)\slashed{d}_Xy^i,
\end{split}
\end{equation}
where $\check{g}_{ia} := g_{ia}-m_{ia}$.
\end{itemize}
\end{lemma}
In addition, the commutators among vector fields, covariant derivatives, and Lie derivatives are also presented in \cite[Appendix A]{Ding4}
as follows.
\begin{lemma}\label{com}
For any vector field $Z \in \{\mathring{L}, T, R\}$, the following identities hold

 \begin{itemize}
  \item the commutators of the vector fields are
  \begin{equation}\label{c}
  \begin{split}
  [\mathring{L}, R] &= \leftidx{^{(R)}}{\slashed{\pi}}_{\mathring{L}}{}^{X}X, \quad
  [\mathring{L}, T] = \leftidx{^{(T)}}{\slashed{\pi}}_{\mathring{L}}{}^{X}X, \quad
  [T, R] = \leftidx{^{(R)}}{\slashed{\pi}}_{T}{}^{X}X.
  \end{split}
  \end{equation}
  \item for any smooth function $f$,
  \begin{align}
  \bigl([\slashed{\nabla}^2, \slashed{\mathcal{L}}_Z]f\bigr)_{XX}
  &= \frac{1}{2} \slashed{\nabla}_X\bigl(\operatorname{tr}\,\leftidx{^{(Z)}}{\slashed{\pi}}\bigr) \slashed{d}_Xf, \label{nZf}\\
  [\slashed{\triangle}, Z]f &= \leftidx{^{(Z)}}{\slashed{\pi}}^{XX} \slashed{\nabla}_X^2f
  + \frac{1}{2} \bigl(\slashed{\nabla}_X \leftidx{^{(Z)}}{\slashed{\pi}}^{XX}\bigr) \slashed{d}_Xf; \label{LZf}
  \end{align}
  \item for any one-form $\eta$ on $S_{t, u}$,
  \begin{equation}\label{nZx}
  \bigl([\slashed{\nabla}_X, \slashed{\mathcal{L}}_Z]\eta\bigr)_X = \frac{1}{2} \slashed{\nabla}_X\bigl(\operatorname{tr}\,\leftidx{^{(Z)}}{\slashed{\pi}}\bigr) \eta_X;
  \end{equation}
  \item for any $(0,2)$-type tensor $\eta$ on $S_{t, u}$,
  \begin{align}
  \bigl([\slashed{\nabla}_X, \slashed{\mathcal{L}}_Z]\eta\bigr)_{XX}
  &= \slashed{\nabla}_X\bigl(\operatorname{tr}\,\leftidx{^{(Z)}}{\slashed{\pi}}\bigr) \eta_{XX}. \label{nZxi}
  \end{align}
 \end{itemize}
\end{lemma}
Finally, we introduce some abbreviated integrations over different domains and norms over $\Sigma_t^u$.
\begin{definition}\label{3.4}
 For any continuous function $f$ and tensor field $\eta$, set
 \begin{align*}
 \int_{S_{t, u}} f &:= \int_{S_{t, u}} f \, d\nu_{\slashed{g}}
 = \int_{\mathbb{S}} f(t, u, \vartheta) \sqrt{\slashed{g}_{XX}(t, u, \vartheta)} \, d\vartheta, \\[6pt]
 \int_{C^t_u} f &:= \int_{t_0}^t \int_{S_{\tau, u}} f(\tau, u, \vartheta) \, d\nu_{\slashed{g}} \, d\tau, \\[3pt]
 \int_{\Sigma_t^u} f &:= \int_{u_0}^u \int_{S_{t, u'}} f(t, u', \vartheta) \, d\nu_{\slashed{g}} \, du', \\[6pt]
 \int_{D^{t, u}} f &:= \int_{t_0}^t \int_{u_0}^u \int_{S_{\tau, u'}} f(\tau, u', \vartheta) \, d\nu_{\slashed{g}} \, du' \, d\tau, \\[6pt]
 \|\eta\|_{t,u} &:= \sqrt{\int_{\Sigma_t^u} |\eta|^2}.
 \end{align*}
\end{definition}

\section{Bootstrap assumptions and $L^{\infty}$ estimates for lower order derivatives}\label{Section 3}

By virtue of Theorem \ref{Local} and the definition of $\{\mathring{L}, T, R\}$ in Section \ref{Section 2}, under the condition
\eqref{condition1}, we make the following bootstrap assumptions in $D^{t,u}$:
\begin{equation}\label{bootstrap-assumptions}
\delta^{l+k(2-\nu)} \| Z^{m}\vp_{\al} \|_{L^{\infty}(\Sigma_t^u)}+\|\slashed\triangle\vp_{\al}\|_{L^{\infty}(\Sigma_t^u)} \leq M \delta^{\nu-1},
\quad \al= 0, 1, 2,
\end{equation}
where $\vp_{\al} = \partial_{\al} \phi$, $t_0\leq t \leq t_*=
 \frac{4c(\rho_0)}{(\gamma+1)\big(\partial_s\Phi_1(s_0,\omega_0)-\mathcal A\,\partial_s^2\Phi_0(s_0,\omega_0)\big)}\big(1-\frac{1}{M_0^2}\big)^{3/2}\delta^{2-\nu}$,
$m \leq N$ with $N$ a large positive integer, $M$ is a suitably chosen positive number,
$Z \in \{\mathring{L}, T, R\}$, and $l$ (resp. $k$) denotes the number of $T$ (resp. $\mathring{L}$)
contained in $Z^{m}$.
Note that \eqref{bootstrap-assumptions} holds at $t=t_0$ by Theorem \ref{Local}.

By a derivation process similar to that in \cite[(2.24) and (2.40)-(2.41)]{Ding4}, we can obtain that $\vp_{\alpha}$ satisfies
\begin{equation}\label{ge}
\begin{split}
\mu\Box_g \varphi_\al=&\f12(g^{\beta'\beta}G_{\beta'\beta}^
{\al'})(\mu\slashed d^X\varphi_{\al'})\slashed d_X\varphi_\al+f(\varphi, \mathring L^1, \mathring L^2)\left(
\begin{array}{ccc}
T\varphi\\
\mu\mathring L\varphi\\
\mu(\slashed d_Xy)\slashed d^X\varphi\\
\end{array}
\right)
\left(
\begin{array}{ccc}
\mathring L\varphi\\
(\slashed d_Xy)\slashed d^X\varphi\\
\end{array}
\right)
\end{split}
\end{equation}
and
\begin{equation}\label{fequation}
\mathring{L}\mathring{\underline{L}}\varphi_\al+ \frac{1}{2\varrho}\mathring{\underline{L}}\varphi_\al
= \mu\slashed{\triangle}\varphi_\al + H_\al,
\end{equation}
where $\vp=(\vp_0,\vp_1,\vp_2)$, $H_\al$ is given by
\begin{equation}\label{H}
\begin{split}
H_\al=& -\left(\operatorname{tr}\check{\mathscr X}\right) T\varphi_\al + \frac{1}{2\varrho}\mu\mathring{L}\varphi_\al
+ f\big(\varphi, \mathring{L}^1, \mathring{L}^2, \slashed{d}y, \slashed{g}\big)
\begin{pmatrix}
T\varphi \\
\mu\mathring{L}\varphi \\
\mu\slashed{d}\varphi
\end{pmatrix}
\begin{pmatrix}
\mathring{L}\varphi \\
\slashed{d}\varphi
\end{pmatrix}.
\end{split}
\end{equation}
Here $f$ denotes a generic smooth function of its arguments, and the product
$
\begin{pmatrix} A_1 \\ \vdots \\ A_n \end{pmatrix}
\begin{pmatrix} B_1 \\ \vdots \\ B_m \end{pmatrix}
$
represents all terms of the form $A_iB_j$ for $1 \leq i \leq n$ and $1 \leq j \leq m$.

To estimate $\vp_\alpha$ in terms of \eqref{fequation}, we need to treat the term $\operatorname{tr}\check{\mathscr X}$ in \eqref{H}.
To this end, we now derive the equation of $\operatorname{tr}\check{\mathscr X}$. In fact, by the analogous method to that in \cite[(2.44) and (2.45)]{Ding4}, one has
\begin{lemma}\label{YHCC-555}
 \begin{equation}\label{Lchi'}
 \begin{split}
 \mathring{L}\bigl(\varrho^2\operatorname{tr}&\check{\mathscr X}\bigr)
 = -\bigl(\frac{1}{2}G_{\mathring{L}\mathring{L}}^\al\mathring{L}\varphi_\al
 + G_{\tilde{T}\mathring{L}}^\al\mathring{L}\varphi_\al
 - G_{X\mathring{L}}^\al\slashed{d}^X\varphi_\al
 + \slashed{g}^{XX}G_{XX}^\al\mathring{L}\vp_\al\bigr)
 \big(\varrho^2\operatorname{tr}\check{\mathscr X} + \varrho\big) \\
 &-\varrho^2G_{X\mathring{L}}^\al\bigl(\slashed{d}^X\mathring{L}\varphi_\al\bigr)
 + \frac{1}{2}\varrho^2G_{\mathring{L}\mathring{L}}^\al\slashed{\triangle}\varphi_\al
 - \varrho^2\bigl(\operatorname{tr}\check{\mathscr X}\bigr)^2+ \varrho^2f\bigl(\varphi, \slashed{d}y, \mathring{L}^1, \mathring{L}^2\bigr)\slashed{d}\vp
 \begin{pmatrix} \mathring{L}\varphi \\ \slashed{d}\varphi \end{pmatrix}
 \end{split}
 \end{equation}
 and
 \begin{equation}\label{Tchi'}
 \begin{split}
 T\bigl(&\operatorname{tr}\check{\mathscr X}\bigr)=\slashed\triangle\mu-G_{X\mathring L}^\al\slashed d^XT\vp_\al-\f12\mu G_{\tilde T\tilde T}^\al\slashed\triangle\vp_\al-\mu(\operatorname{tr}\check{\mathscr X})^2+\f{\mu-1}{\varrho^2}\\
 &+\f12\bigr(\mu G_{\mathring L\mathring L}^\al\mathring L\vp_\al+2\mu G_{\tilde T\mathring L}^\al\mathring L\vp_\al-G_{\mathring L\mathring L}^\al T\vp_\al-G_{\tilde T\mathring L}^\al T\vp_\al-2\mu G_{X\mathring L}^\al\slashed d^X\vp_\al-2\operatorname{tr}\leftidx{^{(T)}}{\slashed{\pi}}\bigr)\operatorname{tr}\check{\mathscr X}\\
 &-\slashed d^X\mu\bigr(G_{\tilde T\mathring L}^\al\slashed d_X\vp_\al+G_{\tilde T\tilde T}^\al\slashed d_X\vp_\al+\f12G_{XX}^\al\slashed d_X\vp_\al\bigr)+f(\varphi, \slashed dy,\slashed d\vp,\slashed\triangle y, \mathring L^1,\mathring L^2)
 \left(
 \begin{array}{ccc}
 \mu\mathring L\varphi\\
 \mu\slashed d\varphi\\
 T{\varphi}\\
 \end{array}
 \right).
 \end{split}
 \end{equation}
\end{lemma}

\begin{remark}
Note that the last term in \eqref{Lchi'} contains only $(\slashed{d}\vp)(\mathring{L}\vp)$ and $\slashed{d}\vp \cdot \slashed{d}\vp$
rather than the more troublesome term $(\mathring{L}\vp)^2$, which is one of the reasons that we introduce the new good unknown $\check{\mathscr{X}}$.
This property will be crucial for estimating $\operatorname{tr}\check{\mathscr{X}}$ precisely later (see \eqref{eq:lower-order-Linfty-estimates} below).
\end{remark}

Next, we derive some rough $L^{\infty}$ estimates for the lower order derivatives of some quantities.

\begin{proposition}\label{prop:lower-order-Linfty-estimates}
 Under the assumptions \eqref{bootstrap-assumptions}, when $\de > 0$ is small,
the following estimates hold that for $t_0 \leq t \leq t_*$,
 \begin{equation}\label{eq:lower-order-Linfty-estimates}
 \begin{split}
 &\|\check{\mathscr{X}}\|_{L^\infty(\Si_t^u)} + \|\check{L}^i\|_{L^\infty(\Si_t^u)} + \|\check{T}^i\|_{L^\infty(\Si_t^u)} + \|\check{\varrho}\|_{L^\infty(\Si_t^u)} + \|\upsilon\|_{L^\infty(\Si_t^u)} \les M\de^{\nu-1}, \\
 &\|\mu\|_{L^\infty(\Si_t^u)} \les M, \qquad \|\slashed{d}y^i\|_{L^\infty(\Si_t^u)} \lesssim 1.
 \end{split}
 \end{equation}
\end{proposition}

\begin{proof}
It follows from \eqref{LTTT} that $g(\tilde{T}, \tilde{T}) = 1$, which implies
 \[
 1 = g_{ij}\tilde{T}^i\tilde{T}^j = \big(1 + O(M\delta^{\nu-1})\big)\sum_{i=1}^2|\tilde{T}^i|^2.
 \]
This leads to
 \begin{equation}\label{eq:LT}
 |\tilde{T}^i|, \, |\mathring{L}^i| \leq 1 + O(M\delta^{\nu-1}).
 \end{equation}

 In addition, by the identity $|\slashed{d}y^i|^2 = \slashed{g}^{ab}\slashed{d}_a y^i\slashed{d}_b y^i = g^{ii} + (g^{0i})^2 - (\tilde{T}^i)^2$ together with \eqref{bootstrap-assumptions} and \eqref{eq:LT}, we have $|\slashed{d}y^i| \lesssim 1.$

 On the other hand, it follows from \eqref{LeL} and \eqref{bootstrap-assumptions} that $|\mathring{L}(\varrho\check{L}^i)|\lesssim M\delta^{2\nu-3}$. Then integrating $\mathring{L}(\varrho\check{L}^i)$ along the integral curves of $\mathring{L}$ and using the relation $\check{T}^i = -g^{0i} - \check{L}^i$ yield
 \[
 |\check{L}^i| + |\check{T}^i| \lesssim M\delta^{\nu-1}.
 \]

 Note that $g_{ij}\big(\check{T}^i - \frac{y^i}{\varrho}\big)\big(\check{T}^j - \frac{y^j}{\varrho}\big) = 1$. We can then get
 \[
 \left(g_{ij}\omega^i\omega^j\right)\frac{r^2}{\varrho^2} - \left(2g_{ij}\check{T}^i\omega^j\right)\frac{r}{\varrho} + g_{ij}\check{T}^i\check{T}^j - 1 = 0.
 \]
 Thus, it holds that
 \begin{equation}\label{rrho}
 \check{\varrho} = \frac{1 - g_{ij}\omega^i\omega^j - g_{ij}\check{T}^i\check{T}^j + 2g_{ij}\check{T}^i\omega^j}{\sqrt{g_{ij}\omega^i\omega^j - (g_{ij}\omega^i\omega^j)(g_{ab}\check{T}^a\check{T}^b) + (g_{ij}\check{T}^i\omega^j)^2} + g_{ij}\omega^i\omega^j - g_{ij}\check{T}^i\omega^j}.
 \end{equation}
 This implies
 \[
 |\check{\varrho}| \lesssim M\delta^{\nu-1},
 \]
 and hence, by the definition of $\upsilon$ in \eqref{omega}, one arrives at
 \[
 |\upsilon| \lesssim M\delta^{\nu-1}.
 \]
Due to \eqref{lmu} and \eqref{Lchi'}, we have
 \begin{equation}\label{YHCCC-X1}
\begin{split}
 &|\mathring{L}\mu| \lesssim M\delta^{\nu-2 } + (M\delta^{2\nu-3})\mu, \\
 &|\mathring{L}(\varrho^2\operatorname{tr}\check{\mathscr{X}})| \lesssim M\delta^{2\nu-3}
 +(\varrho^2\operatorname{tr}\check{\mathscr{X}})^2 + M\delta^{2\nu-3}(\varrho^2\operatorname{tr}\check{\mathscr{X}}),
 \end{split}
\end{equation}
 which lead to
 \[
 |\mu| \lesssim M \quad \text{and} \quad |\operatorname{tr}\check{\mathscr{X}}| \lesssim M\delta^{\nu-1}.
 \]
\end{proof}

Note that under the assumptions \eqref{bootstrap-assumptions} and the estimate \eqref{eq:lower-order-Linfty-estimates},
the operator $R$ is merely the scaling operator $\slashed{\nabla}$. Based on this, we have

\begin{corollary}\label{cor:12form}
Under the assumptions \eqref{bootstrap-assumptions}, for small $\delta>0$, it holds that
 \begin{enumerate}
  \item if $\eta=\eta_idy^i$ is a $1$-form on $S_{t, u}$, then
  \begin{align}
  (\eta_a R^a)^2 &\sim |\eta|^2, \label{eq:1-f}\\
  |\slashed{\mathcal L}_{R}\eta|^2 &\sim |\slashed{\nabla}\eta|^2 + O(M\delta^{\nu-1})|\eta|^2; \label{eq:1f}
  \end{align}

  \item if $\eta$ is a $2$-form on $S_{t, u}$, then
  \begin{equation}\label{eq:2-f}
  |\slashed{\mathcal L}_{R}\eta|^2 \sim |\slashed{\nabla}\eta|^2 + O(M\delta^{\nu-1})|\eta|^2
  \end{equation}
  and
  \begin{equation}\label{eq:2f}
  |\slashed{\nabla}^2\eta \lesssim |\slashed{\mathcal L}_R^{\leq 2}\eta|.
  \end{equation}
 \end{enumerate}
\end{corollary}
\begin{proof}
Although the proof can follow the framework of \cite[Lemma 12.22]{Sp} for the 3D quasilinear wave equations with small
initial data, we still provide the details for reader's convenience.

 For a 1-form $\eta$ on $S_{t, u}$, \eqref{eq:1-f} can be derived from the identity
 \[
 (\eta_a R^a)^2 = r^2\sum_{i = 1}^2 (\eta_i)^2 - (\varrho\eta_i\check{T}^i)^2.
 \]

 Note that
 \begin{equation}\label{YHCC6}
 \begin{split}
 (\slashed{\nabla}_X R)_X &= \epsilon^a_i g_{ab}(\slashed{d}y^i)(\slashed{d}y^b)
 + \upsilon\big(\check{\mathscr{X}}_{XX} + \frac{\slashed{g}_{XX}}{\varrho}
 + G_{X\mathring{L}}^\al \slashed{d}_X \vp_\al + G_{X\tilde{T}}^\al \slashed{d}_X \vp_\al\big)
 + \frac{1}{2}G_{XX}^\al R \vp_{\al} \\
 &= O(M\delta^{\nu-1}).
 \end{split}
 \end{equation}

 Set $\check{g}^{ij} = g^{ij} - m^{ij}$, one then has
 \begin{equation}\label{YHCC7}
 \begin{split}
 |\slashed{\nabla}_R \eta|^2 = r^2\slashed{g}^{ab}(\slashed{\nabla}_i \eta_a)(\slashed{\nabla}_j \eta_b)
 \big(\slashed{g}^{ij} - \check{g}^{ij} + \big(\frac{1}{1+\check{\varrho}}\big)^2 \check{T}^i \check{T}^j\big)
 = \bigl(1 + O(M\delta^{\nu-1})\bigr)r^2|\slashed{\nabla} \eta|^2.
 \end{split}
 \end{equation}

 Combining \eqref{YHCC6} with \eqref{YHCC7} yields
 \begin{equation*}
 \begin{split}
 |\slashed{\mathcal{L}}_{R} \eta|^2 &= |\slashed{\nabla}_R \eta|^2
 + \left(\eta^X \slashed{\nabla}_X R^X\right)\left(\eta_X \slashed{\nabla}_X R^X\right)
 + 2\left(\slashed{\nabla}_R \eta\right)^X \eta^X \left(\slashed{\nabla}_X R\right)_X \\
 &= \bigl(1 + O(M\delta^{\nu-1})\bigr)r^2|\slashed{\nabla} \eta|^2 + O(M\delta^{\nu-1})|\eta|^2,
 \end{split}
 \end{equation*}
 which implies \eqref{eq:1f}.

\eqref{eq:2-f} and \eqref{eq:2f} can be shown by analogous arguments.
\end{proof}

As the first step toward closing the bootstrap assumptions \eqref{bootstrap-assumptions},
we now establish the corresponding estimates for some related quantities.

\begin{proposition}\label{prop:higher-order-L-infty-estimates}
 Under the assumptions \eqref{bootstrap-assumptions}, for any vector field $Z \in \{\mathring{L}, T, R\}$, when $\de > 0$ is small, the following estimates hold that for $m\leq N-2$,
 \begin{equation}\label{eq:higher-order-L-infty-estimates}
 \begin{split}
 &\de^{l+k(2-\nu)} \big( |\Lies_{Z}^{\leq m} \check{\mathscr{X}}|+ |\Lies_{Z}^{\leq m} {}^{(R)}\pis|\big) \les M\de^{\nu-1}, \\
 &\de^{l+k(2-\nu)} |\Lies_{Z}^{\leq m} {}^{(R)}\pis_L|\les M\de^{\nu-1}(1 + M^2\delta^{3\nu-4}), \\
 &\de^{l+k(2-\nu)} |\Lies_{Z}^{\leq m} {}^{(R)}\pis_T|\lesssim M^2\de^{\nu-1}(1 + M\delta^{2\nu-3}), \\
 &\de^{l+k(2-\nu)} |Z^{\leq m+1} \mu| \les M, \qquad\de^{l+k(2-\nu)} |\Lies_{Z}^{\leq m+1} \ds y^j| \les 1, \\
 &\de^{l+k(2-\nu)} |\Lies_Z^{\leq m} {}^{(T)}\pis|\les M\de^{\nu-2}, \qquad\de^{l+k(2-\nu)} |\Lies_Z^{\leq m} {}^{(T)}\pis_L|\lesssim M(1 + M\delta^{2\nu-3}), \\
 &\de^{l+k(2-\nu)} \big(|Z^{\leq m+1} \Lc^j| + |Z^{\leq m+1} \Tc^j|+ |Z^{\leq m+1} \upsilon| +|Z^{\leq m+1} \check\varrho| + |\Lies_{Z}^{\leq m+1} R| \big) \les M\de^{\nu-1},
 \end{split}
 \end{equation}
 where $l$ (resp.\ $k$) is the number of $T$ (resp.\ $\mathring{L}$) appearing in the string of $Z$.
\end{proposition}

\begin{proof}
 \textbf{Step 1.} The case of $Z = R$
 \vskip 0.1 true cm

When $m = 0$, the estimate $|\check{\mathscr{X}}| \lesssim M\delta^{\nu-1}$ follows from \eqref{eq:lower-order-Linfty-estimates}
directly.

According to the expression of $g_{\alpha\beta}$ given in \eqref{metric}, for any $\Gamma \in \{\mathring{L}, T\}$, one has
\begin{equation}\label{GG}
\begin{split}
G_{X\mathring{L}}^\al \Gamma \vp_\al &= \mathcal{G} \mathcal{A}^{-2} \vp_i (\slashed{d}_X y^i) \Gamma \vp_0
+ \mathcal{G} \mathcal{A}^{-2} (\vp_0 + q_0) \Gamma^\alpha \slashed{d}_X \vp_\alpha \\
&\quad + \mathcal{G} (\gamma+ 1)(\vp_0 + q_0)
\big(\check{g}_{\alpha\beta} \slashed{d}_X y^\alpha \mathring{L}^\beta\big)\Gamma \vp_0 \\
&\quad + \mathcal{G} \big(\mathcal{A}^{-2} \vp_i \delta_{jl} + \mathcal{A}^{-2} \vp_j \delta_{il}
- \delta_{ij} \vp_l (\gamma-1+2\mathcal{A}^{-2})\big)
\slashed{d}_X y^i \mathring{L}^j \Gamma \vp_l.
\end{split}
\end{equation}
Using \eqref{GG}, \eqref{bootstrap-assumptions}, Lemma \ref{lem3.4}, and Proposition \ref{prop:lower-order-Linfty-estimates}, we can
obtain the corresponding estimates for ${}^{(R)}\pis$, ${}^{(R)}\pis_L$, ${}^{(R)}\pis_T$, ${}^{(T)}\pis$, and ${}^{(T)}\pis_L$
in \eqref{eq:higher-order-L-infty-estimates}.

Additionally, from \eqref{eq:1-f} and \eqref{R}, one has $|R\check{L}^i| \sim |\slashed{d}\check{L}^i|$.
Then the estimate of $R\check{L}^i$ follows from \eqref{deL}. Analogously, $R\check{T}^i$, $R\upsilon$ and $R\check{\varrho}$
can be treated. Hence, $\Lies_{R} \slashed{d} y^j = \slashed{d}(\Omega^i - g_{aj}\tilde{T}^a\Omega^j\tilde{T}^i)$ is also estimated.

For $R\mu$, one has that by \eqref{lmu},
\begin{equation*}
\mathring{L}(R\mu) = {}^{(R)}\slashed{\pi}_L^X \slashed{d}_X\mu + R\big(-\frac{1}{2}\mu G_{\mathring{L}\mathring{L}}^\al \mathring{L}\varphi_\al - \mu G_{\tilde{T}\mathring{L}}^\al \mathring{L}\varphi_\al + \frac{1}{2}G_{\mathring{L}\mathring{L}}^\al T\varphi_\al\big).
\end{equation*}
This implies $|\mathring{L}R\mu| \lesssim M\delta^{2\nu-3}|R\mu| + M\delta^{\nu-2}$. Thus,
integrating along the integral curves of $\mathring{L}$ yields $|R\mu| \lesssim M$.

Therefore, \eqref{eq:higher-order-L-infty-estimates} holds for $m=0$.

For general $m$, we can use the induction argument to prove \eqref{eq:higher-order-L-infty-estimates}. Assume that \eqref{eq:higher-order-L-infty-estimates} holds for all vector fields $R$ and $m=p \leq N-3$, then it follows from \eqref{Lchi'} that
\begin{equation*}
\begin{split}
|\mathring{L}R^{p+1}(\varrho^2\operatorname{tr}\check{\mathscr{X}})| &\lesssim |[\mathring{L}, R^{p+1}](\varrho^2\operatorname{tr}\check{\mathscr{X}})| + |R^{p+1}\mathring{L}(\varrho^2\operatorname{tr}\check{\mathscr{X}})| \\
&\lesssim M\delta^{2\nu-3} + M\delta^{2\nu-3}|R^{p+1}(\varrho^2\operatorname{tr}\check{\mathscr{X}})|,
\end{split}
\end{equation*}
which yields $|R^{p+1}(\operatorname{tr}\check{\mathscr{X}})| \lesssim M\delta^{\nu-1}$. From this, the estimates of $\Lies_{R}^{p+1} {}^{(R)}\pis$, $\Lies_{R}^{p+1} {}^{(R)}\pis_{\mathring{L}}$, $\Lies_{R}^{p+1} {}^{(T)}\pis$, $R^{p+2}\check{L}^i$, $R^{p+2}\check{T}^i$, $R^{p+2}\check{\varrho}$, $R^{p+2}\upsilon$ and $\Lies_{R}^{p+2}\slashed{d}y^i$ can be obtained with the help of Lemma \ref{lem3.4}, \eqref{deL}, and the definitions of $\check{T}^i$, $\check{\varrho}$ and $\upsilon$.

Analogously,
\[
|\mathring{L}R^{p+2}\mu| \lesssim M\delta^{2\nu-3}|R^{p+2}\mu| + M\delta^{\nu-2}
\]
holds by \eqref{lmu}, which can be applied to estimate $R^{p+2}\mu$, $\Lies_{R}^{p+1} {}^{(R)}\pis_T$ and $\Lies_{R}^{p+1} {}^{(T)}\pis_{\mathring{L}}$
in \eqref{eq:higher-order-L-infty-estimates}.

\textbf{Step 2.} The general cases with $Z \in \{\mathring{L}, T, R\}$
\vskip 0.1 true cm

At this time,
we may use \eqref{Lchi'}, \eqref{Tchi'}, \eqref{TL}, \eqref{LeL}, \eqref{lmu} and the estimates in Step 1 to prove
the corresponding results. One can be also referred to see the analogous details in \cite[Propositions 6.2 and 6.3]{Ding3},
where the properties of long time decay rates are simultaneously derived.
\end{proof}
From the identity $T=\f{\partial}{\p u}-\Theta^XX$, we can estimate $\Theta^X$.
\begin{lemma}\label{eXi}
 Under the assumptions \eqref{bootstrap-assumptions}, for small $\delta>0$, it holds that
 \begin{equation}\label{Xi}
 |\Theta|\lesssim M^2\delta^{\nu-1}+M\delta^{2-\nu}.
 \end{equation}
\end{lemma}
\begin{proof}
 Note that $\Theta$ is a vector field on $S_{t, u}$ and one has
 \begin{equation}\label{LXi}
 \mathring L\Theta^X=[T,\mathring L]^X=\big(\mathscr{D}_T\mathring L-\mathscr{D}_{\mathring L}T\big)^X=\slashed d^X\mu+2\mu\zeta^X.
 \end{equation}
 Then
 \[
 \mathring L(\varrho^{-2}|\Theta|^2)=2\varrho^{-2}|\Theta|^2(\operatorname{tr}\check{\mathscr X}+\f12\slashed g^{XX}G_{XX}^\gamma\mathring L\vp_\gamma)+2\varrho^{-2}(\slashed d_X\mu
 +2\mu\zeta_X)\Theta^X.
 \]
 This, together with the estimates \eqref{eq:higher-order-L-infty-estimates}, \eqref{bootstrap-assumptions} and \eqref{zeta}, yields
 \begin{equation}\label{Y1}
 |\mathring L(\varrho^{-1}|\Theta|)|\lesssim M\delta^{2\nu-3} \varrho^{-1}|\Theta| + M^2\delta^{2\nu-3}+M.
 \end{equation}
 Thus, \eqref{Xi} follows immediately from the integration on \eqref{Y1} along integral curves of $\mathring L$.
\end{proof}

Next, we improve the bootstrap assumptions \eqref{bootstrap-assumptions} up to the $(N-2)^{\text{th}}$-order derivatives.

\begin{proposition}\label{prop:bootstrap-higher-derivatives}
 For small $\de > 0$, it holds that for all $m \leq N-2$,
 \begin{equation}\label{eq:bootstrap-higher-derivatives}
 \de^{l+k(2-\nu)}| Z^{m} \vp_{\al} |\les \de^{\nu-1}, \quad \al = 0,1,2,
 \end{equation}
 where $Z \in \{ \mathring{L}, T, R \}$, $l$ denotes the number of $T$ operator, and $k$ the number of $\mathring{L}$ operator in the derivative $Z^{m}$. In particular,
 \begin{equation}\label{Tvp}
 |\varrho^{1/2}T\vp_\al(t,u,\vartheta)-\varrho_0^{1/2}T\vp_\al(t_0,u,\vartheta)|\lesssim M^2\delta^{2\nu-3},
 \end{equation}
 where $\varrho_0=t_0-u$.
\end{proposition}

\begin{proof}
It follows from \eqref{fequation} and Proposition \ref{prop:higher-order-L-infty-estimates} that for all $m \leq N-2$,
\begin{equation*}
\begin{split}
\big| \mathring{L}{Z}^m \underline{L} \vp_\al + \frac{1}{2\varrho}{Z}^m \underline{L} \vp_\al \big|
&= \big|{Z}^m (\mu\slashed\triangle\vp_\al+H_\al )+ [\mathring{L}, {Z}^m] \underline{L} \vp_\al
- \sum_{\substack{m_1 + m_2 =m \\ m_1 > 0}}{Z}^{m_1}(\frac{1}{2\varrho}){Z}^{m_2} \underline{L} \vp_\al\big| \\
&\lesssim (M+M^2\delta^{3\nu-5})\delta^{-l-k(2-\nu)}.
\end{split}
\end{equation*}
In particular, $\big| \mathring{L}\underline{L} \vp_\al + \frac{1}{2\varrho}\underline{L} \vp_\al \big|\lesssim M^2\delta^{3\nu-5}.$
Hence,
$\big| \mathring{L} \left( \varrho^{1/2}{Z}^m \underline{L} \vp_\al \right)\big|
\lesssim (M+M^2\delta^{3\nu-5}) \delta^{-l-k(2-\nu)}$ and $\big|\mathring L\left( \varrho^{1/2}\underline{L} \vp_\al \right)\big|\lesssim M^2\delta^{3\nu-5}$ hold,
which implies
\begin{equation*}
\begin{split}
&\big| \varrho^{1/2} {Z}^m \underline{L} \vp_\al(t,u,\vartheta)-\varrho_0^{1/2}{Z}^m \underline{L} \vp_\al(t_0,u,\vartheta)\big|
\lesssim (\delta^{2-\nu}M+M^2\delta^{2\nu-3}) \delta^{-l-k(2-\nu)},\\
&\big| \varrho^{1/2}\underline{L} \vp_\al(t,u,\vartheta)-\varrho_0^{1/2}\underline{L} \vp_\al(t_0,u,\vartheta)\big|
\lesssim M^2\delta^{2\nu-3}.
\end{split}
\end{equation*}
Then \eqref{Tvp} is derived and $|{Z}^m T \vp_\al| \lesssim \delta^{\nu-2-l-k(2-\nu)}$ holds. Furthermore, one has
$|T Z^m\vp_\al|\lesssim \delta^{\nu-2-l-k(2-\nu)}$.
Therefore, together with \eqref{Xi} and \eqref{bootstrap-assumptions}, we have
\begin{equation}\label{YHCC8}
\big| \frac{\partial}{\partial u}{Z}^m \vp_\al\big|
\lesssim \big|T{Z}^m \vp_\al\big| + \big| \Theta^X \slashed{d}_X {Z}^m \vp_\al \big|
\lesssim \delta^{\nu-2-l-k(2-\nu)},
\end{equation}
which yields $|{Z}^m \vp_\al| \lesssim \delta^{\nu-1-l-k(2-\nu)}$ by integrating \eqref{YHCC8} from $u_0$ to $u$.
\end{proof}

It is observed that under the assumptions \eqref{bootstrap-assumptions},
the estimates in Proposition \ref{prop:higher-order-L-infty-estimates} are obtained with
the corresponding bounds depending on $M$.
In contrast, the improved estimates established in \eqref{eq:bootstrap-higher-derivatives} of Proposition \ref{prop:bootstrap-higher-derivatives} yield the bounds independent of $M$.
Consequently, starting from Proposition \ref{prop:bootstrap-higher-derivatives}
and repeating the relevant arguments as in the proof of Proposition \ref{prop:higher-order-L-infty-estimates},
we can improve the estimate \eqref{eq:higher-order-L-infty-estimates} for $m \leq N-4$ ($N$ can be chosen sufficiently large)
so that the constants in those $L^\infty$ norms no longer depend on $M$. In addition,
the higher order derivative estimates in \eqref{eq:bootstrap-higher-derivatives} for $N-1\leq m\leq N$
will be postponed to Sections \ref{EE}--\ref{ert}. Accordingly, the closure of the bootstrap arguments \eqref{bootstrap-assumptions} will
be completed therein.

\section{Detailed analysis of $\mu$ and shock formation}\label{Sectionmu}
In this section, we aim to prove that $\mu \to 0^+$ as $t \to t^*$ with $t^*$ being the maximal
existence distance of smooth solution to problem \eqref{mainequation} ($t^*\le t_*$). Once $\mu$ approaches $0$, the coordinate transformation
from $(t, y^1, y^2)$ to $(t, u, \vartheta)$ is no longer regular. In this case, as shown in \cite{Ch1, Sp, M-Y},
the second order derivatives of $\phi$ will blow up and the shock is then formed.

Introduce a function $\mathscr{M}$ as follows
\[
\mathscr{M}(t,u,\vartheta) := \frac{(\gamma+1) q_{0}^{3}} {2c^{2}(\rho_{0})\big(q_{0}^{2}-c^{2}(\rho_{0}) \big)} \,\varrho^{1/2} T\varphi_{0}(t,u,\vartheta).
\]
Then one has
\begin{proposition}\label{Prop.5.1}
 For small $\delta > 0$, it holds that
 \begin{equation}\label{mathM}
 \mu(t,u,\vartheta) = 1 + 2\mathscr{M}(t_0,u,\vartheta)(\varrho^{1/2} - \varrho_0^{1/2}) + O(\delta^{\nu-1}).
 \end{equation}
 Furthermore, when $\mu < \frac{1}{10}$,
 \begin{equation}\label{blowup}
 \mathring{L}\mu=\varrho^{-1/2}\mathscr M(t,u,\vartheta)+O(\delta^{2\nu-3}) \sim -\delta^{\nu-2}.
 \end{equation}
\end{proposition}
\begin{proof}
 According to the expressions of $g_{\alpha\beta}$ in \eqref{metric}, one has
 \begin{equation*}
 \begin{split}
 G_{\mathring{L}\mathring{L}}^\al T\varphi_\al
 &= \mathcal{G}(\varphi_0 + q_0)
 \big(
 -\mathcal{A}^{-2}(\gamma- 1)T\varphi_0
 -\delta_{ij}\mathring{L}^i\mathring{L}^j(\gamma+1)T\varphi_0
 + 2\mathcal{A}^{-2}\mathring{L}^iT\varphi_i
 \big) \\
 &= -\mathcal{G}(\mathcal{A}^{-2} + 1)(\gamma+ 1)(\varphi_0 + q_0)T\varphi_0
 + O(\delta^{2\nu - 3}) \\
 &= \frac{q_0^3(\gamma+ 1)}{c^2(\rho_0)\big(q_{0}^{2}-c^{2}(\rho_{0}) \big)}T\varphi_0
 + O(\delta^{2\nu - 3}),
 \end{split}
 \end{equation*}
 where we have used the facts of $g_{\alpha\beta}\mathring{L}^\alpha\mathring{L}^\beta
 = \check{g}_{\alpha\beta}\mathring{L}^\alpha\mathring{L}^\beta
 + \delta_{ij}\mathring{L}^i\mathring{L}^j-1=0$, $\mathring L^iT\vp_i=\mu\tilde T^i\mathring L\vp_i-T\vp_0$
 and $\mathcal{G} = \frac{1}{c^2(\rho_0) - q_0^2} + O(\delta^{\nu - 1})$.
 Hence, it follows from \eqref{lmu} and \eqref{Tvp} that
 \begin{equation}\label{rLmu}
 \varrho^{1/2}\mathring{L}\mu=\mathscr{M}(t, u, \vartheta) + O(\delta^{2\nu - 3})
 =\mathscr{M}(t_0, u, \vartheta) + O(\delta^{2\nu-3}).
 \end{equation}
 Due to
 \begin{equation*}
 \begin{split}
 \mu(t,u,\vartheta)
 &= \mu(t_0,u,\vartheta)
 + \int_{t_0}^t \frac{\varrho^{1/2}\mathring{L}\mu(s,u,\vartheta)}{\varrho^{1/2}}ds \\
 &= \mu(t_0,u,\vartheta)
 + 2\mathscr{M}(t_0,u,\vartheta)\big(\varrho^{1/2} - \varrho_0^{1/2}\big)
 + O(\delta^{\nu - 1}),
 \end{split}
 \end{equation*}
 then \eqref{mathM} is proved.

If $\mu(t,u,\varrho) < \frac{1}{10}$, then it holds that by \eqref{mathM},
\[
\mathscr{M}(t_0,u,\vartheta) \leq \frac{-9}{20\big(\varrho^{1/2} - \varrho_0^{1/2}\big)}.
\]
Therefore, using \eqref{rLmu} again, we obtain
\begin{equation*}
\begin{split}
\varrho^{1/2}\mathring{L}\mu
&\leq \frac{-9}{10\int_{t_0}^t (s-u)^{-1/2}ds} + O\big(\delta^{2\nu-3}\big) \\
&= \frac{-9(\tau-u)^{1/2}}{10(t-t_0)} + O\big(\delta^{2\nu-3}\big)
\lesssim -\delta^{\nu-2},
\end{split}
\end{equation*}
where $\tau \in (t_0,t)$ and $\tau - u \sim \frac{1}{\mathcal{A}}$. In addition, it follows from \eqref{eq:higher-order-L-infty-estimates} that $|\mathring{L}\mu| \lesssim \delta^{\nu-2}$. Hence, \eqref{blowup} holds.
\end{proof}

Note that when $\mu<\f{1}{10}$, we have
\[
T\vp_0 \sim -\delta^{\nu-2} \frac{2c^2(\rho_0) \big(q_{0}^{2}-c^{2}(\rho_{0}) \big)}{(\gamma+ 1)q_0^3} \sim -\delta^{\nu-2}
\]
by \eqref{blowup} and the fact of $q_0 > c(\rho_0)$. Therefore,
$\mu \, |\nabla_y \partial_t \phi| \gtrsim |T\vp_0| \sim \delta^{\nu-2}$,
which yields
\begin{equation}\label{shock-0}
|\partial^2\phi| \sim \mu^{-1}\delta^{\nu-2},
\end{equation}
and hence
\begin{equation}\label{shock}
|\partial^2\phi| \rightarrow +\infty \quad \text{as} \quad \mu \rightarrow +0.
\end{equation}

As established in Proposition \ref{Prop.5.1}, if $\mu< \frac{1}{10}$ at some distance, then $\mu$ will converge to $0$
within a finite distance due to $\mathring{L}\mu \sim -\delta^{\nu-2}$. Consequently, $\partial^2\phi$ will blow up in sense of \eqref{shock}.
In order to characterize the behavior of $\mu$ more precisely near the blowup distance, as in \cite{Sp, M-Y}, one can define
$$\mu_{\min}(t') = \min_{(u,\vartheta)} \mu(t',u,\vartheta).$$

\begin{proposition}
Suppose $\mu(t, u, \vartheta) < \frac{1}{10}$ and let $\delta > 0$ be small. Then it holds that
\begin{equation}\label{mumin}
\mu_{\min}(t)\gtrsim\de^{\nu-2}(t^*-t).
\end{equation}
Furthermore,
\begin{equation}\label{Tmu}
\bigl( \mu^{-1} T\mu \bigr)_{+}(t, u, \vartheta) \lesssim \delta^{-\nu/2} (t^* - t)^{-1/2},
\end{equation}
where $\bigl( \mu^{-1} T\mu \bigr)_{+}$ denotes the positive part of $\mu^{-1} T\mu$.
\end{proposition}
\begin{proof}
 At first, we take the following regular transformation from $(u,\vartheta)$ to $(\tilde{u}, \tilde{\vartheta})$
 \[
 \begin{array}{l}
 \begin{cases}
 \tilde{\vartheta} = \vartheta + g(u, \vartheta), \\
 \tilde{u} = u
 \end{cases}
 \quad \text{with} \quad
 \begin{cases}
 \frac{\partial g}{\partial u} - \Theta^X \frac{\partial g}{\partial \vartheta} = \Theta^X, \\
 g\left(u_0, \vartheta\right) = 0.
 \end{cases}
 \end{array}
 \]
 Then the vector field $T=\f{\p}{\p\tilde u}$ holds.

 For fixed $(t,\tilde\vartheta)$, set $f(\tilde u)=(\mu^{-1}T\mu)(\tilde u)$, where $\tilde u\in [u_0,U_0]$. Assume that $u^*\in [u_0,U_0]$ satisfies
 \[
 f(u^*)=\max_{\tilde u\in[u_0,U_0]}f(\tilde u)>0.
 \]

 If $u^* = u_0$, then $(t, u^*, \tilde{\vartheta})$ is a point on the outermost characteristic surface, which deduces $\mu(u^*) = 1$ and $T\mu(u^*) = 0$. This contradicts the assumption $f(u^*) > 0$, and hence $u^* \in(u_0, U_0]$, which yields $f'(u^*) \geq 0$. Therefore, together with the fact $|T^2\mu| \lesssim \delta^{-2}$ by \eqref{eq:higher-order-L-infty-estimates}, one has
 \begin{equation}\label{muT}
 \bigl(\mu^{-1}T\mu\bigr)(u^*) \leq \sqrt{\dfrac{T^2\mu}{\mu}(u^*)} \lesssim \delta^{-1} \big( \min_{\bar u \in [u_0, U_0], \bar\vartheta \in \mathbb{S}^1} \mu(t, \bar u, \bar\vartheta) \big)^{-1/2}.
 \end{equation}
 Furthermore, if $\mu(t, u, \vartheta) < \frac{1}{10}$, then there exists $(u_0, \vartheta_0)$ such that
 \[
 \mu(t, u_0, \vartheta_0) = \mu_{\min}(t)< \frac{1}{10},
 \]
 which implies that for any $\bar{t} \in [t, t^*]$,
 \[
 \mathring{L}\mu(\bar{t}, u_0, \vartheta_0) \sim -\delta^{\nu-2}
 \]
 with the help of \eqref{blowup}. Hence,
 \begin{equation}\label{mu}
 \mu(t, u_0, \vartheta_0) \geq -\int_t^{t^*} \mathring{L}\mu(\bar{t}, u_0, \vartheta_0) \, d\bar{t} \sim \delta^{\nu-2}(t^* - t).
 \end{equation}
 Consequently, \eqref{Tmu} is proved after collecting \eqref{muT} and \eqref{mu}.
\end{proof}
\begin{lemma}\label{lem:higher_order_L2_estimates_mu}
Let $\underline{t} \triangleq \inf_{\tau}\left\{ t_0 \leq \tau < t^* : \mu_{\min}(\tau) < \frac{1}{10} \right\}$. For $a > 1$ and
small $\delta > 0$, when $t > \bar{t}$, one has
\begin{equation}\label{eq:mu_integrability}
\int_{\underline{t}}^{t} \mu_{\min}^{-a}(\tau) \,\mathrm{d}\tau \lesssim \frac{\delta^{2-\nu}}{a-1} \mu_{\min}^{1-a}(t).
\end{equation}
In addition, for any $t, t' \in [t_0, t^*]$ with $t < t'$, it holds
\begin{equation}\label{eq:mu_monotonicity}
\mu_{\min}(t') \lesssim \mu_{\min}(t).
\end{equation}
\end{lemma}
\begin{proof}
Suppose that $\mu(t^*,u^*,\theta^*) = 0$ for some $(u^*,\theta^*) \in [u_0,U_0] \times \mathbb{S}$. It follows from \eqref{eq:higher-order-L-infty-estimates} that
\begin{equation*}
\mu(t,u^*,\theta^*)
= \int_{t^*}^t \mathring{L}\mu(\tau,u^*,\theta^*) \, d\tau
\lesssim \delta^{\nu-2}(t^* - t),
\end{equation*}
and thus $\mu_{\min}(t) \leq \mu(t,u^*,\theta^*) \lesssim \delta^{\nu-2}(t^* - t)$. For $t \in (\underline{t},t^*)$, by \eqref{mumin}
we arrive at
\[
\mu_{\min}(t) \sim \delta^{\nu-2}(t^* - t).
\]
Therefore,
\begin{equation*}
\begin{split}
\int_{\underline{t}}^t \mu_{\min}^{-a}(\tau) \, d\tau
&\sim \int_{\underline{t}}^t \delta^{a(2-\nu)}(t^* - \tau)^{-a} \, d\tau \\
&\lesssim \frac{\delta^{a(2-\nu)}}{a-1} (t^* - t)^{1-a}
\sim \frac{\delta^{2-\nu}}{a-1} \mu_{\min}^{1-a}(t),
\end{split}
\end{equation*}
which proves \eqref{eq:mu_integrability}.

For any $t, t' \in (\underline{t}, t^*]$ with $t < t'$, one has $\mu_{\min}(t) \sim \delta^{\nu-2}(t^* - t)$ and $\mu_{\min}(t') \sim \delta^{\nu-2}(t^* - t')$, which yields $\mu_{\min}(t') \lesssim \mu_{\min}(t)$.

For $t_0 \leq t \leq \underline{t} < t' \leq t^*$, then $\mu_{\min}(t') < \frac{1}{10} \leq \mu_{\min}(t)$.

Finally, for $t_0 \leq t < t' \leq \underline{t}$, we have $\mu_{\min}(t) \sim 1$ and $\mu_{\min}(t') \sim 1$. This completes the proof of \eqref{eq:mu_monotonicity}.
\end{proof}

Until now, we have obtained a series of results under the assumption that ``$\mu < \frac{1}{10}$''. A natural question arises: does
there exist some point $(t, u, \vartheta)$ such that $\mu < \frac{1}{10}$ holds? To address this issue, we revisit \eqref{mathM}.
From \eqref{Tvp}, one has $\varrho^{1/2}T\vp_0(t,u,\vartheta) = \varrho_0^{1/2}T\vp_0(t_0,u,\vartheta) + O(\delta^{2\nu-3})$.
Combining this with $\varrho_0^{1/2} = \sqrt{\frac{1}{\mathcal A}} + O(\delta)$ and $\varrho^{1/2}
= \sqrt{\frac{1}{\mathcal A}} + \frac{2}{\sqrt{\mathcal A}} t + O(t^2) + O(\delta)$, we have
\begin{equation*}
\begin{split}
\mu(t,u,\vartheta)&= 1+\frac{(\gamma+1) q_0^3}{2c^2(\rho_0)\bigl(q_0^2 - c^2(\rho_0)\bigr)} T\varphi_0(t_0,u,\vartheta)\bigl(t + O(t^2)\bigr) + O(\delta^{\nu-1}) \\
&= 1 - \frac{(\gamma+1) q_0^3}{2c^2(\rho_0)\bigl(q_0^2 - c^2(\rho_0)\bigr)} \partial_r\varphi_0(t_0,u,\vartheta)\bigl(t + O(t^2)\bigr) + O(\delta^{\nu-1})\\
&=1+\frac{(\gamma+1) q_0^3}{4c^2(\rho_0)\big(q_0^2-c^2(\rho_0)\big)}\underline{L}\varphi_0(t_0,u,\vartheta)\bigl(t+O(t^2)\bigr)+O(\delta^{\nu-1}).
\end{split}\end{equation*}
By \eqref{y0}, it holds that
\begin{equation*}
\begin{split}
&\underline{L}\varphi_0(t_0,u,\vartheta) \\
=&\delta^{\nu-2}\mathcal{A} \bigl(\mathcal{A}\partial_s^2\Phi_0\big( \frac{-\mathcal{A} u - 1}{\delta},\cos\vartheta,\sin\vartheta \big)-\partial_s\Phi_1\big( \frac{-\mathcal{A} u - 1}{\delta},\cos\vartheta,\sin\vartheta\big)\bigr)+O(\delta^{2\nu-3}).
\end{split}
\end{equation*}
Choosing $(u_0, \vartheta_0)$ such that $\left( \frac{-\mathcal A u_0 - 1}{\delta}, \cos\vartheta_0, \sin\vartheta_0 \right) = (s_0, \omega_0)$, we then deduce
\begin{equation}\label{YHCC-2}
\begin{split}
&\mu(t,u_0,\vartheta_0)\\
=& 1 + \frac{(\gamma+1) q_0^3\mathcal{A}}{4c^2(\rho_0)\bigl(q_0^2 - c^2(\rho_0)\bigr)} \delta^{\nu-2} \bigl(\mathcal{A}\partial_s^2\Phi_0(s_0,\omega_0)-\partial_s\Phi_1(s_0, \omega_0)\bigr)\bigl(t + O(t^2)\bigr) + O(\delta^{\nu-1}).
\end{split}
\end{equation}
Under condition~\eqref{condition1},
$\mu<\tfrac{1}{10}$ is derived for some $t$. Consequently, $\mu\to 0^{+}$ occurs before $t_{*}$.

\section{Energy estimates and higher order $L^2$ estimates}\label{EE}
As stated at the end of Section \ref{Section 3}, the bootstrap assumptions given in \eqref{eq:bootstrap-higher-derivatives} have been improved up to the index $N-2$. However, the closure of estimates for higher order derivatives remains unresolved. To this end, motivated by the techniques in \cite{Ding2,Ding3,Ding4,L-Y,M-Y,Sp}, we first establish the energy estimates for $\vp_\gamma$ up to its $(2N)^{\text{th}}$-order derivatives and then apply the embedding theorem to achieve the closure for all $L^\infty$ estimates in \eqref{bootstrap-assumptions}.
This means that the existence of smooth solution $\phi$ to \eqref{mainequation} before $t_*$ is shown.
To study \eqref{ge}, we now derive the energy estimates for the associated linear equation
\begin{equation}\label{gel}
\mu\Box_g\Psi \;=\; \Phi,
\end{equation}
where $\Phi$ is a given smooth function, and $\Psi$ together with its derivatives vanish on $u =u_0$.
Following \cite{L-Y,M-Y}, we introduce two multipliers
\[
V_1\Psi := \mathring{L}\Psi,\quad V_2\Psi := \mathring{\underline{L}}\Psi.
\]
For $i=1,2,$ the corresponding energies $E_i[\Psi](t, u)$ and fluxes $F_i[\Psi](t, u)$ are defined by integrating $\mu(\Box_g\Psi)(V_i\Psi)$ over the domain $D^{t,u}$ and performing integration by parts:
\begin{align}
E_1[\Psi](t, u) &:= \frac{1}{2} \int_{\Sigma_t^{u}} \mu \bigl\{ (\mathring{L}\Psi)^2 + |\slashed{d}\Psi|^2 \bigr\}, \label{E1} \\
E_2[\Psi](t, u) &:= \frac{1}{2} \int_{\Sigma_t^{u}} \bigl\{ (\mathring{\underline{L}}\Psi)^2 + \mu^2 |\slashed{d}\Psi|^2\bigr\}, \label{E2} \\
F_1[\Psi](t, u) &:= \int_{C_{u}^t} (\mathring{L}\Psi)^2, \label{F1} \\
F_2[\Psi](t, u) &:= \int_{C_{u}^t} \mu \, |\slashed{d}\Psi|^2. \label{F2}
\end{align}
Thus,
\begin{equation}\label{EI}
E_i[\Psi](t, u)-E_i[\Psi](t_0, u)+F_i[\Psi](t, u)=-\int_{D^{t, u}}\Phi\cdot V_i\Psi-\int_{D^{t, u}}\f12\mu Q_{\al\beta}[\Psi]\leftidx{^{(V_i)}}\pi^{\al\beta},
\end{equation}
where
$Q_{\al\beta}[\Psi]:=(\p_\al\Psi)(\p_\beta\Psi)-\f12 g_{\al\beta}g^{\nu\lambda}(\p_\nu\Psi)(\p_\lambda\Psi)$
and $\leftidx{^{(V_i)}}\pi$ is the deformation tensor with respect to the vector field $V_i$ as defined in \eqref{dt}.
Due to
\begin{align*}
-\frac{1}{2} \mu Q_{\al\beta}[\Psi]\leftidx{^{(V_1)}}\pi^{\al\beta}&= \frac{1}{2}\big(- \varrho^{-1}\mu+\mathring L\mu-\mu(\operatorname{tr}\check{\mathscr X}+\f12\slashed g^{XX}G_{XX}^{\al}\mathring L\vp_\al)\big)|\slashed d\Psi|^2 -\f12\mathring L\mu|\mathring L\Psi|^2\\
&+(2\mu\zeta^X+\slashed d^X\mu)(\mathring L\Psi)(\slashed d_X\Psi)-\f12\operatorname{tr}\chi(\mathring L\Psi)(\underline L\Psi),\\
-\frac{1}{2} \mu Q_{\al\beta}[\Psi]\leftidx{^{(V_2)}}\pi^{\al\beta} &= \frac{1}{2}\big(\underline L\mu+\mu\mathring L\mu-\mu^2\operatorname{tr}\chi-\mu\operatorname{tr}\leftidx{^{(T)}}{\slashed\pi}\big)|\slashed d\Psi|^2 -\mu\slashed d^X\mu(\mathring L\Psi)(\slashed d_X\Psi)\\
&-(2\mu\zeta^X+\slashed d^X\mu)(\underline L\Psi)(\slashed d_X\Psi)-\f12(\mu\operatorname{tr}\chi+\operatorname{tr}\leftidx{^{(T)}}{\slashed\pi})(\mathring L\Psi)(\underline L\Psi),
\end{align*}
it follows from \eqref{eq:higher-order-L-infty-estimates}, \eqref{blowup} and \eqref{Tmu} that
\begin{align}
&\int_{D^{t, u}}-\frac{1}{2} \mu Q_{\al\beta}[\Psi]\leftidx{^{(V_1)}}\pi^{\al\beta}\nonumber\\
\lesssim&\int_{D^{t, u}}\Bigl\{(\mu^{-1}\mathring L\mu)(\mu|\slashed d\Psi|^2)+\delta^{2\nu-3}(\mu|\slashed d\Psi|^2)+\delta^{\nu-2}|\mathring L\Psi|^2+(1+\delta^{2\nu-3})|\mathring L\Psi|(|\slashed d\Psi|+|\underline L\Psi|)\Bigr\}\nonumber\\
\lesssim&\delta^{\nu-2}\int_{t_0}^tE_1[\Psi](s,u)ds-\delta^{\nu-2}\int_{D^{t, u}\cap\{\mu<\f{1}{10}\}}|\slashed d\Psi|^2+\delta^{-1}\int_{u_0}^uF_1[\Psi](t,u')du'\label{QV1}\\
&+\delta(1+\delta^{4\nu-6})\int_{t_0}^tE_2[\Psi](s,u)ds
\nonumber
\end{align}
and
\begin{align}
&\int_{D^{t, u}}-\frac{1}{2} \mu Q_{\al\beta}[\Psi]\leftidx{^{(V_2)}}\pi^{\al\beta}\nonumber\\
\lesssim&\int_{D^{t, u}}\Bigl\{(\mu^{-1}T\mu)_+(\mu|\slashed d\Psi|^2)+\delta^{\nu-2}(\mu|\slashed d\Psi|^2)+(1+\delta^{2\nu-3})|\underline L\Psi||\slashed d\Psi|+\mu|\mathring L\Psi||\slashed d\Psi|\nonumber\\
&\qquad\quad+\delta^{\nu-2}|\mathring L\Psi||\underline L\Psi|\Bigr\}\nonumber\\
\lesssim&\delta^{-\nu/2}\int_{t_0}^t(t^*-s)^{-1/2}E_1[\Psi](s,u)ds+\delta^{-1}\int_{u_0}^uF_2[\Psi](t,u')du'+\delta^{\nu-2}\int_{t_0}^tE_2[\Psi](s,u)ds\label{QV2}\\
&+\delta^{2-\nu}(1+\delta^{4\nu-6})\int_{D^{t, u}\cap\{\mu<\f{1}{10}\}}|\slashed d\Psi|^2\nonumber
\end{align}
Substituting \eqref{QV1} and \eqref{QV2} into \eqref{EI} leads to
\begin{equation}\label{e}
\begin{split}
&\delta^{\nu-1} E_2[\Psi](t, u) + \delta^{\nu-1} F_2[\Psi](t, u) + E_1[\Psi](t, u) + F_1[\Psi](t, u)+ \delta^{\nu-2} \int_{D^{t, u} \cap \{\mu < \tfrac{1}{10}\}} |\slashed{d}\Psi|^2 \\
\lesssim &\delta^{\nu-1} E_2[\Psi](t_0, u) + E_1[\Psi](t_0, u) + \delta^{\nu-1} \bigl| \int_{D^{t, u}} \Phi \cdot \mathring{\underline{L}}\Psi \bigr| + \bigl| \int_{D^{t, u}} \Phi \cdot \mathring{L}\Psi \bigr|.
\end{split}
\end{equation}
Choosing $\Psi = \Psi_\al^{n+1} = Z^{n+1}\varphi_\al$ and then
$\Phi = \Phi_\al^{n+1} = \mu\Box_g\Psi_\al^{n+1}$ ($n \leq 2N-6$)
so that \eqref{gel} holds. Note that
\begin{equation}\label{Psi}
\begin{split}
\Phi_\al^{n+1} &= \mu[\Box_g,Z]\Psi_\al^{n} + Z\big(\mu\Box_g\Psi_\al^{n}\big)-(Z\mu)\Box_g\Psi_\al^{n}\\
&= \mu\mathscr{D}^\beta\big({\leftidx{^{(Z)}}C_{\al}^{n}}_{,\beta}\big)
+ \big(Z + \leftidx{^{(Z)}}\lambda\big)\Phi_\al^{n},
\end{split}
\end{equation}
where
\begin{equation}\label{ClP}
\begin{split}
{\leftidx{^{(Z)}}C_{\al}^{n}}_{,\beta} &= \big(\leftidx{^{(Z)}}\pi_{\nu\beta}
- \frac{1}{2}g_{\nu\beta}\big(g_{\kappa\lambda}\leftidx{^{(Z)}}\pi^{\kappa\lambda}\big)\big)g^{\nu\iota}\partial_\iota\Psi_\al^{n}, \\
\leftidx{^{(Z)}}\lambda &= -\mu^{-1}\leftidx{^{(Z)}}\pi_{\mathring{LT}}
+ \frac{1}{2}\operatorname{tr}\leftidx{^{(Z)}}{\slashed{\pi}} - \mu^{-1}Z\mu,\\
\Psi_\al^0 &= \varphi_\al, \quad
\Phi_\al^0 = \mu\Box_g\varphi_\al,
\end{split}
\end{equation}
and $\Phi_\al^0$ equals the right hand side of \eqref{ge}.
Consequently, for $\Psi_\al^{n+1} = Z_{n+1}Z_{n}\cdots Z_{1}\varphi_\al$ with $Z_j \in \{\varrho\mathring{L}, T, R\}$,
by \eqref{Psi}, the induction argument gives
\begin{equation}\label{Phik}
\begin{split}
\Phi_\al^{n+1} &= \sum_{j=1}^{n} \big(Z_{n+1} + \leftidx{^{(Z_{n+1})}}\lambda\big) \cdots
\big(Z_{n+2-j} + \leftidx{^{(Z_{n+2-j})}}\lambda\big)
\big(\mu\mathscr{D}^\beta{\leftidx{^{(Z_{n+1-j})}}C_{\al}^{n-j}}_{,\beta}\big) \\
&\quad + \mu\mathscr{D}^\beta{\leftidx{^{(Z_{n+1})}}C_\al^{n}}_{,\beta}
+ \big(Z_{n+1} + \leftidx{^{(Z_{n+1})}}\lambda\big) \cdots
\big(Z_{1} + \leftidx{^{(Z_1)}}\lambda\big)\Phi_\al^0 \\
&=: J_1^{n+1} + J_2^{n+1}, \qquad n \geq 1, \\
\Phi_\al^{1} &= \big(Z_{1} + \leftidx{^{(Z_1)}}\lambda\big)\Phi_\al^0
+ \mu\mathscr{D}^\beta{\leftidx{^{(Z_{1})}}C_{\al}^{0}}_{,\beta},
\end{split}
\end{equation}
where $J_1^{n+1}$ and $J_2^{n+1}$ stand for the first and second lines on the right hand side of \eqref{Phik}, respectively.
By \eqref{Lpi}--\eqref{Rpi}, one has
\begin{equation}\label{lamda}
\begin{split}
\leftidx{^{(T)}}\lambda &= \frac{1}{2}\operatorname{tr}\leftidx{^{(T)}}{\slashed{\pi}}, \qquad
\leftidx{^{(\varrho\mathring{L})}}\lambda = \varrho\big\{\operatorname{tr}{\check{\mathscr X}}+\f12\slashed g^{XX}G_{XX}^\al\mathring L\vp_\al\big\}+2, \qquad
\leftidx{^{(R)}}\lambda = \frac{1}{2}\operatorname{tr}\leftidx{^{(R)}}{\slashed{\pi}}.
\end{split}
\end{equation}
In addition, as in \cite[(5.1)-(5.6) and (5.10)-(5.12)]{Ding4}, the
term $\mu\mathscr{D}^\beta{\leftidx{^{(Z)}}C_{\al}^{n}}_{,\beta}$ can be written as
\begin{equation}\label{muZC}
\mu\mathscr{D}^\beta{\leftidx{^{(Z)}}C_{\al}^{n}}_{,\beta}
= \leftidx{^{(Z)}}D_{\al,1}^n + \leftidx{^{(Z)}}D_{\al,2}^n + \leftidx{^{(Z)}}D_{\al,3}^n,
\end{equation}
where
\begin{align}
\leftidx{^{(T)}}D_{\al,1}^{n}=&(T\mu)\mathring L^2\Psi_\al^n+\mu(\slashed d_X\mu+2\mu\zeta_X)\slashed d^X\mathring L\Psi_\al^n+\f12\operatorname{tr}\leftidx{^{(T)}}{\slashed\pi}(\mathring L\mathring{\underline L}\Psi_\al^n+\f12\operatorname{tr}\chi\mathring{\underline L}\Psi_\al^n)\no\\
&+(\slashed d_X\mu+2\mu\zeta_X)\slashed d^X\mathring{\underline L}\Psi_\al^n-(T\mu)\slashed\triangle\Psi_\al^n
+\f12\mu\operatorname{tr}\leftidx{^{(T)}}{\slashed\pi}\slashed\triangle\Psi_\al^n,\label{T1}\\
\leftidx{^{(T)}}D_{\al,2}^{n}=&\big\{\mathring L T\mu+\f14\uwave{\mathring{\underline L}(\operatorname{tr}\leftidx{^{(T)}}{\slashed\pi})}+\boxed{\slashed\nabla_X\big(\f12\mu\slashed d^X\mu}+\mu^2\zeta^X\big)\big\}\mathring L\Psi_\al^n+\big\{\f14\mathring L(\operatorname{tr}\leftidx{^{(T)}}{\slashed\pi})\no\\
&+\f12\boxed{\slashed\nabla_X(\slashed d^X\mu}+2\mu\zeta^X)\big\}\mathring{\underline L}\Psi_\al^n+\big\{\f12\slashed{\mathcal L}_{\mathring L}(2\mu^2\zeta_X+\mu\slashed d_X\mu)-\uline{\slashed d_XT\mu}\no\\
&+\underline{\f12\slashed{\mathcal L}_{\mathring{\underline L}}(\slashed d_X\mu}+2\mu\zeta_X)+\underbrace{\f12\slashed d_X(\mu\operatorname{tr}\leftidx{^{(T)}}{\slashed\pi})}\big\}\slashed d^X\Psi_\al^n,\label{T2}\\
\leftidx{^{(T)}}D_{\al,3}^{n}=&\big\{\operatorname{tr}\chi T\mu+\f14(\mu\operatorname{tr}\chi+\operatorname{tr}\leftidx{^{(T)}}{\slashed \pi})\operatorname{tr}\leftidx{^{(T)}}{\slashed \pi}-\f12|\slashed d\mu|^2-\mu\zeta_X(\slashed d^X\mu)\big\}\mathring L\Psi_\al^n\no\\
&+(\f12\mathring L\mu-\mu\operatorname{tr}\chi)(\slashed d_X\mu+2\mu\zeta_X)\slashed d^X\Psi_\al^n,\label{T3}\\
\leftidx{^{(\varrho\mathring L)}}D_{\al,1}^{n}=&(2-\mu+\varrho\mathring L\mu)\mathring L^2\Psi_\al^n-2\varrho(\slashed d_X\mu+2\mu\zeta_X)\slashed d^X\mathring L\Psi_\al^n+\varrho\operatorname{tr}\chi(\mathring L\mathring{\underline L}\Psi_\al^n+\f12\operatorname{tr}\chi\mathring{\underline L}\Psi_\al^n)\no\\
&-(\varrho\mathring L\mu-\varrho\mu\operatorname{tr}{\check{\mathscr X}}-\f12\varrho\mu\slashed g^{XX}G_{XX}^\beta\mathring L\vp_\beta)\slashed\triangle\Psi_\al^n,\label{rL1}\\
\leftidx{^{(\varrho\mathring L)}}D_{\al,2}^{n}=&\big\{\mathring L\big(\varrho\mathring L\mu-\mu\big)+\uwave{\f12\mathring{\underline L}\big(\varrho\operatorname{tr}\check{\mathscr X}+\f12\varrho\slashed g^{XX}G_{XX}^\al\mathring L\vp_\al\big)}-\varrho\boxed{\slashed\nabla_X\big(\slashed d^X\mu}+2\mu\zeta^X\big)\big\}\mathring L\Psi_\al^n\no\\
&+\f12\mathring L\big(\varrho\operatorname{tr}\check{\mathscr X}+\f12\varrho\slashed g^{XX}G_{XX}^\beta\mathring L\vp_\beta\big)\mathring{\underline L}\Psi_\al^n\no\\
&-\big\{\slashed{\mathcal L}_{\mathring L}\big(\varrho\slashed d_X\mu+2\varrho\mu\zeta_X\big)+\slashed d_X\big(\mu+\varrho\mathring L\mu\big)-\varrho\underbrace{\slashed d_X(\mu\operatorname{tr}\chi)}\big\}\slashed d^X\Psi_\al^n,\label{rL2}\\
\leftidx{^{(\varrho\mathring L)}}D_{\al,3}^{n}=&\operatorname{tr}\chi\big\{2-\mu+\varrho\mathring L\mu+\f12\varrho\operatorname{tr}\leftidx{^{(T)}}{\slashed\pi}+\f12\varrho\mu\operatorname{tr}\chi\big\}\mathring L\Psi_\al^n+2\varrho\operatorname{tr}\chi(2\mu\zeta^X+\slashed d^X\mu)\slashed d_X\Psi_\al^n,\label{rL3}\\
\leftidx{^{(R)}}D_{\al,1}^{n}=&(R\mu)\mathring L^2\Psi_\al^n-\leftidx{^{(R)}}{\slashed\pi}_{\mathring{\underline L}X}\slashed d^X\mathring L\Psi_\al^n+\f12\operatorname{tr}\leftidx{^{(R)}}{\slashed\pi}(\mathring L\mathring{\underline L}\Psi_\al^n+\f12\operatorname{tr}\chi\mathring{\underline L}\Psi_\al^n)-\leftidx{^{(R)}}{\slashed\pi}_{\mathring LX}\slashed d^X\mathring{\underline L}\Psi_\al^n\no\\
&-(R\mu)\slashed\triangle\Psi_\al^n
+\f12\mu\operatorname{tr}\leftidx{^{(R)}}{\slashed\pi}\slashed\triangle\Psi_\al^n,\label{R1}\\
\leftidx{^{(R)}}D_{\al,2}^n=&\big\{\mathring LR\mu
-\f12\underbrace{\slashed\nabla^X\leftidx{^{(R)}}{\slashed\pi}_{\mathring{\underline L}X}}+\f14\uwave{\mathring{\underline L}(\operatorname{tr}\leftidx{^{(R)}}{\slashed\pi})}\big\}\mathring L\Psi_\al^n+\big\{\f14\mathring L(\operatorname{tr}\leftidx{^{(R)}}{\slashed{\pi}})\no\\
&-\f12\underbrace{\slashed\nabla^X\leftidx{^{(R)}}{\slashed\pi}_{\mathring LX}}\big\}\mathring{\underline L}\Psi_\al^n+\big\{-\f12\slashed{\mathcal L}_{\mathring L}\leftidx{^{(R)}}{\slashed\pi}_{\mathring{\underline L}X}-\f12\uwave{\slashed{\mathcal L}_{\mathring{\underline L}}\leftidx{^{(R)}}{\slashed\pi}_{\mathring LX}}
-\boxed{\slashed d_XR\mu}\no\\
&+\f12\underbrace{\slashed d_X(\mu\operatorname{tr}\leftidx{^{(R)}}{\slashed\pi})}\big\}\slashed d^X\Psi_\al^n.\label{R2}\\
\leftidx{^{(R)}}D_{\gamma,3}^{n}=&\big\{\operatorname{tr}\chi R\mu+\f14(\operatorname{tr}\leftidx{^{(T)}}{\slashed\pi}+\mu\operatorname{tr}\chi)\operatorname{tr}\leftidx{^{(R)}}{\slashed\pi}+\f12\slashed d^X\mu\leftidx{^{(R)}}{\slashed\pi}_{\mathring LX}\big\}\mathring L\Psi_\al^n
+\big\{\operatorname{tr}\chi\leftidx{^{(R)}}{\slashed\pi}_{TX}\no\\
&+\operatorname{tr}\leftidx{^{(R)}}{\slashed\pi}(\mu\zeta_X+\f12\slashed d_X\mu)+(-\f12\mathring L\mu+\mu\operatorname{tr}\chi+\f12\operatorname{tr}\leftidx{^{(T)}}{\slashed\pi})\leftidx{^{(R)}}{\slashed\pi}_{\mathring LX}\big\}\slashed d^X\Psi_\al^n.\label{R3}
\end{align}
Note that all the terms in $\leftidx{^{(Z)}}D_{\al,1}^m$ are products of the deformation tensor and the secon{-}order derivatives of $\Psi_\al^n$, except for the first term, which contains a factor of the form
\[
\mathring{L}\mathring{\underline{L}}\Psi_\al^n + \frac{1}{2}\operatorname{tr}{\chi}\mathring{\underline{L}}\Psi_\al^n
\]
(see \eqref{T1}, \eqref{rL1} and \eqref{R1}). It should be emphasized that this structure is crucial for our analysis
when $\Psi_\al^n$ includes
some derivatives of $\varphi_\al$. By \eqref{fequation}, the identity
\[
\mathring{L}\mathring{\underline{L}}\varphi_\al + \frac{1}{2}\operatorname{tr}{\chi}\mathring{\underline{L}}\varphi_\al
=\mu\slashed\triangle\vp_\al+ H_\al + \frac{1}{2}(\operatorname{tr}\check{\mathscr X}+\f12\slashed g^{XX}G_{XX}^\beta\mathring L\vp_\beta)\mathring{\underline{L}}\varphi_\al
\]
exhibits a better smallness order than that of either $\mathring{L}\mathring{\underline{L}}\varphi_\al$ or $\frac{1}{2}\operatorname{tr}{\chi}\mathring{\underline{L}}\varphi_\al$ individually.
The term $\leftidx{^{(Z)}}D_{\al,2}^n$ collects all products of the first order derivatives of the deformation tensor and the first
order derivatives of $\Psi_\al^n$, while $\leftidx{^{(Z)}}D_{\al,3}^n$ contains the remaining terms.
The explicit expressions of $\Phi_\al^{n+1}$ given in \eqref{Phik} and \eqref{muZC}--\eqref{R3} are important
for estimating the last two integrals in \eqref{e}. To derive the energy estimates, we introduce the
following definitions of energies and fluxes, inspired by \cite{Ch1,M-Y,Sp}:
\begin{equation*}\label{higher order energy and flux}
\begin{split}
\mathcal{E}_{i,\leq m+1}(t,u) &=
\sum_{\al=0}^2 \sum_{n \leq m}
\delta^{2l+2k(2-\nu)} \, E_i\bigl[Z^n \vp_\al\bigr](t,u),
\quad i=1,2, \\[8pt]
\mathcal{F}_{i,\leq m+1}(t,u) &=
\sum_{\al=0}^2 \sum_{n \leq m}
\delta^{2l+2k(2-\nu)} \, F_i\bigl[Z^n \vp_\al\bigr](t,u),
\quad i=1,2, \\[8pt]
\end{split}
\end{equation*}
\begin{equation}\label{higher order energy and flux}
\begin{split}
\widetilde{\mathcal{E}}_{i,\leq m+1}(t,u) &=
\sup_{t_0 \leq t' \leq t}
\big\{ \mu_{\min}^{2b_{m+1}}(t') \, \mathcal{E}_{i,\leq m+1}(t',u) \big\},
\quad i=1,2, \\[8pt]
\widetilde{\mathcal{F}}_{i,\leq m+1}(t,u) &=
\sup_{t_0 \leq t' \leq t}
\big\{ \mu_{\min}^{2b_{m+1}}(t') \, \mathcal{F}_{i,\leq m+1}(t',u) \big\},
\quad i=1,2.
\end{split}
\end{equation}
Here, $l$ (respectively $k$) denotes the number of $T$ (respectively $\rho \mathring{L}$) operators in $Z^n$,
and the sequence $\{b_k\}$ will be determined subsequently.

To establish higher order energy estimates for \eqref{ge}, we need to derive the higher order $L^2$ estimates for
certain related quantities.
These estimates allow the last two terms in \eqref{e} to be absorbed into the left hand side, thereby yield the desired energy estimates.
To achieve this, we apply two lemmas from \cite[Lemma 7.3]{M-Y} and \cite[Lemma 12.57]{Sp}, respectively.

\begin{lemma}\label{L2T}
Let $\psi \in C^1(D^{t, u})$ be a function vanishing on $C_{u_0}$. For small $\delta > 0$, the following estimates hold
\begin{align}
\int_{S_{t, u}} \psi^2 &\lesssim \delta \int_{\Sigma_t^u} |T\psi|^2 \lesssim \delta \int_{\Sigma_{t}^{u}}
\big(|\mathring{\underline{L}}\psi|^2 + \mu^2 |\mathring{L}\psi|^2\big), \label{SSi} \\
\int_{\Sigma_{t}^{u}} \psi^2 &\lesssim \delta^2 \int_{\Sigma_t^u} |T\psi|^2 \lesssim \delta^2 \int_{\Sigma_{t}^{u}} \big( |\mathring{\underline{L}}\psi|^2 + \mu^2 |\mathring{L}\psi|^2 \big). \label{SiSi}
\end{align}
As a consequence,
 \begin{align}
 \int_{S_{t, u}} \psi^2 &\lesssim \delta \big( E_2[\psi](t, u) + E_1[\psi](t, u) \big), \label{SE} \\
 \int_{\Sigma_{t}^{u}} \psi^2 &\lesssim \delta^2 \big( E_2[\psi](t, u) + E_1[\psi](t, u) \big). \label{SiE}
 \end{align}
\end{lemma}
\begin{lemma}\label{L2L}
 For a given function $f \in C(D^{t, u})$, define
 \[
 F(t,u,\vartheta) := \int_{t_0}^t f(\tau,u,\vartheta) \, d\tau.
 \]
 Under the assumptions \eqref{bootstrap-assumptions}, the following estimate holds for small $\delta > 0$,
 \begin{equation}\label{Ff}
 \|F\|_{t,u} \lesssim \int_{t_0}^t \|f\|_{\tau,u} \, d\tau,
 \end{equation}
 where $\|F\|_{t,u}$ denotes the $L^2$ norm of $F$ over $\Sigma_t^u$, as defined in Definition \ref{3.4}.
\end{lemma}
\subsection{$L^2$ norms of some related quantities}
We now derive the $L^2$ estimates for some quantities such as
${\check{\mathscr X}}$, $\leftidx{^{(Z)}}{\slashed\pi}_{\mathring L}$, $\leftidx{^{(R)}}{\slashed\pi}$, $\check{L}$, $\upsilon$, $y^j$
$(j=1,2)$ and $R$. It should be noted that the choice of the good known $\check{\mathscr X}$ and suitable smallness orders of $\delta$
play crucial roles here.
\begin{proposition}\label{L2chi}
 Under the assumptions \eqref{bootstrap-assumptions}, when $\de>0$ is small, it holds that for $m\leq 2N-6$,
 \begin{equation}
 \begin{split}
 &\de^{l+k(2-\nu)}\|Z^{m}\operatorname{tr}\check{\mathscr X}\|_{t,u}\les\de^{\nu-\f12}+\Theta^1(t,u),\\
 &\de^{l+k(2-\nu)}\|Z^{m+1}\mu\|_{t,u}\les\de^{\f12}+\Theta^2(t,u),\\
 &\de^{l+k(2-\nu)}\|Z^{m+2} y^j\|_{t,u}\les\de^{\f12}+\Theta^3(t,u),\\
 &\de^{l+k(2-\nu)}\|\Lies_Z^{m+1}R\|_{t,u}\les(\delta^{2\nu-2}+\delta^{2-\nu})\big(\de^{\nu-\f12}+\Theta^3(t,u)\big),\\
 &\de^{l+k(2-\nu)}\|\Lies_Z^{m+1}\slashed g\|_{t,u}\les\de^{\f12}(\delta^{\nu-1}+\delta^{2-\nu})+\Theta^3(t,u),\\
 &\de^{l+k(2-\nu)}\|\Lies_Z^{m}{}^{(R)}\pis_L\|_{t,u}\les(1+\delta^{3\nu-4})\big(\delta^{\nu-\f12}+\Theta^3(t,u)\big),\\
 &\de^{l+k(2-\nu)}\|\Lies_Z^{m}{}^{(R)}\pis_T\|_{t,u}\les(1+\delta^{2\nu-3})\big(\delta^{\nu-\f12}+\Theta^3(t,u)\big),\\
 &\de^{l+k(2-\nu)}\big(\|Z^{m+1}\Lc^j\|_{t,u}+\|Z^{m+1}\check\varrho\|_{t,u}+\|Z^{m+1}v\|_{t,u}+\|\Lies_Z^{m}{}^{(R)}\pis\|_{t,u}\big)\les\de^{\nu-\f12}
+\Theta^3(t,u),\\
 &\de^{l+k(2-\nu)}\|\Lies_Z^{m}{}^{(T)}\pis\|_{t,u}\les\delta^{\nu-\f32}+\sqrt{\mathcal E_{1,\leq m+2}(t,u)}
+\sqrt{\mathcal E_{2,\leq m+2}(t,u)}\\
 &\qquad\qquad+\delta^{\nu-2}\int_{t_0}^t\mu_{\min}^{-\f12}(\tau)\sqrt{\mathcal E_{1,\leq m+2}(\tau,u)}d\tau+(1+\delta^{2\nu-3})\int_{t_0}^t\sqrt{\mathcal E_{2,\leq m+2}(\tau,u)}d\tau,\\
 &\de^{l+k(2-\nu)}\|\Lies_Z^{m}{}^{(T)}\pis_L\|_{t,u}\les\delta^{2\nu-\f52}+\delta^{\f12}+\delta^{\nu-1}\big(\sqrt{\mathcal E_{1,\leq m+2}(t,u)}+\sqrt{\mathcal E_{2,\leq m+2}(t,u)}\big)\\
 &\qquad+(1+\delta^{2\nu-3})\int_{t_0}^t\mu_{\min}^{-\f12}(\tau)\sqrt{\mathcal E_{1,\leq m+2}(\tau,u)}d\tau+(1+\delta^{3\nu-4})\int_{t_0}^t\sqrt{\mathcal E_{2,\leq m+2}(\tau,u)}d\tau,
 \end{split}
 \end{equation}
 where $l$ (resp. $k$) is the number of $T$ (resp. $\rho \mathring L$) in the string of $Z$, and
 \begin{equation*}
 \begin{split}
 &\Theta^1(t,u)=\int_{t_0}^t\big(\mu_{\min}^{-\f12}(\tau)\sqrt{\mathcal E_{1,\leq m+2}(\tau,u)}+\de^{\nu-1}\sqrt{\mathcal E_{2,\leq m+2}(\tau,u)}\big)\d\tau,\\
 &\Theta^2(t,u)=\int_{t_0}^t\big(\mu_{\min}^{-\f12}(\tau)\sqrt{\mathcal E_{1,\leq m+2}(\tau,u)}+\sqrt{\mathcal E_{2,\leq m+2}(\tau,u)}\big)\d\tau,\\
 &\Theta^3(t,u)=\Theta^1(t,u)+\delta\big(\sqrt{\mathcal E_{1,\leq m+2}(t,u)}+\sqrt{\mathcal E_{2,\leq m+2}(t,u)}\big).
 \end{split}
 \end{equation*}
\end{proposition}

\begin{remark}
 The proof of Proposition \ref{L2chi} is somewhat similar to that of \cite[Proposition 3.3]{Ding4} concerning the global smooth solution problem.
 However, in \cite{Ding4}, greater attention was paid to the decay rate of each quantity since $\mu \sim 1$ holds. In the present work,
$\mu \to 0^+$ will happen when the finite blowup distance is approached. Therefore, we have to provide the detailed proof of Proposition \ref{L2chi}.
\end{remark}

\begin{proof}
By Lemma \ref{Lem3.3}, the explicit expressions of $\varrho\mathring{L}\check{L}^i$, $T\check{L}^i$ and $R\check{L}^i$
can be given as
\[
\varrho\mathring{L}\check{L}^i = \varrho\mathring{L}\mathring{L}^i - \check{L}^i,
\quad
T\check{L}^i = T\mathring{L}^i + \frac{\mu}{\varrho}g^{0i} + \frac{\mu}{\varrho}\check{L}^i + \frac{\mu-1}{\varrho^2}y^i,
\quad
R\check{L}^i = {R^X}\slashed{d}_X\check{L}^i.
\]
One then gets that by $\|1\|_{s,u}\lesssim\delta^{1/2}$,
\begin{equation}\label{ZkL}
\begin{split}
&\delta^{l+k(2-\nu)} \| Z^{m+1}\check{L}^i \|_{t,u} \\
\lesssim&\delta^{\nu-\frac{1}{2}} \bigl( \delta^{\nu-1} + \delta^{2-\nu} \bigr)
+ \delta^{l_1+k_1(2-\nu)} \| Z^{m_1} \operatorname{tr}\check{\mathscr{X}} \|_{t,u} + \delta^{\nu-1} \delta^{l_0+k_0(2-\nu)} \| Z^{m_0}y \|_{s,u}\\
&+ \bigl( \delta + \delta^{2\nu-2} \bigr) \delta^{l_0+k_0(2-\nu)} \| Z^{m_0}\mu \|_{t,u} + \bigl( \delta + \delta^{2\nu-2} \bigr) \delta^{l_1+k_1(2-\nu)} \| \slashed{\mathcal{L}}_Z^{m_1}\slashed{g} \|_{t,u} \\
&+ \delta \big( \sqrt{\mathcal{E}_{1,\leq m+2}(t,u)} + \sqrt{\mathcal{E}_{2,\leq m+2}(t,u)} \big),
\end{split}
\end{equation}
where $l$ (or $k$) and $l_p$ ($p=0,1$) are the numbers of $T$ (or $\varrho\mathring{L}$) in $Z^{m+1}$ and $Z^{m_p}$, respectively,
meanwhile $1 \leq m_p \leq m+1-p$.
Similarly, due to $\varrho Ly^i=\varrho\check L^i+y^i$, $Ty^i=\mu(-g^{0i}-\check L^i-\f{y^i}\varrho)$ and $Ry^a=\Omega^a+\upsilon(g^{0a}+\check L^a+\f{y^a}\varrho)$, then
\begin{equation}\label{Zkx}
\begin{split}
&\delta^{l+k(2-\nu)}\|Z^{m+2}y\|_{t,u}\\
\lesssim&\delta^{\f12}+\delta^{l_0+k_0(2-\nu)}\|Z^{m_0}\upsilon\|_{t,u}
+(\delta^{\nu-1}+\delta^{2-\nu})\sum_{j=1}^2\delta^{l_0+k_0(2-\nu)}\|Z^{m_0}\check L^j\|_{t,u}\\
&+\delta\delta^{l_{0}+k_0(2-\nu)}\|Z^{m_{0}}\mu\|_{t,u}+\delta^{\nu}\big( \sqrt{\mathcal{E}_{1,\leq m+2}(t,u)} + \sqrt{\mathcal{E}_{2,\leq m+2}(t,u)} \big).
\end{split}
\end{equation}
And it follows from $\slashed{\mathcal L}_{\varrho\mathring L}R=\varrho\leftidx{^{(R)}}{\slashed\pi}_{\mathring L}$, $\slashed{\mathcal L}_TR=\leftidx{^{(R)}}{\slashed\pi}_T$ and \eqref{Rpi} that
\begin{equation}\label{ZR}
\begin{split}
&\delta^{l+k(2-\nu)}\|\slashed{\mathcal L}_Z^{m+1}R\|_{t,u}\\
\lesssim&\delta^{\f32}+(\delta^{3\nu-3}+\delta)\delta^{l_0+k_0(2-\nu)}\|Z^{m_0}y\|_{t,u}+(\delta+\delta^{2\nu-2})\delta^{l_1+k_1(2-\nu)}\|Z^{m_1}\upsilon\|_{t,u}\\
&+\delta^{2-\nu}\delta^{l_1+k_1(2-\nu)}\|Z^{m_1}\operatorname{tr}\check{\mathscr X}\|_{t,u}+(\delta^{3\nu-3}+\delta^{2-\nu})\sum_{j=1}^2\delta^{l_0+k_0(2-\nu)}\|Z^{m_1}\check L^j\|_{t,u}\\
&+(\delta^{3\nu-3}+\delta)\delta^{l_1+k_1(2-\nu)}\|\slashed{\mathcal L}_Z^{m_1}\slashed g\|_{t,u}+\delta^{\nu}\delta^{l_{0}+k_0(2-\nu)}\|Z^{m_{0}}\mu\|_{t,u}\\
&+(\delta^{2\nu-1}+\delta^2) \big( \sqrt{\mathcal{E}_{1,\leq m+2}(t,u)} + \sqrt{\mathcal{E}_{2,\leq m+2}(t,u)}\big).
\end{split}
\end{equation}
Additionally, by \eqref{omega}, \eqref{rrho} and the facts of $\slashed{\mathcal L}_{\varrho\mathring L}\slashed g=\varrho\check\chi+\slashed g$, $\slashed{\mathcal L}_T\slashed g=\leftidx{^{(T)}}{\slashed\pi}$ and $\slashed{\mathcal L}_R\slashed g=\leftidx{^{(R)}}{\slashed\pi}$, it holds
\begin{equation}\label{Zku}
\begin{split}
\delta^{l+k(2-\nu)}\|Z^{m+1}\upsilon\|_{t,u}
\lesssim& \delta^{\nu-1}\delta^{l_0+k_0(2-\nu)}\|Z^{k_0}y\|_{t,u}+\sum_{i=1}^2\delta^{l_0+k_0(2-\nu)}\|Z^{m_0}\check L^i\|_{t,u}\\
&+\delta\bigl( \sqrt{\mathcal{E}_{1,\leq m+2}(t,u)} + \sqrt{\mathcal{E}_{2,\leq m+2}(t,u)} \bigr),
\end{split}
\end{equation}
\begin{equation}\label{Zkr}
\begin{split}
\delta^{l+k(2-\nu)}\|Z^{m+1}\check\varrho\|_{t,u}\lesssim&\delta^{\nu-1}\delta^{l_0+k_0(2-\nu)}\|Z^{k_0}y\|_{t,u}
+\sum_{i=1}^2\delta^{l_0+k_0(2-\nu)}\|Z^{m_0}\check L^i\|_{t,u}\\
&+\delta^{2\nu-\f32}+ \delta \bigl( \sqrt{\mathcal{E}_{1,\leq m+2}(t,u)} + \sqrt{\mathcal{E}_{2,\leq m+2}(t,u)} \bigr),
\end{split}
\end{equation}
\begin{equation}\label{Zkg}
\begin{split}
\delta^{l+k(2-\nu)}\|\slashed{\mathcal L}_Z^{m+1}\slashed g\|_{t,u}\lesssim&(\delta^{\nu-1}+\delta^{2-\nu})\bigl(\delta^{1/2}+\delta^{l_1+k_1(2-\nu)}\|Z^{m_1}\operatorname{tr}\check{\mathscr X}\|_{t,u}\bigr)\\
&+\delta^{l_1+k_1(2-\nu)}\|Z^{m_1}\upsilon\|_{t,u}+\delta\delta^{l_{0}+k_0(2-\nu)}\|Z^{m_{0}}\mu\|_{t,u}\\
&+\delta^{\nu-1}\sum_{i=1}^2\delta^{l_0+k_0(2-\nu)}\|Z^{m_1}\check L^i\|_{t,u}+\delta^{\nu-1}\delta^{l_0+k_0(2-\nu)}\|Z^{k_0}y\|_{t,u}\\
&+ \delta \bigl( \sqrt{\mathcal{E}_{1,\leq m+2}(t,u)} + \sqrt{\mathcal{E}_{2,\leq m+2}(t,u)} \bigr).
\end{split}
\end{equation}
Collecting \eqref{ZkL}-\eqref{Zkg} yields
\begin{align}
\delta^{l+k(2-\nu)}\|Z^{m+2}y^i\|_{t,u}\lesssim &\delta^{1/2}+\Pi(t,u),\label{Zk1x}\\
\delta^{l+k(2-\nu)}\|\slashed{\mathcal L}_Z^{m+1}\slashed g\|_{t,u}\lesssim&\delta^{1/2}(\delta^{\nu-1}+\delta^{2-\nu})+\Pi(t,u),\label{Zk1g}\\
\delta^{l+k(2-\nu)}\bigl(\|Z^{m+1}\check L^i\|_{t,u}&+\|Z^{m+1}\upsilon\|_{t,u}+\|Z^{m+1}\check\varrho\|_{t,u}\bigr)\lesssim\delta^{\nu-\f12}+\Pi(t,u),\label{Zk1L}\\
\delta^{l+k(2-\nu)}\|\slashed{\mathcal L}_Z^{m+1}R\|_{t,u}\lesssim & \delta^{\f32}+\delta^{3\nu-\f52}+(\delta^{2\nu-2}+\delta^{2-\nu})\|Z^{m_1}\operatorname{tr}\check{\mathscr X}\|_{t,u}\no\\
&+(\delta^\nu+\delta^{3-\nu}+\delta^{4\nu-4})\delta^{l_{0}+k_0(2-\nu)}\|Z^{m_{0}}\mu\|_{t,u}\no\\
&+(\delta^{2\nu-1}+\delta^{3-\nu})\bigl( \sqrt{\mathcal{E}_{1,\leq m+2}(t,u)} + \sqrt{\mathcal{E}_{2,\leq m+2}(t,u)} \bigr),\label{Zk1R}
\end{align}
where
\begin{align*}
\Pi(t,u)=&\delta^{l_1+k_1(2-\nu)}\|Z^{m_1}\operatorname{tr}\check{\mathscr X}\|_{t,u}+\bigl( \delta + \delta^{2\nu-2} \bigr) \delta^{l_0+k_0(2-\nu)} \| Z^{m_0}\mu \|_{t,u}\\
&+\delta \big( \sqrt{\mathcal{E}_{1,\leq m+2}(t,u)} + \sqrt{\mathcal{E}_{2,\leq m+2}(t,u)} \big).
\end{align*}
Thanks to \eqref{Lpi}-\eqref{Rpi}, one can deduce
\begin{align}
\delta^{l+k(2-\nu)}\|\slashed{\mathcal L}_Z^{m}\leftidx{^{(R)}}{\slashed\pi}\|_{t,u}\lesssim&\delta^{\nu-\f12}+\Pi(t,u),\label{ZkRp}\\
\delta^{l+k(2-\nu)}\|\slashed{\mathcal L}_Z^{m}\leftidx{^{(R)}}{\slashed\pi}_{\mathring L}\|_{t,u}\lesssim&(1+\delta^{3\nu-4})\bigl(\delta^{\nu-\f12}+\Pi(t,u)\bigr),\label{ZkRLp}\\
\delta^{l+k(2-\nu)}\|\slashed{\mathcal L}_Z^{m}\leftidx{^{(R)}}{\slashed\pi}_T\|_{t,u}\lesssim&(1+\delta^{2\nu-3})\bigl(\delta^{\nu-\f12}+\Pi(t,u)\bigr)+\delta^{\nu-1}\delta^{l+k(2-\nu)} \| Z^{m+1}\mu \|_{t,u},\label{ZkRLpi}\\
\delta^{l+k(2-\nu)}\|\slashed{\mathcal L}_Z^m\leftidx{^{(T)}}{\slashed\pi}\|_{t,u}\lesssim&\delta^{\nu-2}\delta^{l_1+k_1(2-\nu)}\|Z^{m_1}\operatorname{tr}\check{\mathscr X}\|_{t,u}+(1+\delta^{3\nu-4})\delta^{l_0+k_0(2-\nu)} \| Z^{m_0}\mu \|_{t,u}\no\\
&+\delta^{\nu-\f32}
+\sqrt{\mathcal{E}_{1,\leq m+2}(t,u)} + \sqrt{\mathcal{E}_{2,\leq m+2}(t,u)} ,\label{ZkTpi}\\
\delta^{l+k(2-\nu)}\|\slashed{\mathcal L}_Z^m\leftidx{^{(T)}}{\slashed\pi}_{\mathring L}\|_{t,u}\lesssim&\delta^{2\nu-3}\delta^{l_1+k_1(2-\nu)}\|Z^{m_1}\operatorname{tr}\check{\mathscr X}\|_{t,u}+(1+\delta^{3\nu-4})\delta^{l_0+k_0(2-\nu)} \| Z^{m_0}\mu \|_{t,u}\no\\
&+\delta^{2\nu-\f52}
+\delta^{\nu-1}\bigl(\sqrt{\mathcal{E}_{1,\leq m+2}(t,u)} + \sqrt{\mathcal{E}_{2,\leq m+2}(t,u)}\bigr).\label{ZkTLpi}
\end{align}
Note that all the terms in the left hand side of \eqref{Zk1x}-\eqref{ZkTLpi} are controlled by the $L^2-$norms of
the derivatives of $\operatorname{tr}{\check{\mathscr X}}$ and $\mu$.
Next we deal with $\|Z^{m}\operatorname{tr}{\check{\mathscr X}}\|_{t,u}$ and $\|Z^{m+1}\mu\|_{t,u}$.

It follows from \eqref{Lchi'}, \eqref{eq:higher-order-L-infty-estimates} and \eqref{Zk1x}-\eqref{ZkTLpi} that
\begin{align}
&\delta^{l+k(2-\nu)}\|\mathring LZ^m\operatorname{tr}{\check{\mathscr X}}\|_{t,u}\lesssim\delta^{l+k(2-\nu)}\bigl(\|Z^m\mathring L\operatorname{tr}{\check{\mathscr X}}\|_{t,u}+\|[\mathring L,Z^m]\operatorname{tr}{\check{\mathscr X}}\|_{t,u}\bigr)\nonumber\\
\lesssim&\delta^{l_1+k_1(2-\nu)}\bigl(\delta^{\nu-1}\|\slashed{\mathcal L}_Z^{m_1}\leftidx{^{(R)}}{\slashed\pi}_{\mathring L}\|_{t,u}+\delta^{\nu-1}\|\slashed{\mathcal L}_Z^{m_1}\leftidx{^{(R)}}{\slashed\pi}\|_{t,u}+(1+\delta^{2\nu-3})\|Z^{m_1}\operatorname{tr}{\check{\mathscr X}}\|_{t,u}\bigr)\nonumber\\
&+\delta^{l_0+k_0(2-\nu)}\bigl(\|\mathring LZ^{m_0}\varphi\|_{t,u}+\|\slashed dZ^{m_0}\varphi\|_{t,u}+\delta^{2\nu-2}(1+\delta^{3\nu-4})\|\slashed{\mathcal L}_Z^{m_0}\slashed g\|_{t,u}\bigr)\nonumber\\
&+\delta^{2\nu-3}\delta^{l_1+k_1(2-\nu)}\sum_{i=1}^2\bigl(\|\slashed dZ^{m_1}y^i\|_{t,u}+\|\slashed dZ^{m_1}\check L^i\|_{t,u}\bigr)+\delta^{\nu-1}\sqrt{\mathcal{E}_{1,\leq m+2}(t,u)}\nonumber\\
&+\delta^{\nu}\delta^{l_1+k_1(2-\nu)}\bigl(\|\slashed{\mathcal L}_Z^{m_1}\leftidx{^{(T)}}{\slashed\pi}_{\mathring L}\|_{t,u}+\|\slashed{\mathcal L}_Z^{m_1}\leftidx{^{(T)}}{\slashed\pi}\|_{t,u}\bigr)+\delta^{\nu-1}\sqrt{\mathcal{E}_{2,\leq m+2}(t,u)}\nonumber\\
\lesssim&(1+\delta^{2\nu-3})\delta^{l_1+k_1(2-\nu)}\|Z^{m_1}\operatorname{tr}{\check{\mathscr X}}\|_{t,u}+\delta^{2\nu-2}(1+\delta^{2\nu-3})\delta^{l_0+k_0(2-\nu)} \| Z^{m_0}\mu \|_{t,u}\label{LX}\\
&+\delta^{2\nu-\f52}+\mu_{\min}^{-\f12}\sqrt{\mathcal{E}_{1,\leq m+2}(t,u)}+\delta^{\nu-1}\sqrt{\mathcal{E}_{2,\leq m+2}(t,u)}.\nonumber
\end{align}
Similarly, with the help of \eqref{lmu}, we have
\begin{align}
&\delta^{l+k(2-\nu)}\|\mathring LZ^{m+1}\mu\|_{t,u}\nonumber\\
\lesssim&(1+\delta^{2\nu-3})\delta^{l_0+k_0(2-\nu)}\|Z^{m_0}\mu\|_{t,u}+\delta^{l_0+k_0(2-\nu)}\|\mathring LZ^{m_0}\varphi\|_{t,u}+\delta^{l_0+k_0(2-\nu)}\|TZ^{m_0}\varphi\|_{t,u}\nonumber\\
&+\delta^{l_1+k_1(2-\nu)}\bigl(\delta^{\nu-1}\|\slashed{\mathcal L}_Z^{m_1}\leftidx{^{(R)}}{\slashed\pi}_{T}\|_{t,u}+\|\slashed{\mathcal L}_Z^{m_1}\leftidx{^{(R)}}{\slashed\pi}_{\mathring L}\|_{t,u}+\delta\|\slashed{\mathcal L}_Z^{m_1}\leftidx{^{(T)}}{\slashed\pi}_{\mathring L}\|_{t,u}\bigr)\nonumber\\
&+\delta^{\nu-2}\delta^{l_0+k_0(2-\nu)}\bigl(\|Z^{m_0}\vp\|_{t,u}+\sum_{i=1}^2(\|Z^{m_0}y^i\|_{t,u}+\|Z^{m_0}\check L^i\|_{t,u})\bigr)\nonumber\\
&+\delta^{\nu-1}(1+\delta^{3\nu-4})\delta^{l_0+k_0(2-\nu)}\|\slashed{\mathcal L}_Z^{m_0}\slashed g\|_{t,u}\nonumber\\
\lesssim&(1+\delta^{2\nu-3})\delta^{l_0+k_0(2-\nu)}\|Z^{m_0}\mu\|_{t,u}+\delta^{\nu-\f32}+\delta^{\nu-2}\delta^{l_1+k_1(2-\nu)}\|Z^{m_1}\operatorname{tr}{\check{\mathscr X}}\|_{t,u}\label{Lmu}\\
&+\mu_{\min}^{-\f12}\sqrt{\mathcal{E}_{1,\leq m+2}(t,u)}+\sqrt{\mathcal{E}_{2,\leq m+2}(t,u)}.\nonumber
\end{align}
By substituting \eqref{Ff} into \eqref{LX} and \eqref{Lmu}, and applying Gronwall's inequality, one obtains the following estimates:
\begin{align}
\delta^{l+k(2-\nu)} \| Z^{m} \, \operatorname{tr}\check{\mathscr{X}} \|_{t,u} \les \delta^{\nu-\frac{1}{2}} + \Theta^1(t,u), \label{ZX} \\
\delta^{l+k(2-\nu)} \| Z^{m+1} \mu \|_{t,u} \les \delta^{\frac{1}{2}} + \Theta^2(t,u). \label{m}
\end{align}
Substituting \eqref{ZX} and \eqref{m} into \eqref{Zk1x}-\eqref{ZkTLpi} completes the proof of Proposition \ref{L2chi}.
\end{proof}

\subsection{Top order $L^2$ estimates for the derivatives of $\check{\mathscr X}$ and $\mu$}\label{Section 6.2}

It is clear from the energy estimate \eqref{e} and the expression \eqref{Psi} that the highest order derivatives of $\varphi$ in the energy estimate for \eqref{Psi} are of order $2N-4$. On the other hand, by the expression of $\leftidx{^{(Z)}}D_{\al,2}^k$ in \eqref{T2}, \eqref{rL2}, and \eqref{R2}, the highest{-}order derivatives of $\check{\mathscr{X}}$ and $\mu$ in \eqref{e} are of orders $2N-5$ and $2N-4$, respectively. However, it follows from Proposition \ref{L2chi} that the $L^2$ estimates for the $(2N-5)$-th order derivatives of $\check{\mathscr{X}}$ and the $(2N-4)$-th order derivatives of $\mu$ should be controlled by the energies of $\vp$ with orders up to $2N-3$, which exceeds the range of energies (containing up to $(2N-4)$-th order derivatives of $\vp$). To overcome this difficulty, as in \cite{Ch1, M-Y, Sp}, one needs to treat $\operatorname{tr}\,\check{\mathscr{X}}$ and $\slashed{\triangle}\mu$ together with their corresponding highest order derivatives.

\subsubsection{Estimates for the derivatives of $\slashed d\operatorname{tr}\check{\mathscr{X}}$}\label{trchi}
If there is at least one vector field $\varrho\mathring L$ in $Z^m$, that is, $Z^m=Z^{p_1}(\varrho\mathring L)\bar Z^{p_2}$ and $p_1+p_2=m-1$ with $\bar Z\in\{T,R\}$, then by \eqref{Lchi'}, Proposition \ref{L2chi} and \eqref{eq:mu_integrability}, we have
\begin{align}
&\delta^{l+(k+1)(2-\nu)}\|\mu\slashed dZ^m(\operatorname{tr}\check{\mathscr{X}})\|_{t,u}\nonumber\\
=&\delta^{l+(k+1)(2-\nu)}\|\slashed dZ^{p_1}\bar Z^{p_2}(\varrho\mathring L) (\operatorname{tr}\check{\mathscr{X}})
+\slashed dZ^{p_1}[\varrho\mathring L,\bar Z^{p_2}](\operatorname{tr}\check{\mathscr{X}})\|_{t,u}\nonumber\\
\lesssim&\delta^{2-\nu}(1+\delta^{2\nu-3})\delta^{l_1+k_1(2-\nu)}\|Z^{m_1}(\operatorname{tr}\check{\mathscr{X}})\|_{t,u}+\delta^{l_0+k_0(2-\nu)}\|\mu\slashed dZ^{m_0}\varphi\|_{t,u}\nonumber\\
&+\delta^{\nu-1}\Big\{\delta^{l_1+k_1(2-\nu)}\Big(\|\slashed{\mathcal L}_Z^{m_1}\slashed g\|_{t,u}+\sum_{i=1}^2\|Z^{m_1}\check L^i\|_{t,u}\Big)+\delta^{l_0+k_0(2-\nu)}\|Z^{k_0}y\|_{t,u}\Big\}\nonumber\\
&+\delta^{l_0+k_0(2-\nu)}\|Z^{m_0}\varphi\|_{t,u}+\delta^{4-2\nu}\delta^{l_1+k_1(2-\nu)}(\|\slashed{\mathcal L}_Z^{m_1}\leftidx^{{(R)}}\slashed\pi_{\mathring L}\|_{t,u}+\delta\|\slashed{\mathcal L}_Z^{m_1}\leftidx^{{(T)}}\slashed\pi_{\mathring L}\|_{t,u})\nonumber\\
\lesssim&\delta^{\nu-\f12}+\mu_{\min}^{-b_{m+2}}\sqrt{\widetilde{\mathcal E}_{1,\leq m+2}(t,u)}+\delta\mu_{\min}^{-b_{m+2}}\sqrt{\widetilde{\mathcal E}_{2,\leq m+2}(t,u)},\label{ZmX}
\end{align}
where $l$ (resp. $k$) is the number of $T$ (resp. $\varrho\mathring L$) appearing in $Z^{p_1}\bar Z^{p_2}$.

It follows from \eqref{Lchi'} that
 \begin{equation}\label{Ltrchi}
\begin{split}
\mathring{L}\big(\operatorname{tr}&\check{\mathscr X}\big)
= -\big(\frac{1}{2}G_{\mathring{L}\mathring{L}}^\al\mathring{L}\varphi_\al
+ G_{\tilde{T}\mathring{L}}^\al\mathring{L}\varphi_\al
- G_{X\mathring{L}}^\al\slashed{d}^X\varphi_\al
+ \slashed{g}^{XX}G_{XX}^\al\mathring{L}\vp_\al\big)
\bigl(\operatorname{tr}\check{\mathscr X} + \varrho^{-1}\bigr) \\
&-G_{X\mathring{L}}^\al\bigl(\slashed{d}^X\mathring{L}\varphi_\al\bigr)
+ \frac{1}{2}G_{\mathring{L}\mathring{L}}^\al\slashed{\triangle}\varphi_\al
- \bigl(\operatorname{tr}\check{\mathscr X}\bigr)^2-2\varrho^{-1}\operatorname{tr}\check{\mathscr X}+f\bigl(\varphi, \slashed{d}y, \mathring{L}^1, \mathring{L}^2\bigr)\slashed{d}\vp
\begin{pmatrix} \mathring{L}\varphi \\ \slashed{d}\varphi \end{pmatrix},
\end{split}
\end{equation}
where
\begin{equation}\label{triphi}
\begin{split}
&\slashed\triangle\varphi_\al=\mu^{-1}\big(\mathring L\mathring{\underline L}\varphi_\al+\f{1}{2\varrho}\mathring {\underline L}\varphi_\al-{H}_\al\big)
\end{split}
\end{equation}
by \eqref{fequation}. In addition, one has that from \eqref{triphi}
\begin{equation*}
\begin{split}
-\mu G_{X\mathring L}^\al\slashed d^X\mathring L\varphi_\al=&-\mathring L(\mu G_{X\mathring L}^\al\slashed d^X\varphi_\al)-\mu(\operatorname{tr}\check{\mathscr X}+\varrho^{-1})G_{X\mathring L}^\al\slashed d^X\varphi_\al\\
&+\f12(G_{\mathring L\mathring L}^\al T\vp_\al)(G_{X\mathring L}^\beta\slashed d^X\vp_\beta)+\mu f\bigl(\varphi, \slashed{d}y, \mathring{L}^1, \mathring{L}^2\bigr)\slashed{d}\vp
\begin{pmatrix} \mathring{L}\varphi \\ \slashed{d}\varphi \end{pmatrix}
\end{split}
\end{equation*}
and
\begin{equation*}
\begin{split}
\f12\mu G_{\mathring L\mathring L}^\al\slashed\triangle\vp_\al=L(\f12G_{\mathring L\mathring L}^\al\underline L\vp_\al)+\f12(G_{\mathring L\mathring L}^\al T\vp_\al)(\operatorname{tr}\check{\mathscr X}+\varrho^{-1})+f\bigl(\varphi, \slashed{d}y, \mathring{L}^1, \mathring{L}^2\bigr)
\begin{pmatrix} \mathring{L}\varphi \\ \slashed{d}\varphi \end{pmatrix}\begin{pmatrix} \mu\mathring{L}\varphi \\ \mu\slashed{d}\varphi\\T\vp \end{pmatrix},
\end{split}
\end{equation*}
then
\begin{equation}\label{Ltr}
\begin{split}
\mathring L\big(\mu\operatorname{tr}\check{\mathscr X}-E)=\big(-\f{2\mu}{\varrho}+\mathcal E\big)\operatorname{tr}\check{\mathscr X}+\f{\mathcal E-\mathring L\mu}{\varrho}-\mu(\operatorname{tr}\check{\mathscr X})^2+e,
\end{split}
\end{equation}
where
\begin{equation}\label{Ee}
\begin{split}
E=&-\mu G_{X\mathring L}^\al\slashed d^X\varphi_\al+\f12G_{\mathring L\mathring L}^\al\mathring{\underline L}\varphi_\al,\\
\mathcal E=&2\mathring L\mu-\mu\slashed g^{XX}G_{XX}^\al\mathring L\vp_\al,\\
e=&f\bigl(\varphi, \slashed{d}y, \mathring{L}^1, \mathring{L}^2\bigr)\begin{pmatrix} \mathring{L}\varphi \\ \slashed{d}\varphi \end{pmatrix}\begin{pmatrix} \mu\mathring{L}\varphi \\ \mu\slashed{d}\varphi\\T\vp \end{pmatrix}.
\end{split}
\end{equation}
Let $F^m=\mu\slashed d\bar Z^m\operatorname{tr}\check{\mathscr X}-\slashed d\bar Z^m E$ with $\bar Z\in\{T,R\}$. Then by an induction
argument on \eqref{Ltr}, we can get
\begin{equation}\label{LFal}
\begin{split}
\slashed{\mathcal L}_{\mathring L}F^m=&(-\f{2}{\varrho}+\mu^{-1}\mathcal E)F^m+(-\f{2}{\varrho}+\mu^{-1}\mathcal E)\slashed d\bar Z^m E-\mu\slashed d\bar Z^{m}\bigl((\operatorname{tr}\check{\mathscr X})^2\bigr)+e^m,
\end{split}
\end{equation}
where
\begin{align}
e^m=&\slashed{\mathcal L}_{\bar Z}^m e^0+\underbrace{\sum_{p_1+p_2=m-1}{\tiny {\tiny {\tiny {\tiny }}}}\slashed{\mathcal L}_{\bar Z}^{p_1}\slashed{\mathcal L}_{[\mathring L,\bar Z]}F^{p_2}}_{\text{vanish when $m=1$}}\nonumber\\
&-\underbrace{\sum_{p_1+p_2=m-1}{\tiny {\tiny {\tiny {\tiny }}}}\slashed{\mathcal L}_{\bar Z}^{p_1}\bigl\{\leftidx^{{(\bar Z)}}\slashed\pi_{\mathring L}^X\slashed d_X\mu+\bar Z(\mathring L\mu+\f{2\mu}{\varrho}-\mathcal E)\bigr\}\slashed d\bar Z^{p_2}\operatorname{tr}\check{\mathscr X}}_{\text{vanish when $m=1$}}\label{eal}\\
&-\underbrace{\sum_{p_1+p_2=m-1}{\tiny {\tiny {\tiny {\tiny }}}}\slashed{\mathcal L}_{\bar Z}^{p_1}\bigl\{(\bar Z\mu)\Bigl(\slashed d\mathring L\bar Z^{p_2}\operatorname{tr}\check{\mathscr X}+\slashed d\bar Z^{p_2}\bigl((\operatorname{tr}\check{\mathscr X})^2\bigr)\Bigr)\bigr\}}_{\text{vanish when $m=1$}}\nonumber
\end{align}
and
\begin{equation}\label{e0}
\begin{split}
e^0=&\slashed de+(\slashed d\mathcal E-\slashed d\mathring L\mu)(\operatorname{tr}\check{\mathscr X}+\varrho^{-1})-\slashed d\mu(\operatorname{tr}\check{\mathscr X})^2-\f{\slashed d\mu}{\varrho^2}\mathring L(\varrho^2\operatorname{tr}\check{\mathscr X}).
\end{split}
\end{equation}
Note that for any one form $\xi$ on $S_{t,u}$, it holds that
\begin{equation}\label{Lxi}
\mathring{L}\bigl(\varrho^2|\xi|^2\bigr)
= -2\varrho^2\bigl(\operatorname{tr}\check{\mathscr{X}}
+ \frac{1}{2}\slashed{g}^{XX}G_{XX}^\al\mathring{L}\vp_\al\bigr)|\xi|^2
+ 2\varrho^2\slashed{g}^{XX}\bigl(\slashed{\mathcal{L}}_{\mathring{L}}\xi_X\bigr)\xi_X.
\end{equation}
By taking $\xi = \varrho^2 F^m$ in \eqref{Lxi} and using \eqref{LFal}, one can obtain
\begin{equation*}
\begin{split}
\mathring{L}\bigl(\varrho^6|F^m|^2\bigr)
&= 2\varrho^6\big\{ \bigl(-\operatorname{tr}\check{\mathscr{X}}
-\frac{3}{2}\slashed{g}^{XX}G_{XX}^\al\mathring{L}\vp_\al
+2\mu^{-1}\mathring L\mu\bigr)|F^m|^2 + e^m \cdot F^m \\
& \quad + \big(-\frac{2}{\varrho} + \mu^{-1}\mathcal{E}\big)\slashed{d}\bar{Z}^m E \cdot F^m
- \mu\slashed{d}\bar{Z}^m\big((\operatorname{tr}\check{\mathscr{X}})^2\big) \cdot F^m \big\}.
\end{split}
\end{equation*}
Since $L\mu \sim -\delta^{\nu-2}$ when $\mu < \frac{1}{10}$ (see \eqref{blowup}), we have $\mu^{-1}\mathring{L}\mu
\leq 10|\mathring{L}\mu| \lesssim \delta^{\nu-2}$. This implies
\[
\mathring{L}\bigl(\varrho^3|F^m|\bigr) \lesssim \delta^{\nu-2}\varrho^3|F^m|
+ \mu^{-1}\delta^{\nu-2}\bigl|\slashed{d}\bar{Z}^m E\bigr|
+ \mu\bigl|\slashed{d}\bar{Z}^m\bigl((\operatorname{tr}\check{\mathscr{X}})^2\bigr)\bigr|
+ |e^m|,
\]
and hence
\begin{equation}\label{LrhoF}
\begin{split}
\varrho^3|F^m(t,u,\vartheta)|\lesssim\varrho_0^3|F^m(t_0,u,\vartheta)|+\int_{t_0}^t\big\{&\delta^{\nu-2}\varrho^3|F^m|+ \mu^{-1}\delta^{\nu-2}|\slashed{d}\bar{Z}^m E| \\
&+ \mu|\slashed{d}\bar{Z}^m\big((\operatorname{tr}\check{\mathscr{X}})^2\big)|
+ |e^m|\}(\tau,u,\vartheta)d\tau.
\end{split}
\end{equation}
It follows from \eqref{Ff} and \eqref{LrhoF} that
\begin{equation*}
\begin{split}
\delta^l\|F^m\|_{t,u}\lesssim\delta^{\nu-\f32}+\delta^l\int_{t_0}^t\big\{&\delta^{\nu-2}\|F^m\|_{\tau,u}+\mu_{\min}^{-1}(\tau)\delta^{\nu-2}\|\slashed d\bar Z^m E\|_{\tau,u}\\
&+\|\mu\slashed d\bar Z^m\bigl((\operatorname{tr}\check{\mathscr{X}})^2\bigr)\|_{\tau,u}+\| e^m\|_{L^2(\Sigma_{\tau}^u)}\big\}d\tau,
\end{split}
\end{equation*}
where $l$ is the number of $T$ in $F^m$. This, together with Gronwall's inequality, yields
\begin{equation}\label{Eal}
\begin{split}
\delta^l\|F^m\|_{t,u}\lesssim\delta^{\nu-\f32}+\delta^l\int_{t_0}^t\big\{&\mu_{\min}^{-1}(\tau)\delta^{\nu-2}\|\slashed d\bar Z^m E\|_{\tau,u}\\
&+\|\mu\slashed d\bar Z^m\bigl((\operatorname{tr}\check{\mathscr{X}})^2\big)\|_{\tau,u}+\| e^m\|_{L^2(\Sigma_{\tau}^u)}\big\}d\tau.
\end{split}
\end{equation}
Each term in the integrand of \eqref{Eal} will be estimated as follows.
\vskip 0.2 true cm
{\bfseries (1) The estimate of $\|\slashed d\bar Z^m E\|_{t,u}$}
\vskip 0.2 true cm
It follows from $Ty^i=\mu(-g^{0i}-\check L^i-\f{y^i}\varrho)$ and $Ry^a=\Omega^a+\upsilon(g^{0a}+\check L^a+\f{y^a}\varrho)$ that
\[
\delta^l|\bar Z^{m+2}y^a|\lesssim1+\delta^{l_0}|\big\{|\bar Z^{m_0}\nu|+\delta^{\nu-1}|\bar Z^{m_0}\check L^1|+\delta^{\nu-1}|\bar Z^{m_0}\check L^2|+\delta^{\nu-1}|\bar Z^{m_0}\vp|+\delta|\bar Z^{m_0}\mu|\big\},
\]
then
\begin{equation*}
\begin{split}
\delta^l\|\slashed d\bar Z^m(G_{\mathring L\mathring L}^\gamma T\vp_{\gamma})\|_{t,u}
\lesssim&\delta^{l_0}\|\bar Z^{m_0}T\vp\|_{t,u}+\delta^{\nu-2}\delta^{l_0}\big\{\|\bar Z^{m_0}\vp\|_{t,u}+\|\bar Z^{m_0}\check L^1\|_{t,u}\\
&+\|\bar Z^{m_0}\check L^2\|_{t,u}+\|\bar Z^{m_0}\nu\|_{t,u}+\delta\|\bar Z^{m_0}\mu\|_{t,u}\big\}.
\end{split}
\end{equation*}
Due to $E=-\mu G_{X\mathring L}^\al\slashed d^X\varphi_\al+\f12G_{\mathring L\mathring L}^\al\mathring{\underline L}\varphi_\al$, then it holds that
\begin{align}
&\delta^l\|\slashed d\bar Z^m E\|_{t,u}\nonumber\\
\lesssim&\delta^{\nu-2}\delta^{l_0}\big\{\|\bar Z^{m_0}\check L^1\|_{t,u}+\|\bar Z^{m_0}\check L^2\|_{t,u}+\|\bar Z^{m_0}\upsilon\|_{t,u}+\|\bar Z^{m_0}\vp\|_{t,u}\big\}+\delta^{l_1}\|\bar Z^{m_1}\mathring L\vp\|_{t,u}\nonumber\\
&+\delta^{\nu-1}(1+\delta^{3\nu-4})\delta^{l_0}\|\slashed{\mathcal L}_{\bar Z}^{m_0}\slashed g\|_{t,u}
+\delta^{l_0}\|T\bar Z^{m_0}\varphi\|_{t,u}+\delta^{l_0}\|\mu\slashed d\bar Z^{m_0}\varphi\|_{t,u}+\delta^l\|\mu\mathring LR\bar Z^m\vp\|_{t,u}\nonumber\\
&+\delta^{2\nu-3}\delta^{l_0}\|\bar Z^{m_0}\mu\|_{t,u}+\de^{\nu-1}\de^{l_1}\big\{\|\slashed{\mathcal L}_{\bar Z}^{m_1}\leftidx{^{(R)}}{\slashed\pi}_T\|_{t,u}+\|\slashed{\mathcal L}_{\bar Z}^{m_1}\leftidx{^{(R)}}{\slashed\pi}_{\mathring L}\|_{t,u}+\delta\|\slashed{\mathcal L}_{\bar Z}^{m_1}\leftidx{^{(T)}}{\slashed\pi}_{\mathring L}\|_{t,u}\big\}\nonumber\\
\lesssim&\delta^{2\nu-\f52}+\sqrt{\mathcal E_{1,\leq m+2}(t,u)}+\sqrt{\mathcal E_{2,\leq m+2}(t,u)}\label{dZE}\\
&+\delta^{\nu-2}\int_{t_0}^t\mu_{\min}^{-\f12}(\tau)\sqrt{\mathcal E_{1,\leq m+2}(\tau,u)}d\tau+\delta^{2\nu-3}\int_{t_0}^t\sqrt{\mathcal E_{2,\leq m+2}(\tau,u)}d\tau,\nonumber
\end{align}
this, together with \eqref{eq:mu_integrability}, yields
\begin{equation}\label{sdE}
\begin{split}
&\int_{t_0}^t\mu_{\min}^{-1}(\tau)\delta^{\nu-2}\delta^l\|\slashed d\bar Z^m E\|_{\tau,u}d\tau\\
\lesssim&\delta^{3\nu-\f92}\int_{t_0}^t\mu_{\min}^{-1}(\tau)d\tau+\f{\mu_{\min}^{-b_{m+2}}(t)}{b_{m+2}}\Big\{\sqrt{\widetilde{\mathcal{E}}_{1,\leq m+2}(t,u)}+\sqrt{\widetilde{\mathcal{E}}_{2,\leq m+2}(t,u)}\Big\}.
\end{split}
\end{equation}
\vskip 0.2 true cm
{\bfseries (2) The estimate of $\|\mu\slashed d\bar Z^m\bigl((\operatorname{tr}\check{\mathscr{X}})^2\bigr)\|_{t,u}$}
\vskip 0.2 true cm
Due to
\[
\mu\slashed d\bar Z^m((\operatorname{tr}\check{\mathscr{X}})^2)
=2\operatorname{tr}\check{\mathscr{X}}(F^m+\slashed d\bar Z^m E)+2\sum_{p\leq m-1}\mu(\bar Z^{m-p}\operatorname{tr}\check{\mathscr{X}})(\slashed d\bar Z^p\operatorname{tr}\check{\mathscr{X}}),
\]
then, together with Proposition \ref{L2chi} and \eqref{sdE}, we have
\begin{equation*}
\begin{split}
&\delta^l\|\mu\slashed d\bar Z^m((\operatorname{tr}\check{\mathscr{X}})^2)\|_{t,u}\\
\lesssim&\delta^{\nu-1}\delta^l\|F^m\|_{t,u}+\delta^{\nu-1}\delta^l\|\slashed d\bar Z^m E\|_{t,u}+\delta^{\nu-1}\delta^{l_1}\|\bar Z^{m_1}\operatorname{tr}\check{\mathscr{X}}\|_{t,u}\\
\lesssim&\delta^{\nu-1}\delta^l\|F^m\|_{t,u}+\delta^{3\nu-\f72}+\delta^{\nu-1}\Big(\sqrt{\mathcal E_{1,\leq m+2}(t,u)}+\sqrt{\mathcal E_{2,\leq m+2}(t,u)}\Big)\\
&+\delta^{2\nu-3}\int_{t_0}^t\mu_{\min}^{-\f12}(\tau)\sqrt{\mathcal E_{1,\leq m+2}(\tau,u)}d\tau+\delta^{3\nu-4}\int_{t_0}^t\sqrt{\mathcal E_{2,\leq m+2}(\tau,u)}d\tau,
\end{split}
\end{equation*}
and hence,
\begin{equation}\label{Y-19}
\begin{split}
&\int_{t_0}^t\delta^l\|\mu\slashed d\bar Z^m((\operatorname{tr}\check{\mathscr{X}})^2)\|_{\tau,u}d\tau\\
\lesssim & \int_{t_0}^t\delta^{\nu-1}\delta^l\|F^m\|_{\tau,u}d\tau+\delta^{2\nu-\f32}+\delta\mu_{\min}^{-b_{m+2}}\Big\{\sqrt{\widetilde{\mathcal{E}}_{1,\leq m+2}(t,u)}+\sqrt{\widetilde{\mathcal{E}}_{2,\leq m+2}(t,u)}\Big\}.
\end{split}
\end{equation}
\vskip 0.2 true cm
{\bfseries (3) The estimate of $\|e^m\|_{t,u}$}
\vskip 0.2 true cm
It follows from \eqref{Ee} that
\begin{align}
&\delta^l\|\slashed{d}\bar{Z}^m e\|_{t,u}\nonumber \\
\lesssim &\, \delta^{\nu-2}\delta^{l_0}\bigl(\|\slashed{d}\bar{Z}^{m_0}\vp\|_{t,u}
+\|\mathring{L}\bar{Z}^{m_0}\vp\|_{t,u}
+\delta^{\nu-1}\|T\bar{Z}^{m_0}\vp\|_{t,u}\bigr)
+\delta^{3\nu-4}\delta^{l_0}\|\bar{Z}^{m_0}\mu\|_{t,u} \nonumber\\
&+\delta^{2\nu-3}\delta^{l_1}\bigl(\delta^{\nu-1}\|\slashed{\mathcal{L}}_{\bar{Z}}^{m_1}\leftidx{^{(R)}}{\slashed{\pi}}_T\|_{t,u}
+\|\slashed{\mathcal{L}}_{\bar{Z}}^{m_1}\leftidx{^{(R)}}{\slashed{\pi}}_{\mathring{L}}\|_{t,u}
+\delta\|\slashed{\mathcal{L}}_{\bar{Z}}^{m_1}\leftidx{^{(T)}}{\slashed{\pi}}_{\mathring{L}}\|_{t,u}\bigr) \nonumber\\
&+\delta^{\nu-2}(1+\delta^{3\nu-4})\delta^{l_0}\|\bar{Z}^{m_0}\vp\|_{t,u}
+\delta^{3\nu-5}\delta^{l_0}\Bigl(\sum_{i=1}^2\|\bar{Z}^{m_0}\check{L}^i\|_{t,u}
+\|\bar{Z}^{m_0}\upsilon\|_{t,u}\Bigr)\nonumber \\
&+\delta^{2\nu-3}(1+\delta^{4\nu-5})\delta^{l_0}\|\slashed{\mathcal{L}}_{\bar{Z}}^{m_0}\slashed{g}\|_{t,u}\nonumber \\
\lesssim &\, \delta^{\nu-\f12}(1+\delta^{3\nu-5})
+\delta^{\nu-2}\mu_{\min}^{-1/2}\sqrt{\mathcal{E}_{1,\leq m+2}(t,u)}
+\delta^{\nu-1}(1+\delta^{2\nu-3})\sqrt{\mathcal{E}_{2,\leq m+2}(t,u)}  \label{de}\\
&+\delta^{3\nu-5}\int_{t_0}^t\mu_{\min}^{-\f12}(\tau)\sqrt{\mathcal{E}_{1,\leq m+2}(\tau,u)}\,\mathrm{d}\tau
+\delta^{4\nu-6}\int_{t_0}^t\sqrt{\mathcal{E}_{2,\leq m+2}(\tau,u)}\,\mathrm{d}\tau.\nonumber
\end{align}
In addition, \eqref{lmu}, \eqref{Lchi'} and \eqref{Ee} give that
\begin{align}
&\delta^l\Bigl\|\slashed{\mathcal{L}}_{\bar{Z}}^m\Bigl\{(\slashed{d}\mathcal{E}-\slashed{d}\mathring{L}\mu)
(\operatorname{tr}\check{\mathscr{X}}+\varrho^{-1})
-\slashed{d}\mu(\operatorname{tr}\check{\mathscr{X}})^2
-\f{\slashed{d}\mu}{\varrho^2}\mathring{L}(\varrho^2\operatorname{tr}\check{\mathscr{X}})\Bigr\}\Bigr\|_{t,u} \nonumber\\
\lesssim&\, \delta^l\bigl(\|\mathring{L}R\bar{Z}^m\vp\|_{t,u}
+\|TR\bar{Z}^m\vp\|_{t,u}
+\|\slashed{d}\bar{Z}^{m+1}\vp\|_{t,u}\bigr)
+\delta^{2\nu-3}\delta^{l_0}\|\bar{Z}^{m_0}\mu\|_{t,u} \nonumber\\
&+\delta^{\nu-1}\delta^{l_1}\bigl(\|\slashed{\mathcal{L}}_{\bar{Z}}^{m_1}\leftidx{^{(R)}}{\slashed{\pi}}_T\|_{t,u}
+\|\slashed{\mathcal{L}}_{\bar{Z}}^{m_1}\leftidx{^{(R)}}{\slashed{\pi}}_{\mathring{L}}\|_{t,u}
+\delta\|\slashed{\mathcal{L}}_{\bar{Z}}^{m_1}\leftidx{^{(T)}}{\slashed{\pi}}_{\mathring{L}}\|_{t,u}
+\|\slashed{\mathcal{L}}_{\bar{Z}}^{m_1}\leftidx{^{(R)}}{\slashed{\pi}}\|_{t,u}\bigr) \nonumber\\
&+\delta^{\nu}\delta^{l_1}\|\slashed{\mathcal{L}}_{\bar{Z}}^{m_1}\leftidx{^{(T)}}{\slashed{\pi}}\|_{t,u}
+\delta^{\nu-2}\delta^{l_0}\Bigl(\sum_{i=1}^2\|\bar{Z}^{m_0}\check{L}^i\|_{t,u}
+\|\bar{Z}^{m_0}\upsilon\|_{t,u}
+\|\bar{Z}^{m_0}\vp\|_{t,u}\Bigr)\nonumber \\
&+\delta^{\nu-2}\delta^{l_1}\|\bar{Z}^{m_1}\operatorname{tr}\check{\mathscr{X}}\|_{t,u}
+\delta^{2\nu-3}\delta^{l_0}\|\slashed{\mathcal{L}}_{\bar{Z}}^{m_0}\slashed{g}\|_{t,u} \nonumber\\
\lesssim&\, \delta^{2\nu-\f52}
+\mu_{\min}^{-1/2}\sqrt{\mathcal{E}_{1,\leq m+2}(t,u)}
+\sqrt{\mathcal{E}_{2,\leq m+2}(t,u)} \label{add-eq}\\
&+\delta^{\nu-2}\int_{t_0}^t\mu_{\min}^{-\f12}(\tau)\sqrt{\mathcal{E}_{1,\leq m+2}(\tau,u)}\,\mathrm{d}\tau
+\delta^{2\nu-3}\int_{t_0}^t\sqrt{\mathcal{E}_{2,\leq m+2}(\tau,u)}\,\mathrm{d}\tau.\nonumber
\end{align}
Due to $\slashed{\mathcal L}_{[\mathring L,\bar Z]}F^{p_2}=\leftidx^{{(\bar Z)}}{\slashed\pi_{\mathring L}}^X\slashed\nabla_XF^{p_2}+F^{p_2}_X\slashed\nabla\leftidx^{{(\bar Z)}}{\slashed\pi_{\mathring L}}^X$, then
\begin{align}
&\sum_{p_1+p_2=m-1}\delta^l\|\slashed{\mathcal L}_{\bar Z}^{p_1}\slashed{\mathcal L}_{[\mathring L, \bar Z]}F^{p_2}\|_{t,u}\nonumber\\
\lesssim&\delta^{\nu-1}(1+\delta^{3\nu-4})\delta^l\|F^m\|_{t,u}+\delta^{\nu-2}\delta^{l_1}\bigl(\|\slashed{\mathcal{L}}_{\bar{Z}}^{m_1}\leftidx{^{(R)}}{\slashed{\pi}}_{\mathring{L}}\|_{t,u}
+\delta\|\slashed{\mathcal{L}}_{\bar{Z}}^{m_1}\leftidx{^{(T)}}{\slashed{\pi}}_{\mathring{L}}\|_{t,u} \bigr)\nonumber\\
&+\delta^{2\nu-3}(1+\delta^{3\nu-4})\delta^{l_1}\|\slashed{\mathcal{L}}_{\bar{Z}}^{m_1}\slashed{g}\|_{t,u}+\delta^{l_1}\|\bar{Z}^{m_1}\operatorname{tr}\check{\mathscr{X}}\|_{t,u} +\delta^{2\nu-3}\delta^{l_1}\|\bar Z^{m_1}\mu\|_{t,u}\nonumber\\
&+\delta^{l_1}\big(\|\slashed d\bar Z^{m_1}\vp\|_{t,u}+\|\bar Z^{m_1}\mathring L\vp\|_{t,u}+\|\bar Z^{m_1}T\vp\|_{t,u}+\delta^{\nu-2}\|\bar Z^{m_1}\vp\|_{t,u}\big)\nonumber\\
&+\delta^{\nu-2}\delta^{l_1}\Bigl(\sum_{i=1}^2\|\bar{Z}^{m_1}\check{L}^i\|_{t,u}
+\|\bar{Z}^{m_1}\upsilon\|_{t,u} \Bigr)\nonumber\\
\lesssim&(1+\delta^{3\nu-4})\Big\{\delta^{\nu-1}\delta^l\|F^m\|_{t,u}+\delta^{2\nu-\f52}+\sqrt{\mathcal E_{1,\leq m+2}(t,u)}+\sqrt{\mathcal E_{2,\leq m+2}(t,u)}\label{LF}\\
&+\delta^{\nu-2}\int_{t_0}^t\mu_{\min}^{-\f12}(\tau)\sqrt{\mathcal E_{1,\leq m+2}(\tau,u)}d\tau+\delta^{2\nu-3}\int_{t_0}^t\sqrt{\mathcal E_{2,\leq m+2}(\tau,u)}d\tau\Big\}.\nonumber
\end{align}
For the remainder terms, it follows from Propositions \ref{prop:higher-order-L-infty-estimates} and \ref{L2chi} that
\begin{align}
&\sum_{p_1+p_2=m-1}\delta^l\|\slashed{\mathcal L}_{\bar Z}^{p_1}\bigl\{\leftidx^{{(\bar Z)}}\slashed\pi_{\mathring L}^X\slashed d_X\mu+\bar Z(\mathring L\mu+\f{2\mu}{\varrho}-\mathcal E)\bigr\}\slashed d\bar Z^{p_2}\operatorname{tr}\check{\mathscr X}\|_{t,u}\nonumber\\
&+\sum_{p_1+p_2=m-1}\delta^l\|\slashed{\mathcal L}_{\bar Z}^{p_1}\bigl\{(\bar Z\mu)\Bigl(\slashed d\mathring L\bar Z^{p_2}\operatorname{tr}\check{\mathscr X}+\slashed d\bar Z^{p_2}\bigl((\operatorname{tr}\check{\mathscr X})^2\bigr)\Bigr)\bigr\}\|_{t,u}\nonumber\\
\lesssim &\delta^{\nu-1}\delta^{l_1}\bigl(\|\slashed{\mathcal{L}}_{\bar{Z}}^{m_1}\leftidx{^{(R)}}{\slashed{\pi}}_{\mathring{L}}\|_{t,u}
+\delta\|\slashed{\mathcal{L}}_{\bar{Z}}^{m_1}\leftidx{^{(T)}}{\slashed{\pi}}_{\mathring{L}}\|_{t,u} \bigr)+\delta^{2\nu-3}\delta^{l_1}\|\bar Z^{m_1}\mu\|_{t,u}\nonumber\\
&+\delta^{\nu-2}\delta^{l_1}\|\bar{Z}^{m_1}\operatorname{tr}\check{\mathscr{X}}\|_{t,u}+\delta^{2\nu-2}(1+\delta^{3\nu-4})\delta^{l_1}\|\slashed{\mathcal{L}}_{\bar{Z}}^{m_1}\slashed{g}\|_{t,u}+\delta^{l_1}\|\bar Z^{m_1}\mathring L\vp\|_{t,u}\nonumber\\
&+\delta^{2\nu-3}\delta^{l_1}\Bigl(\sum_{i=1}^2\|\bar{Z}^{m_1}\check{L}^i\|_{t,u}
+\|\bar{Z}^{m_1}\upsilon\|_{t,u}+\|\bar{Z}^{m_1}\vp\|_{t,u} \Bigr)+\delta^{l_0}\|\bar Z^{m_0}\vp\|_{t,u}\nonumber\\
&+\delta^{\nu-1}\delta^{l_1}\|\bar Z^{m_1}T\vp\|_{t,u}+\delta^{l}\big(\|\mathring LR\bar Z^m\vp\|_{t,u}+\|\bar Z^m\slashed\triangle\vp\|_{t,u}\big)\nonumber\\
\lesssim&\delta^{2\nu-\f52}
+\mu_{\min}^{-1/2}\sqrt{\mathcal{E}_{1,\leq m+2}(t,u)}
+\delta^{\nu-1}\sqrt{\mathcal{E}_{2,\leq m+2}(t,u)} \label{remainder}\\
&+\delta^{\nu-2}\int_{t_0}^t\mu_{\min}^{-\f12}(\tau)\sqrt{\mathcal{E}_{1,\leq m+2}(\tau,u)}\,\mathrm{d}\tau
+\delta^{2\nu-3}\int_{t_0}^t\sqrt{\mathcal{E}_{2,\leq m+2}(\tau,u)}\,\mathrm{d}\tau.\nonumber
\end{align}
Collecting \eqref{de}--\eqref{remainder} and using \eqref{eal}, one has
\begin{equation*}
\begin{split}
\delta^l \|e^m\|_{t,u}
\lesssim& \delta^{\nu-1} \bigl(1 + \delta^{3\nu-4}\bigr) \delta^l \|F^m\|_{t,u}
+ \delta^{2\nu-\frac{5}{2}} \bigl(1 + \delta^{2\nu-3}\bigr) \\
&+ \delta^{\nu-2} \mu_{\min}^{-1/2} \sqrt{\mathcal{E}_{1,\leq m+2}(t,u)}
+ \bigl(1 + \delta^{3\nu-4}\bigr) \sqrt{\mathcal{E}_{2,\leq m+2}(t,u)} \\
&+ \bigl(1 + \delta^{2\nu-3}\bigr) \Bigl\{
\delta^{\nu-2} \int_{t_0}^t \mu_{\min}^{-\frac{1}{2}}(\tau) \sqrt{\mathcal{E}_{1,\leq m+2}(\tau,u)} \,\mathrm{d}\tau \\
&+ \delta^{2\nu-3} \int_{t_0}^t \sqrt{\mathcal{E}_{2,\leq m+2}(\tau,u)} \,\mathrm{d}\tau
\Bigr\},
\end{split}
\end{equation*}
and hence, with the help of \eqref{eq:mu_integrability},
\begin{equation}\label{Y-21}
\begin{split}
\int_{t_0}^t \delta^l \|e^m\|_{\tau,u} \,\mathrm{d}\tau
\lesssim &\delta^{\nu-1} \bigl(1 + \delta^{3\nu-4}\bigr) \int_{t_0}^t \delta^l \|F^m\|_{\tau,u} \,\mathrm{d}\tau
+ \delta^{\nu-\frac{1}{2}} \bigl(1 + \delta^{2\nu-3}\bigr) \\
& + \frac{\mu_{\min}^{\frac{1}{2} - b_{m+2}}(t)}{b_{m+2}} \sqrt{\widetilde{\mathcal{E}}_{1,\leq m+2}(t,u)} + \frac{\delta^{2-\nu} + \delta^{2\nu-2}}{b_{m+2}} \mu_{\min}^{1 - b_{m+1}} \sqrt{\widetilde{\mathcal{E}}_{2,\leq m+2}(t,u)}.
\end{split}
\end{equation}
Inserting \eqref{sdE}, \eqref{Y-19} and \eqref{Y-21} into \eqref{Eal}, and then applying Gronwall's inequality, one has
\begin{equation}\label{Fal}
\begin{split}
\delta^l \| F^m \|_{t,u} \lesssim \,&\delta^{\nu - \frac{3}{2}} + \delta^{3\nu - \frac{9}{2}} \int_{t_0}^t \mu_{\min}^{-1}(\tau) \, d\tau \\
&+ \frac{\mu_{\min}^{-b_{m+2}}(t)}{b_{m+2}} \Big\{ \sqrt{\widetilde{\mathcal{E}}_{1,\leq m+2}(t,u)} + \sqrt{\widetilde{\mathcal{E}}_{2,\leq m+2}(t,u)} \Big\}.
\end{split}
\end{equation}
In addition, by the definition of $F^m$, it holds
\[
\mu \slashed{d} \bar{Z}^m \, \mathrm{tr}\check{\mathscr{X}} = F^m + \slashed{d} \bar{Z}^m E.
\]
Combining inequalities \eqref{Fal} and \eqref{dZE}, we thus have
\begin{equation}\label{d}
\begin{split}
\delta^l \| \mu \slashed{d} \bar{Z}^m \, \mathrm{tr}\check{\mathscr{X}} \|_{t,u} \lesssim \,&\delta^{\nu - \frac{3}{2}} + \delta^{3\nu - \frac{9}{2}} \int_{t_0}^t \mu_{\min}^{-1}(\tau) \, d\tau \\
&+ \mu_{\min}^{-b_{m+2}}(t)\Big\{ \sqrt{\widetilde{\mathcal{E}}_{1,\leq m+2}(t,u)} + \sqrt{\widetilde{\mathcal{E}}_{2,\leq m+2}(t,u)} \Big\}.
\end{split}
\end{equation}
In summary, for any $Z\in\{\varrho\mathring L, T, R\}$, it follows from \eqref{ZmX} and \eqref{d} that
\begin{equation}\label{dnnchi}
\begin{split}
\delta^{l+k(2-\nu)} \| \mu \slashed{d}{Z}^m \, \mathrm{tr}\check{\mathscr{X}} \|_{t,u} \lesssim \,&\delta^{\nu - \frac{3}{2}} + \delta^{3\nu - \frac{9}{2}} \int_{t_0}^t \mu_{\min}^{-1}(\tau) \, d\tau \\
&+ \mu_{\min}^{-b_{m+2}}(t) \Big\{ \sqrt{\widetilde{\mathcal{E}}_{1,\leq m+2}(t,u)} + \sqrt{\widetilde{\mathcal{E}}_{2,\leq m+2}(t,u)} \Big\}.
\end{split}
\end{equation}
\subsubsection{Estimates on the derivatives of $\slashed\triangle\mu$}\label{Sformu}
We next derive the transport equation for $\mu\slashed{\triangle}\mu$ under the derivative of $\mathring{L}$.
Following analogous computations in \cite[Section 4.2]{Ding4}, one has
\begin{equation*}
\begin{split}
\mathring{L}\tilde{F} &= \big(-2\mu\,\mathrm{tr}\check{\mathscr{X}} - \frac{2\mu}{\varrho} + \tilde{\mathcal{E}}\big)\slashed{\triangle}\mu+ \mu\bigl(\slashed{d}^X\mathrm{tr}\check{\mathscr{X}}\bigr)\mathcal{Q}_X + \tilde{e},
\end{split}
\end{equation*}
where
\begin{equation}\label{EEF}
\begin{split}
\tilde{F} &= \mu\slashed{\triangle}\mu - \tilde{E}, \\
\tilde{E} &= -\frac{1}{2}\mu^2 \bigl(G_{\mathring{L}\mathring{L}}^\al + 2G_{\tilde{T}\mathring{L}}^\al\bigr)\slashed{\triangle}\varphi_\al
+ \frac{1}{2}G_{\mathring{L}\mathring{L}}^\al T\underline{\mathring{L}}\varphi_\al, \\
\tilde{\mathcal{E}} &= -\mu G^\al_{\mathring{L}\mathring{L}}\mathring{L}\varphi_\al
- 2\mu G^\al_{\tilde{T}\mathring{L}}\mathring{L}\varphi_\al
+\f12 G_{\mathring{L}\mathring{L}}^\al T\varphi_\al, \\
\mathcal{Q}_X &= -\slashed{d}_X\mu
- \mu G^\al_{X\tilde{T}}\mathring{L}\varphi_\al
- \mu G^{\al}_{\tilde{T}\mathring{L}}\slashed{d}_X\varphi_\al
- \mu G^\al_{\mathring{L}\mathring{L}}\slashed{d}_X\varphi_\al
+ G^\al_{X\mathring{L}}T\varphi_\al
\end{split}
\end{equation}
and
\begin{align}\label{te}
\tilde e=&-\mu\slashed d^X\mu\Big\{\slashed d_X\big(\f12\slashed g^{XX}G_{XX}^\al\mathring L\vp_\al+G_{\mathring L\mathring L}^\al\mathring L\vp_\al+2G_{\tilde T\mathring L}^\al\mathring L\vp_\al\big)\no\\
&\qquad+\f12G_{\mathring L\mathring L}^\beta\slashed d_X\vp_\beta\big(G_{X\mathring L}^\al\slashed d^X\vp_\al+G_{X\tilde T}^\al\slashed d^X\vp_\al+\mathrm{tr}\check{\mathscr{X}}+\varrho^{-1}
\big)
\Big\}-\f12\mathring L(G_{\mathring L\mathring L}^\al)T\underline{\mathring L}\varphi_\al\no\\
&-\f12G_{\mathring L\mathring L}^\al\big\{T\mu(\slashed\triangle\varphi_\al)+\leftidx^{{(T)}}{\slashed\pi}_{\mathring L}^X\slashed d_X\underline{\mathring L}\varphi_\al-\f{1}{2\varrho^2}\underline{\mathring L}\varphi_\al-\f{1}{2\varrho}T\underline{\mathring L}\varphi_\al+T H_\al\big\}\no\\
&+f_1(\varphi, \mathring L^1, \mathring L^2,\slashed dy,\varrho,\slashed\triangle y,\operatorname{tr}\check{\mathscr X})\mu\slashed\triangle\varphi\left(
\begin{array}{ccc}
\mu\mathring L\varphi\\
\mu\slashed d\varphi\\
T\varphi\\
\mu
\end{array}
\right)\\
&+f_2(\varphi,\mathring L^1,\mathring L^2,\slashed d\mathring L^1,\slashed d\mathring L^2,\varrho,\slashed\triangle y)\mu\left(
\begin{array}{ccc}
\slashed d\varphi\\
\slashed d\mathring L^1\\
\slashed d\mathring L^2
\end{array}
\right)
\left(
\begin{array}{ccc}
\slashed d\varphi\\
\mathring L\varphi\\
\slashed d\mathring L^1\\
\slashed d\mathring L^2
\end{array}
\right)
\left(
\begin{array}{ccc}
\mu\mathring L\varphi\\
\mu\slashed d\varphi\\
T\varphi
\end{array}
\right)\no\\
&+f_3(\varphi, \mathring L^1, \mathring L^2,\slashed dy,\varrho,\operatorname{tr}\check{\mathscr X})\mu\slashed\triangle y\left(
\begin{array}{ccc}
\mu\mathring L\varphi\\
\mu\slashed d\varphi\\
T\varphi
\end{array}
\right)+f_4(\varphi,\mathring L^1,\mathring L^2,\slashed dy)\mu\left(
\begin{array}{ccc}
\slashed dT\varphi\\
\mu\slashed d\mathring L\varphi
\end{array}
\right)
\left(
\begin{array}{ccc}
\slashed d\varphi\\
\slashed d\mathring L^1\\
\slashed d\mathring L^2
\end{array}
\right).\no
\end{align}
Analogously to the introduction of $F^m$ in Subsection \ref{trchi}, one  introduces
\[
\tilde{F}^m \;=\; \mu \bar{Z}^m \slashed{\triangle} \mu \;-\; \bar{Z}^m \tilde{E}
\]
with $\bar{Z}$ being any vector field in $\{T, R\}$. Then an induction argument yields that for $m \geq 1$,
\begin{equation}\label{LbarF}
\begin{split}
\mathring{L} \tilde{F}^m
&=\big(-2\,\mathrm{tr}\check{\mathscr{X}} \;-\; \frac{2}{\varrho} \;+\; \mu^{-1}\tilde{\mathcal{E}}\big)
\big(\tilde{F}^m \;+\; \bar{Z}^m \tilde{E}\big)+\mu\big(\slashed{d}^X \bar{Z}^m \,\mathrm{tr}\check{\mathscr{X}}\big)\mathcal{Q}_X
\;+\; \tilde{e}^m,
\end{split}
\end{equation}
where
\begin{equation}\label{bare}
\begin{split}
\tilde e^m=&\bar Z^{m}\tilde e+\sum_{m_1+m_2=m-1}{\bar Z}^{m_1}\Big\{-(\bar Z\mu)\mathring L\bar Z^{m_2}\slashed\triangle\mu+(\slashed d^X{\bar Z}^{m_2}\mathrm{tr}\check{\mathscr{X}})\slashed{\mathcal L}_{\bar Z}(\mu\mathcal Q_X)\Big\}\\
&+\sum_{m_1+m_2=m-1}{\bar Z}^{m_1}\Big\{\leftidx^{{(\bar Z)}}{\slashed\pi}_{\mathring L}^X(\mu\slashed d_X\bar Z^{m_2}\slashed\triangle\mu-\slashed d_X\bar Z^{m_2}\tilde E)\\
&\qquad-\bar Z\bigl(\f12\mu G_{\mathring L\mathring L}^\al\mathring L\vp_\al+\mu G_{\tilde T\mathring L}^\al\mathring L\vp_\al+\mu\slashed g^{XX}G_{XX}^\al\mathring L\vp_\al+2\mu\,\mathrm{tr}\check{\mathscr{X}}+\f\mu\varrho\bigr)\bar Z^{m_2}\slashed \triangle\mu\Big\}.
\end{split}
\end{equation}
Note that
\[\mu^{-1}\tilde{\mathcal E}=\mu^{-1}\mathring L\mu-\f12G_{\mathring L\mathring L}^\al\mathring L\vp_\al-G_{\tilde T\mathring L}^\al\mathring L\vp_\al\lesssim\delta^{\nu-2}\]
holds by \eqref{blowup}.
Then it follows from \eqref{LbarF} that
\begin{equation*}
\begin{split}
\mathring{L}\bigl( \varrho^4 |\tilde{F}^m|^2 \bigr)
&= 2\varrho^2 \tilde{F}^m \mathring{L}\bigl( \varrho^2 \tilde{F}^m \bigr) \\
&\lesssim \delta^{\nu-2} \varrho^4 |\tilde{F}^m|^2 + |\varrho^2 \tilde{F}^m| \big(
\mu^{-1} \delta^{\nu-2} |\bar{Z}^m \tilde{E}| + (1+\delta^{2\nu-3}) |\mu \slashed{d} \bar{Z}^m \mathrm{tr}\check{\mathscr{X}}| + |\tilde{e}^m|
\big),
\end{split}
\end{equation*}
which deduces that
\begin{equation}\label{r2F}
\begin{split}
\delta^{l} \varrho^2 |\tilde{F}^m(t)|
&\lesssim \delta^{l} \varrho_0^2 |\tilde{F}^m(t_0)| + \delta^{l} \int_{t_0}^t \big\{
\mu_{\min}^{-1} \delta^{\nu-2} |\bar{Z}^m \tilde{E}| + (1+\delta^{2\nu-3}) |\mu \slashed{d} \bar{Z}^m \mathrm{tr}\check{\mathscr{X}}| + |\tilde{e}^m|
\big\} d\tau.
\end{split}
\end{equation}
Applying \eqref{Ff} to \eqref{r2F} implies
\begin{equation}\label{F}
\begin{split}
\delta^{l} \|\tilde{F}^m\|_{t,u}
&\lesssim \delta^{\nu-\frac{5}{2}} + \delta^{l} \int_{t_0}^t \mu_{\min}^{-1} \delta^{\nu-2} \|\bar{Z}^m \tilde{E}\|_{\tau,u} d\tau \\
&\quad + \delta^{l} \int_{t_0}^t (1+\delta^{2\nu-3}) \|\mu \slashed{d} \bar{Z}^m \mathrm{tr}\check{\mathscr{X}}\|_{\tau,u} d\tau + \delta^{l} \int_{t_0}^t \|\tilde{e}^m\|_{\tau,u} d\tau,
\end{split}
\end{equation}
where
\begin{equation}\label{chi}
\begin{split}
\delta^{l} \int_{t_0}^t (1+\delta^{2\nu-3}) \|&\mu \slashed{d} \bar{Z}^m \mathrm{tr}\check{\mathscr{X}}\|_{\tau,u} d\tau
\lesssim(1+\delta^{2\nu-3}) \Big\{\delta^{1/2}+\delta^{2\nu-\f52}\int_{t_0}^t\mu_{\min}^{-1}(\tau)d\tau\\
&+\f{\delta^{2-\nu}}{b_{m+2}}\mu_{\min}^{1-b_{m+2}}\bigl(\sqrt{\widetilde{\mathcal{E}}_{1,\leq m+2}(t,u)} + \sqrt{\widetilde{\mathcal{E}}_{2,\leq m+2}(t,u)} \bigr)
\Big\}.
\end{split}
\end{equation}
In addition, the expression of $\tilde E$ in \eqref{EEF} gives that
\begin{align}
&\delta^l\|\bar Z^m\tilde E\|_{t,u}\nonumber\\
\lesssim & \de^{l_1}\bigl(\|\mu\slashed\triangle\bar Z^{m_1}\vp\|_{t,u}+\|T\bar Z^{m_1}\mathring{\underline L}\vp\|_{t,u}\|+\delta^{\nu-1}\|\slashed{\mathcal{L}}_{\bar{Z}}^{m_1}\leftidx{^{(R)}}{\slashed{\pi}}\|_{t,u}
+\delta^\nu\|\slashed{\mathcal{L}}_{\bar{Z}}^{m_1}\leftidx{^{(T)}}{\slashed{\pi}}\|_{t,u}\bigr)\nonumber\\
&+\delta^{\nu-3}\delta^{l_1}\Bigl(\sum_{i=1}^2\|\bar Z^{m_1}\check L^i\|_{t,u}+\|\bar Z^{m_1}\upsilon\|_{t,u}+\delta\|\bar Z^{m_1}\mu\|_{t,u}+\delta^{2}(1+\delta^{2\nu-3})\|\bar Z^{m_1}\mathring{\underline L}\vp\|_{t,u}\Bigr)\nonumber\\
&+\delta^{\nu-3}\delta^{l_0}\|\bar Z^{m_0}\vp\|_{t,u}+\delta^{\nu-2}\delta^{l_1}\|\slashed{\mathcal{L}}_{\bar{Z}}^{m_1}\leftidx{^{(R)}}{\slashed{\pi}}_T\|_{t,u} \nonumber\\
\lesssim &\delta^{2\nu-\f72}
+\delta^{\nu-2}\sqrt{\mathcal{E}_{1,\leq m+2}(t,u)}
+\delta^{-1}\sqrt{\mathcal{E}_{2,\leq m+2}(t,u)} \label{ZtE}\\
&+\delta^{\nu-3}\int_{t_0}^t\mu_{\min}^{-\f12}(\tau)\sqrt{\mathcal{E}_{1,\leq m+2}(\tau,u)}\,\mathrm{d}\tau
+\delta^{2\nu-4}\int_{t_0}^t\sqrt{\mathcal{E}_{2,\leq m+2}(\tau,u)}\,\mathrm{d}\tau,\nonumber
\end{align}
which means
\begin{equation}\label{ZmE}
\begin{split}
&\delta^{\nu-2}\int_{t_0}^t\mu_{\min}^{-1}(\tau)\delta^l\|\bar Z^m\tilde E\|_{\tau,u}d\tau\\
\lesssim&\delta^{3\nu-\f{11}2}\int_{t_0}^t\mu_{\min}^{-1}(\tau)d\tau+\f{\delta^{-1}}{b_{m+2}}\mu_{\min}^{-b_{m+2}}
\big(\sqrt{\widetilde{\mathcal{E}}_{1,\leq m+2}(t,u)} + \sqrt{\widetilde{\mathcal{E}}_{2,\leq m+2}(t,u)}\big).
\end{split}
\end{equation}
Since \eqref{te} and \eqref{H} imply that $\bar{Z}^m \tilde{e}$ contains the term
$\frac{1}{2}(G_{\mathring{L}\mathring{L}}^\gamma T\varphi_\gamma) \bar{Z}^m T\operatorname{tr}\check{\mathscr{X}}$,
which equals
\[
\frac{1}{2} (G_{\mathring{L}\mathring{L}}^\gamma T\varphi_\gamma) \bar{Z}^m \slashed{\triangle}\mu
+ \cdots=\frac{1}{2} \mu^{-1}( G_{\mathring{L}\mathring{L}}^\gamma T\varphi_\gamma)(\tilde F^m+\bar Z^m\tilde E)+\cdots
\]
by \eqref{Tchi'}, then it follows from Proposition \ref{L2chi} and \eqref{ZtE} that
\begin{align}
&\delta^l\|\bar Z^m\tilde e\|_{t,u}\nonumber\\
\lesssim&\delta^{2\nu-4}\delta^{l_0}\|\bar Z^{m_0}\mu\|_{t,u}+\delta^l\bigl(\|\mu\slashed d\bar Z^m\mathring L\vp\|_{t,u}+\delta^{-1}\|\slashed\triangle\bar Z^m\vp\|_{t,u}+(1+\delta^{2\nu-3})\|\slashed d\bar Z^m\mathring{\underline L}\vp\|_{t,u}\bigr)\nonumber\\
&+\delta^{\nu-2}\delta^l\|T\bar Z^{m+1}\vp\|_{t,u}+(1+\delta^{2\nu-3})\bigl(\delta^{l_1}\|T\bar Z^{m_1}\mathring{\underline L}\vp\|_{t,u}+\delta^{-2}\delta^{l_0}\|\bar Z^{m_0}\vp\|_{t,u}\bigr)\nonumber\\
&+\delta^{2\nu-3}\delta^{l_1}\|T\bar Z^{m_1}T\vp\|_{t,u}+\delta^{l_1}\|T\bar Z^{m_1}\mathring L\vp\|_{t,u}+\delta^{\nu-2}(1+\delta^{2\nu-3})\delta^{l_1}\|\slashed{\mathcal{L}}_{\bar{Z}}^{m_1}\leftidx{^{(R)}}{\slashed{\pi}}_T\|_{t,u} \nonumber\\
&+\delta^{\nu-2}\delta^{l_1}\bigl(\delta\|\slashed{\mathcal{L}}_{\bar{Z}}^{m_1}\leftidx{^{(T)}}{\slashed{\pi}}\|_{t,u}
+\|\slashed{\mathcal{L}}_{\bar{Z}}^{m_1}\leftidx{^{(T)}}{\slashed{\pi}}_{\mathring{L}}\|_{t,u}
+\|\slashed{\mathcal{L}}_{\bar{Z}}^{m_1}\leftidx{^{(R)}}{\slashed{\pi}}\|_{t,u}\bigr)+\delta^{\nu-3}\delta^{l_1}\|\bar Z^{m_1}\operatorname{tr}\check{\mathscr X}\|_{t,u} \nonumber\\
&+\delta^{\nu-3}(1+\delta^{2\nu-3})\delta^{l_0}\bigl(\sum_{i=1}^2\|\bar Z^{m_0}\check L^i\|{t,u}+\|\bar Z^{m_0}\upsilon\|_{t,u}+\delta\|\slashed{\mathcal{L}}_{\bar{Z}}^{m_0}\slashed g\|_{t,u} \bigr)\nonumber\\
&+\delta^{\nu-2}\delta^{l_1}\|\mu^{-1}(\tilde F^{m_1}+\bar Z^{m_1}\tilde E\|_{t,u})\nonumber\\
\lesssim &\delta^{\nu-2}\mu_{\min}^{-1}\de^l\|\tilde F^m\|_{t,u}+\delta^{3\nu-\f{11}2}\mu_{\min}^{-1}(t)
\label{Zme}\\
&+\delta^{\nu-3}\mu_{\min}^{-1-b_{m+2}}(t)\big(\sqrt{\widetilde{\mathcal{E}}_{1,\leq m+2}(t,u)} + \sqrt{\widetilde{\mathcal{E}}_{2,\leq m+2}(t,u)} \big).\nonumber
\end{align}
Furthermore, when $m_1+m_2=m-1$,
\begin{align}
&\delta^l\|{\bar Z}^{m_1}\big\{-(\bar Z\mu)\mathring L\bar Z^{m_2}\slashed\triangle\mu+(\slashed d^X{\bar Z}^{m_2}\mathrm{tr}\check{\mathscr{X}})\slashed{\mathcal L}_{\bar Z}(\mu\mathcal Q_X)\big\}\|_{t,u}\nonumber\\
&+\delta^l\|{\bar Z}^{m_1}\big\{\leftidx^{{(\bar Z)}}{\slashed\pi}_{\mathring L}^X(\mu\slashed d_X\bar Z^{m_2}\slashed\triangle\mu-\slashed d_X\bar Z^{m_2}\tilde E)\nonumber\\
&-\bar Z\bigl(\f12\mu G_{\mathring L\mathring L}^\al\mathring L\vp_\al+\mu G_{\tilde T\mathring L}^\al\mathring L\vp_\al+\mu\slashed g^{XX}G_{XX}^\al\mathring L\vp_\al+2\mu\,\mathrm{tr}\check{\mathscr{X}}+\f\mu\varrho\bigr)\bar Z^{m_2}\slashed \triangle\mu\big\}\|_{t,u}\nonumber\\
\lesssim &\delta^{\nu-3}\delta^{l_1}\bigl(\|\slashed{\mathcal{L}}_{\bar{Z}}^{m_1}\leftidx{^{(R)}}{\slashed{\pi}}_{\mathring L}\|_{t,u}
+\delta\|\slashed{\mathcal{L}}_{\bar{Z}}^{m_1}\leftidx{^{(T)}}{\slashed{\pi}}_{\mathring{L}}\|_{t,u}
+\delta^2\|\slashed{\mathcal{L}}_{\bar{Z}}^{m_1}\leftidx{^{(R)}}{\slashed{\pi}}_T\|_{t,u}\bigr)\nonumber\\
&+\delta^{\nu-1}(1+\delta^{3\nu-4})\bigl(\delta^{l}\|\bar Z^m\slashed\triangle\mu\|_{t,u}+\delta^{l_1}\|\bar Z^{m-1}\tilde E\|_{t,u}+\delta^{\nu-3}\delta^{l_1}\|\slashed{\mathcal{L}}_{\bar{Z}}^{m_1}\slashed g\|_{t,u}\bigr)\nonumber\\
&+\delta^{\nu-2}\delta^{l_0}\bigl(\|\bar Z^{m_0}\mu\|_{t,u}+\sum_{i=1}^2\|\bar Z^{m_0}\check L^i\|_{t,u}+\|\bar Z^{m_0}\upsilon\|_{t,u}\bigr)+\delta^{l_0}\|T\bar Z^{m_0}\vp\|_{t,u}\nonumber\\
&+(1+\delta^{2\nu-3})\delta^{l_1}\|\bar Z^{m_1}\operatorname{tr}\check{\mathscr X}\|_{t,u}+\delta^{\nu-2}\delta^{l_0}\|\bar Z^{m_0}\vp\|_{t,u}+\delta^{l_0}\|\mu\mathring L\bar Z^{m_0}\vp\|_{t,u}\nonumber\\
\lesssim &(1+\delta^{3\nu-4})\big\{\delta^{\nu-1}\mu_{\min}^{-1}\de^l\|\tilde F^m\|_{t,u}+\de^{2\nu-\f72}\mu_{\min}^{-1}(t)+\delta^{-1}\mu_{\min}^{-1-b_{m+2}}(t)\sqrt{\widetilde{\mathcal{E}}_{1,\leq m+2}(t,u)}\label{eleft}\\
&\qquad\qquad\qquad+\delta^{\nu-2}\mu_{\min}^{-1-b_{m+2}}(t)\sqrt{\widetilde{\mathcal{E}}_{2,\leq m+2}(t,u)}\big\}.\nonumber
\end{align}
By substituting \eqref{Zme} and \eqref{eleft} into \eqref{bare}, one has
\begin{equation}\label{intem}
\begin{split}
\delta^l \int_{t_0}^t \|\tilde{e}^m\|_{\tau,u} \, d\tau
\lesssim \,&\delta^{\nu-2} \int_{t_0}^t \mu_{\min}^{-1}(\tau) \, \dl^l \|\tilde{F}^m\|_{\tau,u} \, d\tau
+ \delta^{3\nu-\frac{11}{2}} \int_{t_0}^t \mu_{\min}^{-1}(\tau) \, d\tau \\
&+ \frac{\delta^{-1}}{b_{m+2}} \, \mu_{\min}^{-b_{m+2}}(t)
\big(\sqrt{\widetilde{\mathcal{E}}_{1,\leq m+2}(t,u)}
+ \sqrt{\widetilde{\mathcal{E}}_{2,\leq m+2}(t,u)}\big).
\end{split}
\end{equation}
In all, using \eqref{chi}, \eqref{ZmE}, \eqref{intem} and \eqref{F}, we obtain
\begin{equation*}
\mu_{\min}^{b_{m+2}} \, \delta^l \|\tilde{F}^m\|_{t,u}
\lesssim \delta^{\nu - \frac{5}{2}}
+ \frac{\delta^{-1}}{b_{m+2}}
\big(\sqrt{\widetilde{\mathcal{E}}_{1,\leq m+2}(t,u)}
+ \sqrt{\widetilde{\mathcal{E}}_{2,\leq m+2}(t,u)}\big).
\end{equation*}
Combining this with \eqref{ZtE} yields
\begin{equation*}
\mu_{\min}^{b_{m+2}} \, \delta^l \|\mu \bar{Z}^m \slashed{\triangle} \mu\|_{t,u}
\lesssim \delta^{\nu - \frac{5}{2}}
+ \delta^{-1} \big( b_{m+2}^{-1} \sqrt{\widetilde{\mathcal{E}}_{1,\leq m+2}(t,u)}
+ \sqrt{\widetilde{\mathcal{E}}_{2,\leq m+2}(t,u)}\big).
\end{equation*}
When $Z^m$ contains at least one $\varrho\mathring L$, similarly to \eqref{ZmX},
we could take use of \eqref{lmu} and estimate $\|\mu{Z}^m \slashed{\triangle} \mu\|_{t,u}$ directly.

In the end, for any $Z\in\{\varrho\mathring L,T,R\}$,
\begin{equation}\label{Zmu}
\mu_{\min}^{b_{m+2}} \, \delta^{l+k(2-\nu)} \|\mu{Z}^m \slashed{\triangle} \mu\|_{t,u}
\lesssim \delta^{\nu - \frac{5}{2}}
+ \delta^{-1} \big( b_{m+2}^{-1} \sqrt{\widetilde{\mathcal{E}}_{1,\leq m+2}(t,u)}
+ \sqrt{\widetilde{\mathcal{E}}_{2,\leq m+2}(t,u)}\big).
\end{equation}

\section{Estimates for the error terms}\label{ert}

With all preparations made for the optimal $L^2$ estimates of the relevant quantities in Section \ref{EE},
we are now ready to treat the error terms $\delta^{\nu-1}\int_{D^{t, u}}|\Phi \cdot \mathring{\underline{L}}\Psi|$
and $\int_{D^{t, u}}|\Phi \cdot \mathring{L}\Psi|$ in \eqref{e} so that the final energy estimates for $\vp$ can be established.
As in \cite{M-Y, L-Y, Sp}, we need to examine all terms in \eqref{Phik} with explicit expressions listed in \eqref{T1}--\eqref{R3}.

\subsection{Treatment of $J_1^{m+1}$ in \eqref{Phik}}\label{tot}
\begin{enumerate}
\item
By the explicit expression of $J_1^{m+1}$, one may directly conclude that the highest order of derivatives involved in $\leftidx{^{(Z)}}D_{\al,i}^{n}$ $(i=1,3)$ does not exceed $(m-n)$. Since neither the corresponding $L^\infty$ estimates in Section \ref{Section 3} nor the relevant $L^2$ estimates in Proposition \ref{L2chi} rely on the top order derivative estimates in Section \ref{Section 6.2},
then the $L^2$ norms  of
all terms contained in $J_1^{m+1}$ can be bounded via the aforementioned results.
Therefore,
\begin{align}
&\delta^{\nu-1}\delta^{2l+2k(2-\nu)}
\big|
\int_{D^{t, u}}
\sum_{p=1}^{m}
\bigl(Z_{m+1}+\leftidx{^{(Z_{m+1})}}\lambda\bigr)
\cdots
\bigl(Z_{m+2-p}+\leftidx{^{(Z_{m+2-p})}}\lambda\bigr)
\leftidx{^{(Z_{m+1-p})}}D_{\al,1}^{m-p}
\cdot\mathring{\underline L}\varphi_\al^{m+1}
\big|
\notag\\
\lesssim&\,
\delta^{\nu-1}\delta^{2l+2k(2-\nu)}
\int_{t_0}^t
\sum_{p=1}^{m}
\bigl\|
\bigl(Z_{m+1}+\leftidx{^{(Z_{m+1})}}\lambda\bigr)
\cdots
\bigl(Z_{m+2-p}+\leftidx{^{(Z_{m+2-p})}}\lambda\bigr)
\leftidx{^{(Z_{m+1-p})}}D_{\al,1}^{m-p}
\bigr\|_{\tau,u}
\notag\\
&\qquad\qquad\qquad\qquad
\cdot\bigl\|\mathring{\underline L}Z^{m+1}\varphi_\al\bigr\|_{\tau,u}\,d\tau
\notag\\
\lesssim&\,
\delta^{\nu-1}
\int_{t_0}^t
\mu_{\min}^{-b_{m+2}}(\tau)
\sqrt{\widetilde{\mathcal{E}}_{2,\leq m+2}(\tau,u)}
\Big\{
\delta^{3\nu-\frac{9}{2}}
+ \delta^{\nu-2}\mu_{\min}^{-1/2-b_{m+2}}(\tau)\sqrt{\widetilde{\mathcal{E}}_{1,\leq m+2}(\tau,u)}
\notag\\
&\qquad\qquad\qquad\qquad\qquad\qquad\qquad
+ \delta^{2\nu-3}\mu_{\min}^{-b_{m+2}}(\tau)\sqrt{\widetilde{\mathcal{E}}_{2,\leq m+2}(\tau,u)}
\Big\}\!d\tau
\notag\\
\lesssim&\,
\delta^{5\nu-6}
+ \frac{\delta^{\nu-1}}{b_{m+2}}
\mu_{\min}^{-2b_{m+2}}(t)
\big({\widetilde{\mathcal{E}}_{1,\leq m+2}(t,u)}
+{\widetilde{\mathcal{E}}_{2,\leq m+2}(t,u)}
\big).\label{D11}
\end{align}
Similarly,
\begin{align}
&\delta^{\nu-1}\delta^{2l+2k(2-\nu)}\big|\int_{D^{t, u}}\sum_{p=1}^{m}\big(Z_{m+1}+\leftidx{^{(Z_{m+1})}}\lambda\big)\dots\big(Z_{m+2-p}
+\leftidx{^{(Z_{m+2-p})}}\lambda\big)\leftidx{^{(Z_{m+1-p})}}D_{\al,3}^{m-p}\cdot\mathring{\underline L}Z^{m+1}\varphi_\al\big|\notag\\
\lesssim&\delta^{3\nu-2}(1+\delta^{4\nu-6})+\frac{\delta^{\nu-1}}{b_{m+2}}
\mu_{\min}^{-2b_{m+2}}(t)
\bigl(\delta^{2\nu-2}\mu_{\min}(t)
{\widetilde{\mathcal{E}}_{1,\leq m+2}(t,u)}
+{\widetilde{\mathcal{E}}_{2,\leq m+2}(t,u)}
\bigr).\label{D33}
\end{align}
The terms in $\delta^{2l+2s(2-\nu)}\int_{D^{t, u}}|\Phi\cdot \mathring L\Psi|$
related to the integrand factors $\leftidx{^{(Z)}}D_{\al,i}^{n}$ $(i=1,3)$ can also be estimated as follows
\begin{align}\label{D13}
&\delta^{2l+2k(2-\nu)}\int_{D^{t, u}}\big|\sum_{p=1}^{m}\big(Z_{m+1}+\leftidx{^{(Z_{m+1})}}\lambda\big)\dots\big(Z_{m+2-p}
+\leftidx{^{(Z_{m+2-p})}}\lambda\big)\big(\leftidx{^{(Z_{m+1-p})}}D_{\al,1}^{m-p}\no\\
&\qquad\qquad\qquad+\leftidx{^{(Z_{m+1-p})}}D_{\al,3}^{m-p}\big)\mathring L\varphi_\al^{m+1}\big|\no\\
\lesssim&\delta^{2l+2k(2-\nu)+1}\int_{D^{t, u}}\Big\{\sum_{p=1}^{m}\big(Z_{m+1}+\leftidx{^{(Z_{m+1})}}\lambda\big)\dots\big(Z_{m+2-p}
+\leftidx{^{(Z_{m+2-p})}}\lambda\big)\big(\leftidx{^{(Z_{m+1-p})}}D_{\al,1}^{m-p}\no\\
&\qquad+\leftidx{^{(Z_{m+1-p})}}D_{\al,3}^{m-p}\big)\Big\}^2+\delta^{2l+2s(2-\nu)-1}\int_{D^{t, u}}|\mathring LZ^{m+1}\varphi_\al|^2\no\\
\lesssim&\delta^{5\nu-6}+\frac{\delta^{\nu-1}}{b_{m+2}}
\mu_{\min}^{-2b_{m+2}}(t)
\bigl(
{\widetilde{\mathcal{E}}_{1,\leq m+2}(t,u)}
+\delta^{2\nu-2}\mu_{\min}(t){\widetilde{\mathcal{E}}_{2,\leq m+2}(t,u)}
\bigr)\\
&+\delta^{-1}\int_0^uF_{1,m+2}(s,u')du'.\no
\end{align}

\item As stated in \cite[Section 5]{Ding4}, we now proceed to estimate the terms in $J_1^{m+1}$ that involve $\leftidx{^{(Z)}}D_{\al,2}^{n}$ for
$0 \leq n \leq m$. Especially, in the case of $n=0$, the index $j$ in $J_1^{m+1}$ equals $m$ and the top-order derivatives in $\leftidx{^{(Z)}}D_{\al,2}^{0}$ are of order $m$. Consequently, $\leftidx{^{(Z)}}D_{\gamma,2}^{0}$ contains the terms including $(m+1)^\text{th}$-order derivatives of the deformation tensor.
This prevents the direct use of Proposition \ref{L2chi} to estimate the $L^2$ norm of $\leftidx{^{(Z)}}D_{\al,2}^0$
since the $L^2$ norm of the $(m+1)^\text{th}$-order derivatives of the deformation tensor is controlled by $\mathcal{E}_{1,\leq m+3}(t,u)$ and $\mathcal{E}_{2,\leq m+3}(t,u)$ in the energy estimates. However, it follows from \eqref{e} that $\mathcal{E}_{1,\leq m+3}(t,u)$ and $\mathcal{E}_{2,\leq m+3}(t,u)$ cannot be directly absorbed by the energies $\mathcal{E}_{1,\leq m+2}(t,u)$ and $\mathcal{E}_{2,\leq m+2}(t,u)$ on the left hand side of the resulting energy inequality. To overcome this difficulty, we will analyze the expression of $\leftidx{^{(Z)}}D_{\al,2}^n$ more carefully and apply the estimates from Subsections \ref{trchi}--\ref{Sformu} to handle the top order derivatives of $\operatorname{tr}{\check{\mathscr X}}$ and $\mu$.

According to the expressions given in \eqref{T2}, \eqref{rL2} and \eqref{R2}, special attention should be {paid} to the terms marked with underlines, wavy lines, boxes, or braces. For $\leftidx{^{(T)}}D_{\al,2}^{n}$, one can utilize the relation $\mathring{\underline L} = \mu\mathring L + 2T$ to obtain the identity $\frac{1}{2}\slashed{\mathcal L}_{\mathring{\underline L}}\slashed d_X\mu = \slashed d_XT\mu + \frac{1}{2}\mu\slashed d_X\mathring L\mu$. From this, the underlined term $-\slashed d_XT\mu + \frac{1}{2}\slashed{\mathcal L}_{\mathring{\underline L}}\slashed d_X\mu$ in $\leftidx{^{(T)}}D_{\al,2}^{n}$ becomes
\begin{equation}\label{Y-25}
-\slashed d_XT\mu + \frac{1}{2}\slashed{\mathcal L}_{\mathring{\underline L}}\slashed d_X\mu = \frac{1}{2}\mu\slashed d_X\mathring L\mu,
\end{equation}
which can be estimated by means of \eqref{lmu}.
From \eqref{Rpi} and \eqref{theta}, it can be inferred that all wavy-lined terms in \eqref{T2}, \eqref{rL2} and \eqref{R2} involve the quantity $\slashed{\mathcal L}_{\mathring{\underline L}}\check{\mathscr X}$. As indicated in \eqref{Tchi'}, $\slashed{\mathcal L}_{\mathring{\underline L}}\check{\mathscr X}$ may be expanded as $2\slashed\nabla^2\mu + \mu\slashed{\mathcal L}_{\mathring{L}}\check{\mathscr X}+ \text{``good terms"}$. Specifically, the terms contained in $\mu\slashed{\mathcal L}_{\mathring{L}}\check{\mathscr X}+\text{``good terms"}$ can be estimated by \eqref{Lchi'} and Proposition \ref{L2chi}, whereas the $L^2$-norm of the terms corresponding to $\slashed\triangle\mu$ is handled
by use of \eqref{Zmu}.
Meanwhile, the boxed terms and braced terms may be estimated through \eqref{Zmu} and \eqref{dnnchi}, respectively.

On the other hand, it is noteworthy that there exist some terms whose factors include the derivatives of deformation tensors with respect to $\mathring L$. A typical example is $\frac{1}{4}\mathring L(\operatorname{tr}\leftidx{^{(T)}}{\slashed\pi})\underline{\mathring L}\varphi_\gamma^n$ appearing in \eqref{T2}. In fact, such terms are not ``bad" since a detailed examination of each term in \eqref{Lpi} and \eqref{Rpi} reveals that the $\mathring L$-derivatives of the involved deformation tensors are associated with the ``good" quantities, which can be treated through \eqref{lmu} and \eqref{Lchi'}.

Therefore, we arrive at
\begin{align}
&\delta^{l+k(2-\nu)}\|\sum_{p=1}^{m}\big(Z_{m+1}+\leftidx{^{(Z_{m+1})}}\lambda\big)\dots\big(Z_{m+2-p}
+\leftidx{^{(Z_{m+2-p})}}\lambda\big)\leftidx{^{(Z_{m+1-p})}}D_{\al,2}^{m-p}\|_{t,u}\nonumber \\
\lesssim &\delta^{l_0+k_0(2-\nu)}\big(\delta^{2\nu-3}(1+\delta^{2\nu-3})\|Z^{m_0}\mu\|_{t,u}+\delta^{2\nu-3}\|TZ^{m_0}\vp\|_{t,u}
+\delta^{\nu-2}\|\mathring LZ^{m_0}\vp\|_{t,u}\nonumber\\
&+(1+\delta^{2\nu-3})\|\slashed dZ^{m_0}\vp\|_{t,u}+\delta^{2\nu-4}\|Z^{m_0}\vp\|_{t,u}
+\delta^{3\nu-5}\sum_{i=1}^2\|Z^{m_0}\check L^i\|_{t,u}+\delta^{2\nu-4}\|Z^{m_0}\upsilon\|_{t,u}\big)\nonumber\\
&+\delta^{\nu-2}\delta^{l_1+k_1(2-\nu)}\big(\|\slashed dZ^{m_1}\operatorname{tr}\check{\mathscr X}\|_{t,u}+\delta\|Z^{m_0}\slashed\triangle\mu\|_{t,u}+(1+\delta^{2\nu-3})\|Z^{m_1}\operatorname{tr}\check{\mathscr X}\|_{t,u}\nonumber\\
&+\delta^{2\nu-2}\|\slashed{\mathcal{L}}_{{Z}}^{m_1}\leftidx{^{(R)}}{\slashed{\pi}}_T\|_{t,u}
+\delta^{\nu}\|\slashed{\mathcal{L}}_{{Z}}^{m_1}\leftidx{^{(T)}}{\slashed{\pi}}_{\mathring L}\|_{t,u}+\delta^{\nu-1}\|\slashed{\mathcal{L}}_{{Z}}^{m_1}\leftidx{^{(R)}}{\slashed{\pi}}_{\mathring L}\|_{t,u}\nonumber\\
&+\delta^{2\nu-1}\|\slashed{\mathcal{L}}_{{Z}}^{m_1}\leftidx{^{(T)}}{\slashed{\pi}}\|_{t,u}\big)\nonumber\\
\lesssim &\delta^{2\nu-\f72}\mu_{\min}^{-1-b_{m+2}}(t)+\delta^{\nu-2}\mu_{\min}^{-1-b_{m+2}}(t)\big({\widetilde{\mathcal{E}}_{1,\leq m+2}(t,u)}
+{\widetilde{\mathcal{E}}_{2,\leq m+2}(t,u)}
\big),\label{Dg2}
\end{align}
which yields
\begin{align}
&\delta^{\nu-1}\delta^{2l+2k(2-\nu)}
\big|
\int_{D^{t, u}}
\sum_{p=1}^{m}
\big(Z_{m+1}+\leftidx{^{(Z_{m+1})}}\lambda\big)
\cdots
\big(Z_{m+2-p}+\leftidx{^{(Z_{m+2-p})}}\lambda\big)
\leftidx{^{(Z_{m+1-p})}}D_{\al,2}^{m-p}
\cdot\mathring{\underline L}\varphi_\al^{m+1}
\big|\notag\\
&\lesssim\f{\delta^{3\nu-4}}{b_{m+2}}\mu_{\min}^{-2b_{m+2}}(t)+\f{\delta^{\nu-1}}{b_{m+2}}\mu_{\min}^{-2b_{m+2}}(t)
\big({\widetilde{\mathcal{E}}_{1,\leq m+2}(t,u)}
+{\widetilde{\mathcal{E}}_{2,\leq m+2}(t,u)}
\big).\label{iDg2}
\end{align}
It should be noted that \eqref{Dg2} is insufficient to close the energy estimate for $\delta^{2l+2s(2-\nu)}\int_{D^{t, u}}|\Phi \cdot \mathring{L}\Psi|$ if one adopts the same strategy as in \eqref{D13}. Indeed, after examining all terms in $\leftidx{^{(Z)}}D_{\al,2}^{0}$, we find that the worst terms which cannot be controlled as in \eqref{D13} are $(R^{m+1}\operatorname{tr}\check{\mathscr{X}})T\vp_\al$ (arising from $R^m(\leftidx{^{(R)}}D_{\al,2}^{0})$) and $(\bar{Z}^m\slashed{\triangle}\mu) T\vp_\al$ (appearing in $\bar{Z}^m(\leftidx{^{(T)}}D_{\al,2}^{0})$). These two terms cannot be treated by applying \eqref{dnnchi} and \eqref{Zmu}
due to the lack of the degenerate factor $\mu$, which prevents us from closing the final energy estimates.

Inspired by \cite[Section 9.2]{M-Y} and \cite[Section 7.2]{L-Y}, one has
\begin{align}\label{RmX}
&\int_{D^{t,u}}(R^{m+1}\operatorname{tr}\check{\mathscr{X}})T\vp_\al(\mathring LR^{m+1}\vp_\al)\notag\\
= &\int_{D^{t,u}}-(R^{m+1}\operatorname{tr}\check{\mathscr{X}})R^{m+1}\vp_\al\big(\mathring LT\vp_\al+(\operatorname{tr}\check{\mathscr{X}}+\f12\slashed g^{XX}G_{XX}^\beta\mathring L\vp_\beta+\f1\varrho)T\vp_\al\big)\notag\\
&+\int_{D^{t,u}}([R^{m+1},\mathring L]\operatorname{tr}\check{\mathscr{X}})T\vp_\al(R^{m+1}\vp_\al)\\
&+\int_{D^{t, u}}(R^m\mathring L\operatorname{tr}\check{\mathscr{X}})\big(\f12T\vp_\al\cdot R^{m+1}\vp_\al\cdot\operatorname{tr}\leftidx{^{(R)}}{\slashed{\pi}}+RT\vp_\al\cdot R^{m+1}\vp_{\al}+T\vp_\al\cdot R^{m+2}\vp_\al\big)\notag\\
&+\int_{\Sigma_t^u}-(R^m\operatorname{tr}\check{\mathscr{X}})\big(R^{m+1}\vp_\al(\f12T\vp_\al\cdot\operatorname{tr}\leftidx{^{(R)}}{\slashed{\pi}}
+RT\vp_\al)+T\vp_\al(R^{m+2}\vp_\al)\big)\notag\\
&-\int_{\Sigma_{t_0}^u}(R^{m+1}\operatorname{tr}\check{\mathscr{X}})T\vp_\al(R^{m+1}\vp_\al).\notag
\end{align}
The first integral term on the right hand side of \eqref{RmX} can be controlled by
\begin{equation}\label{YHCC9}
\delta^{2\nu-4} \int_{t_0}^t \mu_{\min}^{-1}(\tau) \bigl\| \mu \slashed{d} R^m \, \mathrm{tr}\check{\mathscr{X}} \bigr\|_{\tau,u} \bigl\| R^{m+1} \vp \bigr\|_{\tau,u} \, d\tau.
\end{equation}
Together with \eqref{dnnchi} and \eqref{SiE}, the integral \eqref{YHCC9} is bounded by
\[
\frac{\delta^{3\nu-4}}{b_{m+2}} \mu_{\min}^{-2b_{m+2}}(t) + \frac{\delta^{\nu-1}}{b_{m+2}} \mu_{\min}^{-2b_{m+2}}(t)
\big( \widetilde{\mathcal{E}}_{1,\leq m+2}(t,u) + \widetilde{\mathcal{E}}_{2,\leq m+2}(t,u)\big).
\]
For the other integrals on the right hand side of \eqref{RmX}, we could adapt Proposition \ref{L2chi} directly to obtain
\begin{equation}\label{topX}
\begin{split}
&\big|\int_{D^{t,u}}(R^{m+1}\operatorname{tr}\check{\mathscr{X}})T\vp_\al(\mathring LR^{m+1}\vp_\al)\big|\\
\lesssim&\frac{\delta^{3\nu-4}}{b_{m+2}} \mu_{\min}^{-2b_{m+2}}(t) + \frac{1}{b_{m+2}} \mu_{\min}^{-2b_{m+2}}(t) \big( \widetilde{\mathcal{E}}_{1,\leq m+2}(t,u) +\delta^{\nu-1} \widetilde{\mathcal{E}}_{2,\leq m+2}(t,u) \big).
\end{split}
\end{equation}

In addition, due to
\begin{align*}
&\delta^{2l+2}\int_{D^{t, u}}(\bar Z^m\slashed\triangle\mu)T\vp_\al(\mathring L\bar Z^mT\vp_\al)\\
=&\delta^{2l+2}\int_{D^{t, u}}(\bar Z^m\slashed\triangle\mu)T\vp_\al\big(\mathring L[\bar Z^m, T]\vp_\al+[\mathring L,T]\bar Z^m\vp_\al+T\mathring L\bar Z^m\vp_{\al}\big),
\end{align*}
then it follows from \eqref{Zmu} and Proposition \ref{L2chi} that
\begin{equation}\label{topmu}
\begin{split}
&\big|\delta^{2l+2}\int_{D^{t, u}}(\bar Z^m\slashed\triangle\mu)T\vp_\al(\mathring L\bar Z^mT\vp_\al)\big|\\
\lesssim&\frac{\delta^{3\nu-4}}{b_{m+2}} \mu_{\min}^{-2b_{m+2}}(t) + \frac{\delta^{\nu-1}}{b_{m+2}} \mu_{\min}^{-2b_{m+2}}(t) \big( \widetilde{\mathcal{E}}_{1,\leq m+2}(t,u) + \widetilde{\mathcal{E}}_{2,\leq m+2}(t,u) \big).
\end{split}
\end{equation}
Therefore, with the delicate estimates \eqref{topX} and \eqref{topmu}, we can obtain
\begin{align}\label{DL2}
&\delta^{2l+2k(2-\nu)}\big|\int_{D^{t, u}}\sum_{p=1}^{m}\big(Z_{m+1}+\leftidx{^{(Z_{m+1})}}\lambda\big)\dots\big(Z_{m+2-p}
+\leftidx{^{(Z_{m+2-p})}}\lambda\big)\leftidx{^{(Z_{m+1-p})}}D_{\al,2}^{m-p}(\mathring L\varphi_\al^{m+1})\big|\notag\\
&\lesssim\frac{\delta^{3\nu-4}}{b_{m+2}} \mu_{\min}^{-2b_{m+2}}(t) + \frac{1}{b_{m+2}} \mu_{\min}^{-2b_{m+2}}(t) \big( \widetilde{\mathcal{E}}_{1,\leq m+2}(t,u) +\delta^{\nu-1} \widetilde{\mathcal{E}}_{2,\leq m+2}(t,u) \big)\\
&\quad+\delta^{-1}\int_0^uF_{1,m+2}(s,u')du'.\notag
\end{align}
\end{enumerate}

\subsection{Treatments of $J_2^{m+1}$ and $\Phi_\gamma^1$ in \eqref{Phik}}\label{l}
We now analyze the structure of $J_2^{m+1}$ and $\Phi_\al^1$ in \eqref{Phik}.
It is observed that neither $J_2^{m+1}$ nor $\Phi_\al^1$ contains the highest  order derivatives of $\tr\check{\mathscr X}$ and $\slashed{\triangle}\mu$.
Therefore, in view of Proposition~\ref{L2chi} and the expressions of $\leftidx{^{(Z)}}D_{\al, j}^n$ given in \eqref{T1}--\eqref{R3},
one has
\begin{align}\label{ZL}
&\delta^{\nu-1}\delta^{2l+2k(2-\nu)}\int_{D^{t,u}}|\sum_{p=1}^3\leftidx{^{(Z_{m+1})}}D_{\al,p}^m\cdot \mathring{\underline L}{\varphi}_\al^{m+1}|\no\\
\lesssim&\delta^{2+\nu}(1+\delta^{4\nu-6})^2+\f{\delta^{\nu-1}}{b_{m+2}}\mu_{\min}^{-2b_{m+2}}(t)\big(\widetilde{\mathcal{E}}_{1,\leq m+2}(t,u) + \widetilde{\mathcal{E}}_{2,\leq m+2}(t,u)\big)
\end{align}
and
\begin{align}\label{ZrL}
&\delta^{2l+2k(2-\nu)}\int_{D^{s,u}}|\sum_{p=1}^3\leftidx{^{(Z_{m+1})}}D_{\al,p}^m\cdot \mathring L{\varphi}_\al^{m+1}|\no\\
\lesssim&\delta^{2+\nu}(1+\delta^{4\nu-6})^2 + \frac{\delta^{\nu-1}}{b_{m+2}} \mu_{\min}^{-2b_{m+2}}(t) \big(\widetilde{\mathcal{E}}_{1,\leq m+2}(t,u) +\delta^{2\nu-2} \widetilde{\mathcal{E}}_{2,\leq m+2}(t,u) \big)\\
&\quad+\delta^{-1}\int_0^uF_{1,m+2}(s,u')du'.\notag
\end{align}
By the explicit form of $\Phi_\gamma^0 = \mu \Box_g \varphi_\gamma$ in \eqref{ge}, together with
the $L^2$ estimate for $\tr\checkX$ in Proposition~\ref{L2chi} and the related behaviors from \eqref{lamda},
we are able to deduce that
\begin{align}\label{Phi0}
&\delta^{\nu-1}\delta^{2l+2k(2-\nu)}\int_{D^{t,u}}|(Z_{m+1}+\leftidx{^{(Z_{m+1})}}\lambda)\cdots(Z_{1}
+\leftidx{^{(Z_{1})}}\lambda)\Phi_\al^0\cdot\mathring{\underline L}\varphi_\al^{m+1}|\no\\
\lesssim&\delta^{2+\nu}(1+\delta^{6\nu-10})+\f{\delta^{\nu-1}}{b_{m+2}}\mu_{\min}^{-2b_{m+2}}(t)\big(\widetilde{\mathcal{E}}_{1,\leq m+2}(t,u)+ \widetilde{\mathcal{E}}_{2,\leq m+2}(t,u)\big)
\end{align}
and
\begin{equation}\label{rPhi}
\begin{split}
&\delta^{2l+2k(2-\nu)}\int_{D^{s,u}}|(Z_{m+1}+\leftidx{^{(Z_{m+1})}}\lambda)\cdots(Z_{1}
+\leftidx{^{(Z_{1})}}\lambda)\Phi_\al^0\cdot\mathring L\varphi_\al^{m+1}|\\
\lesssim&\delta^{2+\nu}(1+\delta^{6\nu-10})+\f{\delta^{\nu-1}}{b_{m+2}}\mu_{\min}^{-2b_{m+2}}(t)\big(\widetilde{\mathcal{E}}_{1,\leq m+2}(t,u)+ \delta^{2\nu-2}\widetilde{\mathcal{E}}_{2,\leq m+2}(t,u)\big)\\
&+\delta^{-1}\int_0^u F_{1,m+2}(s,u')du'.
\end{split}
\end{equation}

\subsection{Closure of the bootstrap assumptions}
Note that in Subsections~\ref{tot}--\ref{l}, we have established the estimates for all terms in $\delta^{\nu-1}\int_{D^{t, u}}|\Phi \cdot \mathring{\underline{L}}\Psi|$ and $\int_{D^{t, u}}|\Phi \cdot \mathring{L}\Psi|$ of \eqref{e}.
Substituting the inequalities \eqref{D11}--\eqref{D13}, \eqref{iDg2}, \eqref{DL2} and \eqref{ZL}-\eqref{rPhi} into \eqref{e}
yields
\begin{equation}\label{energy}
\begin{split}
\delta^{\nu-1} \widetilde{\mathcal{E}}_{2,\leq m+2}(t,u) + \delta^{\nu-1} \widetilde{\mathcal{F}}_{2,\leq m+2}(t,u) + \widetilde{\mathcal{E}}_{1,\leq m+2}(t,u) + \widetilde{\mathcal{F}}_{1,\leq m+2}(t,u)
\lesssim \delta^{3\nu-4}.
\end{split}
\end{equation}
In \eqref{energy}, we choose the positive constant $b_{m+2}$ included in the energy expressions on the left hand side of \eqref{energy}
to be sufficiently large and let $b_{m+2} = [b_{m+2}] + \frac{1}{4}$. With this choice, it follows immediately that
both $\mathcal{E}_{1,\leq m+2}(t,u)$ and $\mathcal{E}_{2,\leq m+2}(t,u)$ contain the degenerate factor $\mu$. To close the bootstrap assumptions \eqref{bootstrap-assumptions}, we need to construct suitable non-degenerate energies for the lower order derivatives of the solutions
so that the related Sobolev imbedding theorem can be used. We next estimate $\tilde{\mathcal{E}}_{1,\leq m+1}$ and $\tilde{\mathcal{E}}_{2,\leq m+1}(t,u)$ by use of \eqref{energy} with the fixed number $b_{m+2}$.

For any $Z\in\{\rho\mathring L, T, R\}$, it follows from \eqref{fequation} that
\begin{equation*}
\begin{split}
 \mathring{L}\bigl(\varrho^{1/2}\underline{L}{Z}^m \vp_\al \bigr)
&=\varrho^{1/2}\mathring L[\mathring{\underline L},Z^m]\vp_\al- \f12\varrho^{-1/2}[Z^m,\mathring {\underline L}]\vp_{\al}+
\varrho^{1/2}{Z}^m (\mu\slashed\triangle\vp_\al+H_\al )\\
&+ \varrho^{1/2}[\mathring{L}, {Z}^m] \underline{L} \vp_\al
- \varrho^{1/2}\sum_{\substack{m_1 + m_2 = m \\ m_1 > 0}}{Z}^{m_1}\big(\frac{1}{2\varrho}\big){Z}^{m_2} \underline{L}\vp_\al,
\end{split}
\end{equation*}
together with Proposition \ref{L2chi} and \eqref{energy}, this yields
\begin{align}
&\delta^{l+k(2-\nu)}\|\mathring{L}\bigl(\varrho^{1/2}\underline{L}{Z}^m \vp\bigr)\|_{t,u}\nonumber\\
\lesssim&\delta^{l_0+k_0(2-\nu)}\bigl(\de^{3\nu-5}\|Z^{m_0}\mu\|_{t,u}+(\delta^{-1}+\delta^{2\nu-4})\|Z^{m_0}\vp\|_{t,u}\bigr)
+\delta^{l+k(2-\nu)}\|\mu\slashed\triangle Z^m\vp\|_{t,u}\nonumber\\
&+\delta^{\nu-2}\delta^{l_1+k_1(2-\nu)}\bigl(\|Z^{m_1}\operatorname{tr}\check{\mathscr X}\|_{t,u}+\delta^{\nu-1}\|\slashed{\mathcal{L}}_{{Z}}^{m_1}\leftidx{^{(R)}}{\slashed{\pi}}_T\|_{t,u}
+\delta\|\slashed{\mathcal{L}}_{{Z}}^{m_1}\leftidx{^{(T)}}{\slashed{\pi}}_{\mathring L}\|_{t,u}\nonumber\\
&+\|\slashed{\mathcal{L}}_{{Z}}^{m_1}\leftidx{^{(R)}}{\slashed{\pi}}_{\mathring L}\|_{t,u}+\delta^{2}\|\slashed{\mathcal{L}}_{{Z}}^{m_1}\leftidx{^{(T)}}{\slashed{\pi}}\|_{t,u}
+\delta\|\slashed{\mathcal{L}}_{{Z}}^{m_1}\leftidx{^{(R)}}{\slashed{\pi}}\|_{t,u}\bigr)\nonumber\\
&+\delta^{\nu-1}(1+\delta^{2\nu-3})\delta^{l_1+k_1(2-\nu)}\|\slashed{\mathcal{L}}_{{Z}}^{m_1}\slashed g\|_{t,u}\nonumber\\
\lesssim&\delta^{3\nu-\f92}+(1+\delta^{2\nu-3})\mu_{\min}^{-b_{m+2}}(t)\bigl( \sqrt{\widetilde{\mathcal{E}}_{1,\leq m+2}(t,u)} + \sqrt{\widetilde{\mathcal{E}}_{2,\leq m+2}(t,u)}\bigr)\nonumber\\
\lesssim&\delta^{\nu-\f32}(1+\delta^{2\nu-3})\mu_{\min}^{-b_{m+2}}(t).\label{YHCC10}
\end{align}
By \eqref{YHCC10} and \eqref{Ff}, we arrive at
\begin{equation}\label{E2m+1}
\begin{split}
\delta^{l+k(2-\nu)}\|\varrho^{1/2}\underline{L}{Z}^m \vp\|_{t,u}
\lesssim&\delta^{\nu-\f32}+\delta^{l+k(2-\nu)}\int_{t_0}^t\|\mathring{L}\bigl(\varrho^{1/2}\underline{L}{Z}^m \vp\bigr)\|_{\tau,u}d\tau\\
\lesssim&\delta^{\nu-\f32}\mu_{\min}^{1-b_{m+2}}(t).
\end{split}
\end{equation}
In addition, since $\mu^{-1}\mathring L\mu\lesssim\delta^{\nu-2}$ holds by \eqref{blowup} and \eqref{eq:higher-order-L-infty-estimates},
then
\begin{equation*}
\begin{split}
&\delta^{2l+2k(2-\nu)}\bigl(\mu|RZ^m\vp_\al|^2+\mu|\mathring LZ^m\vp_\al|^2\bigr)(t,u,\vartheta)\\
\lesssim&\delta^{4\nu-6}+\delta^{2l+2k(2-\nu)}\int_{t_0}^t\mu^{-1}\mathring L\mu\bigl(\mu|RZ^m\vp_\al|^2+\mu|\mathring LZ^m\vp_\al|^2\bigr)(\tau,u,\vartheta)d\tau\\
&+\delta^{2l+2k(2-\nu)}\int_{t_0}^t2\mu\bigl(RZ^m\vp_\al\cdot\mathring LRZ^m\vp_\al+\mathring LZ^m\vp_\al\cdot\mathring L^2Z^m\vp_\al\bigr)(\tau,u,\vartheta)d\tau\\
\lesssim&\delta^{4\nu-6}+\delta^{2l+2k(2-\nu)}\int_{t_0}^t\delta^{\nu-2}\bigl(\mu|RZ^m\vp_\al|^2+\mu|\mathring LZ^m\vp_\al|^2\bigr)(\tau,u,\vartheta)d\tau\\
&+\delta^{2l+2k(2-\nu)}\int_{t_0}^t\mu\bigl(|RZ^m\vp_\al|\cdot|\mathring LRZ^m\vp_\al|+|\mathring LZ^m\vp_\al|\cdot|\mathring L^2Z^m\vp_\al|\bigr)(\tau,u,\vartheta)d\tau,
\end{split}
\end{equation*}
and hence,
\begin{equation}\label{YHCC11}
\begin{split}
&\delta^{l+k(2-\nu)}\bigl(\mu^{1/2}|RZ^m\vp_\al|+\mu^{1/2}|\mathring LZ^m\vp_\al|\bigr)(t,u,\vartheta)\\
\lesssim&\delta^{2\nu-3}+\delta^{l+k(2-\nu)}\int_{t_0}^t\bigl(|\mu^{1/2}\mathring LRZ^m\vp_\al|+|\mu^{1/2}\mathring L^2Z^m\vp_\al|\bigr)(\tau,u,\vartheta)d\tau.
\end{split}
\end{equation}
Collecting \eqref{YHCC11} and \eqref{Ff}, one can deduce
\begin{equation}\label{E1m+1}
\begin{split}
&\delta^{l+k(2-\nu)}\bigl(\|\mu^{1/2}RZ^m\vp\|_{t,u}+\|\mu^{1/2}\mathring LZ^m\vp\|_{t,u}\bigr)\\
\lesssim&\delta^{2\nu-\f52}+\int_{t_0}^t\delta^{\nu-2}\mu_{\min}^{-b_{m+2}}(\tau)\sqrt{\tilde{\mathcal E}_{1,\leq m+2}(\tau,u)}d\tau\lesssim\delta^{\f32\nu-2}\mu_{\min}^{1-b_{m+2}}(t).
\end{split}
\end{equation}
Set $b_{m+1} = b_{m+2} - 1$. Then from \eqref{E2m+1} and \eqref{E1m+1}, we have
\[
\delta^{\nu-1} \widetilde{\mathcal{E}}_{2,\leq m+1}(t,u) + \widetilde{\mathcal{E}}_{1,\leq m+1}(t,u) \lesssim \delta^{3\nu-4}.
\]
From this, one can define $b_{m+2-n} = b_{m+2} - n$ for $1 \leq n \leq [b_{m+2}]$ and $b_{m+2-n} = 0$ for $[b_{m+2}] + 1 \leq n \leq m + 2$.
Repeating the argument in deriving \eqref{E2m+1} and \eqref{E1m+1} step by step, we eventually obtain
\[
\delta^{\nu-1} \widetilde{\mathcal{E}}_{2,\leq k}(t,u) + \widetilde{\mathcal{E}}_{1,k}(t,u) \lesssim \delta^{3\nu-4}, \quad 0 \leq k \leq m+2.
\]
In particular, it holds that
\begin{align}
\mathcal E_{2,\leq m + 1 - [b_{m+2}]}(t,u) &\lesssim \delta^{2\nu-3}, \label{E2low}\\
\mathcal E_{1,\leq m + 1 - [b_{m+2}]}(t,u) &\lesssim \delta^{3\nu-4},\label{E1low}
\end{align}
which does not contain the degenerate factor in the energies $\mathcal E_{2,\leq m + 1 - [b_{m+2}]}(t,u)$
and $\mathcal E_{1,\leq m + 1 - [b_{m+2}]}(t,u)$.

We now use the Sobolev inequality on $S_{t,u}$ and \eqref{SE} to close the bootstrap assumptions \eqref{bootstrap-assumptions}.
In fact, for any $n\leq m-[b_{m+2}]$, one has
\begin{equation}\label{YHCC12}
\begin{split}
&\delta^{l+k(2-\nu)}|Z^n\vp_\al(t,u,\vartheta)|
\lesssim\delta^{l+k(2-\nu)}\sum_{i\leq 1}\|R^iZ^n\vp_\al\|_{L^2(S_{t,u})}\\
\lesssim&\delta^{1/2}\bigl(\sqrt{\mathcal E_{1,\leq n+2}(t,u)}+\sqrt{\mathcal E_{2,\leq n+2}(t,u)}\bigr)\lesssim\delta^{\nu-1},
\end{split}
\end{equation}
where he number of $T$ is $l$ and the number of $\varrho\mathring L$ is $k$ in the multi-index $Z^n$.
Obviously, \eqref{YHCC12} implies \eqref{bootstrap-assumptions}.

\section{Sketch of the proof of Theorem \ref{maintheorem} (B)}\label{incoming}
In this section, we provide a sketch of the proof for Theorem \ref{maintheorem} (B) under the assumption \eqref{condition2}.
For this case, shock formation for problem \eqref{mainequation} arises from the intersection of the incoming characteristic conic surfaces.
To this end, we define the $C^1$ optical function ${\hat u}(t,y)$ associated with the equation in \eqref{mainequation} via
\begin{equation}\label{eikonal equation-1}
\begin{cases}
g^{\al\be}\p_{\al}{\hat  u}\p_{\be}{\hat  u}=0,\\
{\hat  u}(t_0,y)=t_0+r,\quad \p_{0}{\hat  u}>0
\end{cases}
\end{equation}
and the corresponding inverse foliation density \(\hat\mu\) as
\[
\hat\mu=-\frac{1}{g^{0\al}\p_{\al}\hat u}.
\]
Meanwhile,  the null frame $\{\hat{\underline L}, \hat L, \hat X\}$ with respect to the metric $(g_{\al\beta})$
is introduced as
\begin{equation}\label{Lr-1}
\begin{split}	
&\hat{\underline L}=-\hat\mu\bigl(g^{\al\beta}\p_\al\hat u\,\p_\beta\bigr),\\
&\hat L={\hat\mu}\hat{\underline L}+2\hat T
\quad \text{with $\hat T=\hat\mu(-g^{0\al}\p_\al-\hat{\underline L})$},\\
&\hat X=\f{\p}{\p{\hat\vartheta}}\quad\text{with $\hat\vartheta$ being determined by ${\hat{\underline L}}{\hat\vartheta}=0$
and $\hat\vartheta|_{t=t_0}=\th\in\mathbb{S}^1$}.\\	
\end{split}
\end{equation}
We perform the following coordinate transformation in the neighborhood of the incoming light conic surface
\(A_0=\{(t,y): r+t = \frac{1-\delta}{\mathcal A}\}\):
$$(t, y^1, y^2)\longrightarrow (t, \hat u, \hat\vartheta)\quad\text{with $\hat u=\hat u(t,y)$ and $\hat\vartheta=\hat\vartheta(t,y)$}.
$$

As in Lemma \ref{Lem3.3}, under the coordinate $(t, \hat u, \hat\vartheta)$, one has
\begin{equation}\label{lumu}
\hat{\underline L}\hat\mu = -\frac{1}{2}\hat\mu G_{\hat{\underline L}\hat{\underline L}}^\al\hat{\underline L}\vp_\al
- \hat\mu G_{\tilde{\underline T}\hat{\underline L}}^\al \hat{\underline L}\vp_\al
+ \frac{1}{2}G_{\hat{\underline L}\hat{\underline L}}^\al \hat T\vp_\al.
\end{equation}
In addition, under the analogous bootstrapping assumption \eqref{bootstrap-assumptions} near $A_0$
and by the arguments similar to those in Proposition \ref{prop:bootstrap-higher-derivatives},
we arrive at
\begin{align*}
&|\vp_\al|\lesssim\delta^{\nu-1},\quad
|\hat{\underline L}\vp_\al|\lesssim\delta^{2\nu-3},\\
&\bigl|\hat\varrho^{1/2}\hat T\vp_\al(t,\hat u,\hat\vartheta)-\hat\varrho_0^{1/2}\hat T\vp_\al(t_0,\hat u,\hat\vartheta)\bigr|\lesssim\delta^{2\nu-3},
\end{align*}
where \(\hat\varrho=t+\hat u\) and \(\hat\varrho_0=t_0+\hat u\).

Introduce the auxiliary quantity
\[
\hat{\mathscr{M}}(t,\hat u,\hat\vartheta):= \frac{(\gamma+1) q_{0}^{3}} {2c^{2}(\rho_{0})\bigl(q_{0}^{2}-c^{2}(\rho_{0}) \bigr)} \,\hat\varrho^{1/2}\, \hat T\varphi_{0}(t,\hat u,\hat\vartheta).
\]
Combining this with \eqref{lumu}, one deduces
\[
\hat\varrho^{1/2}\hat{\underline L}\hat\mu=\hat{\mathscr M}(t_0,\hat u,\hat\vartheta)+O(\delta^{2\nu-3}).
\]
By integrating along the integral curves of $\hat{\underline L}$, it holds that
\begin{equation}\label{YHCCC-4}
\begin{split}
\hat\mu(t, \hat u,\hat\vartheta)&=1+2\hat{\mathscr M}(t_0,\hat u,\hat\vartheta)\bigl(\hat\varrho_0^{1/2}-\hat\varrho^{1/2}\bigr)+O(\delta^{\nu-1})\\
&=1+\frac{(\gamma+1) q_0^3}{2c^2(\rho_0)\bigl(q_0^2 - c^2(\rho_0)\bigr)} \hat T\vp_0(t_0,\hat u,\hat\vartheta)\bigl(t+O(t^2)\bigr) + O(\delta^{\nu-1}) \\
&= 1 + \frac{(\gamma+1) q_0^3}{2c^2(\rho_0)\bigl(q_0^2 - c^2(\rho_0)\bigr)} \p_r\vp_0(t_0,\hat u,\hat\vartheta)\bigl(t+O(t^2)\bigr) + O(\delta^{\nu-1})\\
&=1+\frac{(\gamma+1) q_0^3}{4c^2(\rho_0)\bigl(q_0^3-c^2(\rho_0)\bigr)} L\vp_0(t_0,\hat u,\hat\vartheta)\bigl(t+O(t^2)\bigr)+O(\delta^{\nu-1}),
\end{split}
\end{equation}
where we have used the key identity \(\hat T\vp_0(t_0,\hat u,\hat\vartheta)=\p_r\vp_0(t_0,\hat u,\hat\vartheta)
+O(\delta^{2\nu-3})\) and the estimate \eqref{initial2}.

It follows from \eqref{YHCCC-4} and \eqref{y1} that
    \begin{equation}\label{YHCCC-5}
    \begin{split}
    &\hat\mu(t,\hat u_0,\hat\vartheta_0)\\
    = & 1 + \frac{(\gamma+1) q_0^3\mathcal A}{4c^2(\rho_0)\bigl(q_0^2 - c^2(\rho_0)\bigr)} \delta^{\nu-2} \bigl(\mathcal A\p_s^2\Phi_0(s_0,\omega_0)+\p_s\Phi_1(s_0, \omega_0)\bigr)
    \big(t+O(t^2)\big) + O(\delta^{\nu-1}),
    \end{split}
    \end{equation}
where \((\hat u_0, \hat\vartheta_0)\) is determined by
$(\frac{\mathcal A \hat u_0 - 1}{\delta}, \cos\hat\vartheta_0, \sin\hat\vartheta_0) = (\tilde s_0, \tilde\omega_0)$
in view of assumption \eqref{condition2}. Consequently, \(\hat\mu\to 0+\) holds before
$\tilde t_*$, which completes the proof of Theorem \ref{maintheorem} (B).
Finally, the closure of the analogous bootstrap assumptions \eqref{bootstrap-assumptions} near $A_0$
can be completed by the analogous arguments in Section \ref{Section 3} and Sections \ref{EE}--\ref{ert}.

\vskip 0.4 true cm

{\bf \color{blue}{Conflict of Interest Statement:}}

\vskip 0.2 true cm

{\bf The authors declare that there is no conflict of interest in relation to this article.}

\vskip 0.3 true cm
{\bf \color{blue}{Data availability statement:}}

\vskip 0.2 true cm

{\bf  Data sharing is not applicable to this article as no data sets are generated
during the current study.}

\end{document}